\documentclass[12pt]{article}
\usepackage[utf8]{inputenc}
\usepackage{amsmath}
\usepackage{amsthm}
\usepackage{amsxtra}
\usepackage{amssymb}
\usepackage{dsfont}
\usepackage{graphicx}
\graphicspath{ {./figures/} }
\usepackage[format=plain, labelfont=bf]{caption}
\usepackage{xcolor}
\usepackage[normalem]{ulem}
\usepackage[USenglish]{babel}
\usepackage{comment}
\usepackage{ulem}
\usepackage[left=3cm, right=2.5cm, top=2.5cm, bottom=2.5cm]{geometry}
\usepackage{subfigure}
\numberwithin{equation}{section}
\usepackage{blindtext}
\usepackage{hyperref}
\usepackage{cancel}

\newcommand{\de}{D}

\newcommand{\IN}{\mathbb{N}}
\newcommand{\IP}{\mathbb{P}}
\newcommand{\IE}{\mathbb{E}}

\newcommand{\IR}{\mathbb{R}}

\newcommand{\cE}{\mathcal{E}}
\newcommand{\cF}{\mathcal{F}}
\newcommand{\cP}{\mathcal{P}}
\newcommand{\cG}{\mathcal{G}}
\newcommand{\cN}{\mathcal{N}}

\newcommand{\cS}{\mathcal{S}}
\newcommand{\cT}{\mathcal{T}}
\newcommand{\cV}{\mathcal{V}}
\newcommand{\cW}{\mathcal{W}}
\newcommand{\bfC}{\mathbf{C}}
\newcommand{\bfB}{\mathbf{B}}

\newcommand{\bfX}{\mathbf{X}}

\newcommand{\la}{\lambda}

\newcommand{\dA}{A}
\newcommand{\dG}{G}
\newcommand{\dV}{V}
\newcommand{\dE}{E}
\newcommand{\gen}{\mathcal{A}}

\newcommand{\1}{\mathds{1}}

\theoremstyle{definition}
\newtheorem{definition}{Definition}[section]

\newtheorem{remark}[definition]{Remark}
\newtheorem{problem}{Problem}
\newtheorem{conjecture}[problem]{Conjecture}
\theoremstyle{plain}
\newtheorem{theorem}[definition]{Theorem}
\newtheorem{lemma}[definition]{Lemma}
\newtheorem{corollary}[definition]{Corollary}

\newtheorem{proposition}[definition]{Proposition}

\title{{\Large \bf The contact process on dynamical random trees  with degree dependence}}
\author{\textsc{Natalia Cardona-Tob\'on}\footnote{Departamento de Estadística, Universidad Nacional de Colombia, Carrera 45 No 26-85, 111321 Bogot\'a D.C., Colombia, \texttt{ncardonat@unal.edu.co}}, \textsc{Marcel Ortgiese}\footnote{Department of Mathematical Sciences, University of Bath, Claverton Down, Bath, BA2 7AY,
United Kingdom, \texttt{m.ortgiese@bath.ac.uk}}, \textsc{Marco Seiler}\footnote{Goethe University Frankfurt, Robert-Mayer-Straße 10, 60486 Frankfurt am Main, Germany, \texttt{seiler@math.uni-frankfurt.de}}\\ \textsc{and Anja Sturm}\footnote{Institute for Mathematical Stochastics. Georg-August-Universität Göttingen, Goldschmidtstraße 7, 37077 Göttingen, Germany, \texttt{anja.sturm@mathematik.uni-goettingen.de}}}

\begin{document}

\date{\today}
\maketitle

\begin{abstract}
The contact process is a simple model for the spread of an infection in a structured population. We investigate
the case when the underlying structure evolves dynamically as a degree-dependent dynamical percolation model. Starting with a connected locally finite base graph we initially declare edges independently open with a connection probability that is allowed to depend on the degree of the adjacent vertices and closed otherwise. Edges are independently updated with a rate depending on the degrees and then are again declared open and closed with the same probabilities. We are interested in the contact process, where infections are only allowed to spread via open edges. Our aim is to analyse the impact of the update speed and the connection probability on the existence of a phase transition. For a general connected locally finite graph, our first result gives sufficient conditions for the critical value for survival to be strictly positive. Furthermore, in the setting of Bienaymé-Galton-Watson trees, we show that the process survives strongly with positive probability for any infection rate if the offspring distribution has a stretched exponential tail with an exponent depending on the connection probability and the update speed. In particular, if the offspring distribution follows a power law and the connection probability is given by a product kernel  and the update speed exhibits polynomial behaviour, we provide a complete characterisation of the phase transition.\\[0.5em]
{\footnotesize\textit{2020 Mathematics Subject Classification --} Primary 60K35, Secondary 05C80, 82C22
	}
\end{abstract}

	\section{Introduction}
Mathematical models describing the spread of particles in a structured population or network are highly relevant in many applications. For example, they are used to understand the course of an epidemic, the spread of a computer virus in a network or the dissemination of misinformation in social media. This motivates us to study the contact process, first introduced by Harris \cite{harris1974contact}, which can be considered as a prototype for models describing the spatial spread of an infection. Here the spatial structure of the population is typically encoded by a graph. 
The vertices represent individuals and the edges indicate which individuals are considered to be neighbours. An individual can either be infected or healthy, and therefore susceptible to the infection. The process evolves according to the following dynamics: an infected individual infects a neighbour with a certain infection rate $\lambda$ and can itself recover independently with constant rate one.

The long-term behaviour of the contact process is very well understood on fixed graphs such as the integer lattice and regular trees (see e.g.\ \cite{lig99}). 
For finite graphs and also for infinite graphs with small infection rates $\lambda$ (possibly only $\lambda =0$) the infection will almost surely die out.	On infinite graphs, for larger values of $\lambda$ the infection may persist forever with positive probability. Within this regime there may still be a distinction between the possibility of \emph{weak survival} referring to long-term persistence somewhere in the graph and \emph{strong survival} which means that the infection returns infinitely often to its starting point.

Since fixed graphs are not particularly well suited for describing a complex interaction structure the contact process has  more recently  been considered on (finite and infinite) \emph{random graphs}, in particular on random graphs with heavy-tailed degree distributions, where so-called hubs or small-world phenomena occur. The realisation that these hubs are important dates back at least to Berger et al. \cite{Berger2005}. Later on, Chatterjee and Durrett \cite{chatterjee2009contact} analyse the occurrence of these hubs to showthat the contact process does not exhibit a subcritical phase on finite graphs with a power-law degree distribution. In the setting of finite graphs this means that for all positive infection rates starting from all vertices infected, with high probability as the number of vertices tends to infinity, the process maintains a positive density of infected vertices for an exponentially long time (in the size of the graph).
Of particular interest for our work are the results of Bhamidi et al.\ \cite{bhamidi2021survival} and Huang and Durrett \cite{huang2020contact} who establish that  a contact process on a (supercritical) Bienaymé-Galton-Watson (BGW) tree (on the event of the tree being of infinite size) exhibits a subcritical phase if and only if the offspring distribution has exponential moments. Recently, these results were extended by \cite{cardona2021contact} to include an inhomogeneous model with random transmission rates.

Another effect, which is often neglected, is that in reality the graph structure is not fixed, but evolves as time progresses. It turns out that the contact process exhibits interesting behaviour on \emph{dynamically evolving graphs}. For example,
Linker and Remenik \cite{linker2020contact} considered the contact process on a dynamical percolation graph and found that a so-called immunisation phase occurs, where the infection dies out almost surely no matter how large the infection rate is. This happens despite the fact that every vertex in the graph can be reached with positive probability. In \cite{seiler2022contact} more general update mechanisms were considered for underlying 
graphs with bounded degree. But also on dynamical graphs of unbounded degree such immunisation effects occur as shown in \cite{seiler2022long}, where a dynamical long range percolation model was considered. For our work the results of Jacob, Linker and Mörters \cite{jacob2017contact, jacob2019metastability, jacob2022contact} are of particular interest. They consider dynamical finite graphs with a power-law degree distribution, which are closely connected to our setting. See Section \ref{sec:discussion} for a discussion of these and other results.

In this article we combine these two research directions and consider questions of extinction and weak and strong survival for the contact process on infinite dynamical graphs, with a focus on underlying graphs that are Bienaymé-Galton-Watson (BGW) trees. We are particularly interested in heavy-tailed offspring/degree distributions since the contact process without dynamically changing edges is in this setting already fairly well understood. 

The base graph of our dynamical model is an infinite, connected and locally finite graph. In this graph we initially declare edges between two vertices $x$ and $y$ to be open with probability $p(d_x, d_y)$ or closed otherwise, where $d_x$ and $d_y$ respectively denote the degree of these two vertices. Then, edges are independently updated with rate $v(d_x, d_y)$ and are    
subsequently again declared open or closed with the same probabilities.
On top of this dynamical structure, we define a standard contact process for which infections are 
only passed along open edges.

We interpret the underlying graph as describing the connectivity of a social network. Since individuals  cannot interact with all their neighbours simultaneously, the open edges model currently active connections. With this interpretation in mind,
we are in particular interested in a kind of \emph{degree penalisation}, this means that we are choosing the probability $p$ 
in such a way that vertices with high degree have a lower probability $p$ of being connected to one particular neighbour. 

In this general setting, our first main result (see Theorem \ref{Thm:SubCriticalPhase}) states that the model exhibits a subcritical phase if the penalisation is strong enough and the update mechanism fast enough. In other words, we derive sufficient conditions for the connection probabilities $p$ and the update speed $v$ under which the contact process goes extinct almost surely for a small enough infection rate. Furthermore, in this context, we show that the process does not explode in finite time almost surely for every infection rate. As an example, we show that these conditions are satisfied in certain parameter regimes for connection probabilities 
given by a product or maximum kernel. 

The second main result (see Theorem \ref{thm:comparisonpenal}) involves a comparison with the penalised version of the contact process introduced recently in \cite{zsolt}.
In particular, in the case when the underlying graph is a supercritical BGW tree, this comparison result together with results obtained in \cite{zsolt}, enables us to deduce the absence of a subcritical phase in certain parameter regions. The third main result (see Theorem \ref{thm:strongsurvival}) focuses on dynamical supercritical BGW trees and  establishes sufficient conditions on the offspring distribution as well as on the $p$ and~$v$ such that  the contact process is supercritical on infinite trees for all positive infection rates.
In particular, on dynamical supercritical BGW trees with a stretched exponential offspring distribution with a certain exponent (depending on $p$ and~$v$), the contact process is always supercritical. Additionally, we show that immunisation does not occur in this setting in a large parameter region, meaning that survival is possible on infinite trees for infection rates that are large enough.

\section{Model dynamics}
Let us denote by $\dG=(\dV,\dE)$ a connected locally finite graph rooted at a vertex $\rho$ with $\dV$ and $\dE$ the set of vertices and edges, respectively. Denote by $d_x$ the degree of the vertex $x\in \dV$ and by $\{x,y\}$ an edge between $x,y \in V.$

The \textit{contact process $(\bfC, \bfB)= (\bfC_t, \bfB_t)_{t\ge 0}$ on a dynamical graph} (CPDG) is a continuous-time Markov process on  $\cP(\dV)\times \cP(\dE)$, where $\cP(\dV)$ and $\cP(\dE)$ denote the power sets of $\dV$ and $\dE$, respectively. We equip $\cP(\dV)\times \cP(\dE)$ with the topology which induces pointwise convergence, i.e.\ $(C_n,B_n)\to (C,B)$ if and only if $\1_{\{x\in C_n,e\in B_n\}} \to \1_{\{x\in C,e\in B\}}$ for all $(x,e)\in V\times E$.

The \textit{background process} $\bfB=(\bfB_t)_{t\geq 0}$ is a Markov process which describes the evolving edge set 
of the underlying random graph and the \textit{infection process} $\bfC=(\bfC_t)_{t\geq 0}$ is a Markov process  which models the spread of the infection on this graph. In order to define the transition rates of $\bfB$, we introduce two functions
\begin{equation*}
	p:\IN\times \mathbb{N}\to[0,1]\qquad \text{and}\qquad v: \IN\times \mathbb{N}\to(0,\infty).
\end{equation*}
The function $p$ is a symmetric function, decreasing in both arguments, that describes the \textit{connection probabilities}. The function $v$ is monotone and symmetric and models the \textit{update speed} of an edge. 

The background process evolves as follows: Each edge $\{x,y\}\in \dE$ \textit{updates} independently after an exponentially distributed waiting time with parameter $ v(d_x,d_y)$. Then, we declare the edge $\{x,y\}$  \textit{open} with probability $p(d_x,d_y)$ and \textit{closed} with probability $1-p(d_x,d_y)$. We let $\bfB_t$ denote the set of open edges at time $t \geq 0$. 
In other words, if $\bfB$ is currently in state $B$ it has transitions
\begin{equation}\label{def:BackgroundProcess}
	\begin{aligned}
		B\mapsto B\cup \{x,y\} \quad &\text{ with rate } \quad  v(d_x,d_y)p(d_x,d_y),\\
		B\mapsto B\backslash \{x,y\} \quad &\text{ with rate } \quad  v(d_x,d_y)(1-p(d_x,d_y)),
	\end{aligned}
\end{equation}
for every $\{x,y\}\in \dE$. Throughout this paper, we always start the background process in its stationary distribution which means that for any fixed time $t \geq 0$ the probability that an edge is open $\IP(\{x,y\}\in \bfB_t)=p(d_x,d_y)$ for all $\{x,y\} \in \dE$
independently of other edges.

Now we define the \textit{infection process} $\bfC=(\bfC_t)_{t\geq 0}$ in this evolving random environment. If $\bfB$  is currently in state $B$,
the transitions of the infection process $\bfC$ currently in state $C$ are for all $x\in \dV$
\begin{equation}\label{def:InfectionProcess}
	\begin{aligned}
		C\mapsto C\cup \{x\} \quad  &\text{ with rate } \quad  \lambda|\{y\in C:\{x,y\}\in B\}|,\\
		C\mapsto C\backslash \{x\} \quad &\text{ with rate } \quad 1,
	\end{aligned}
\end{equation}
where $\lambda>0$ denotes the \textit{infection rate} and $|\cdot|$ denotes the cardinality of the set. On general graphs it is not a priori clear that the infection process is well defined. See Subsection~\ref{sec:GraphRep} for a precise definition and a discussion of this issue. 

We are in particular interested in choices of $p$ and $v$, which satisfy the following assumptions. Let $\alpha\ge 0$ and $\eta \in \mathbb{R}$. For every $m\in \IN$ there exist constants $\kappa_1,\kappa_2, \nu_1,\nu_2>0$ such that
\begin{align}
	\kappa_1 n^{-\alpha} \le p(n, m)&\le \kappa_2 n^{-\alpha}\label{eq:kernelsgen}\\
	\nu_1 n^{\eta} \le v(n, m) &\le \nu_2 n^{\eta}, \label{eq:kernelsgen2}
\end{align}
for all $n\geq m$. (Note that because of monotonicity $\kappa_2$ can in fact be chosen independently of $m$.) Broadly speaking the condition on the connection probability $p$ can be interpreted in the following way. Let $x$ be a vertex with $d_x=N$.Then the probability to share an open edge with a neighbour at some time $t$ is of lower or equal order than $N^{-\alpha}$.
Since there are in total $N$ neighbours it follows 
\begin{equation}\label{eq:averneig}
	|\{y\in \cV:\{x,y\}\in \bfB_t\}| = O(N^{1-\alpha})
\end{equation}
as $N$ grows large. Note that for BGW trees, which we consider later, this quantity is with high probability exactly of order $N^{1-\alpha}$ as $N\to\infty$. We note that condition \eqref{eq:kernelsgen} on the connection probability $p$ is similar to the conditions used recently in~\cite[see equations (1) and (2)]{jacob2022contact} for the case of finite graphs as well as by \cite{zsolt} on infinite graphs.

One interpretation of the parameter $\alpha$ in \eqref{eq:kernelsgen} is that it determines the influence on the connection probability between two neighbours of the vertex with the larger degree. We are in particular interested in functions $p$ of the form 
\begin{equation}\label{eq:kernels}
	p_{\alpha,\sigma}(d_x,d_y) := 1\wedge \kappa  \big((d_x\wedge d_y)^{\sigma}(d_x\vee d_y)\big)^{-\alpha},\qquad \sigma \in [0,1],\quad  \kappa>0,\quad  \alpha \ge 0.
\end{equation}
These functions are natural examples which satisfy \eqref{eq:kernelsgen}. Here the parameter~$\sigma$ indicates how much influence the smaller degree of the two neighbours has on the connection probability. For example $\sigma=0$ corresponds to the maximum kernel, i.e.\ $p_{\alpha,0}(d_x,d_y) = 1\wedge \kappa (d_x\vee d_y)^{-\alpha}$, in this case only the larger of both vertices has a contribution. If we choose $\sigma=1$ we get the product kernel, i.e.\ $p_{\alpha,1}(d_x,d_y)= 1\wedge \kappa(d_xd_y)^{-\alpha}$. Here, both vertices contribute fully relative to their degree. Another reason to consider these types of functions is that especially the maximum and product kernel have already been considered in a number of related articles, for example \cite{komjathy2021penalising,jacob2022contact, zsolt}. This enables us to compare our result with other established results. See Section~\ref{sec:discussion} for more details.

Now, we define the critical values for the infection parameter $\lambda$. For a fixed vertex $\rho$, denote by $\IP^{\{\rho\}}$ the law of $(\bfC, \bfB)$ started with only $\rho$ infected. Given the graph $\dG$, we define the critical value between \textit{extinction} and \textit{weak survival} %(or \textit{global survival})
by
\begin{equation*}
	\lambda_1(\dG):=\inf \big\{\lambda\geq 0:\  \IP^{\{\rho\}}\left(\bfC_t\neq\emptyset,\,     \text{for all}\ t\geq0 \right)>0\big\},
\end{equation*}
and the critical value between weak and \textit{strong survival} %(or \textit{local survival}) 
by 
\begin{equation*}
	\lambda_2( \dG):=\inf \big\{\lambda\geq 0:\   \IP^{\{\rho\}}\big( \mbox{for any } s \geq 0 \mbox{ there exists } t \geq s \, : \, \rho \in \bfC_{t} \big)>0\big\}.
\end{equation*}
Note that $\lambda_1( \dG)\leq \lambda_2( \dG)$ and that the definition of both of these critical values is independent of the specific choice of $\rho\in \dG$. Thus, by the additivity of the contact process $\lambda_1( \dG)$ also has the same value if we start the contact process in any finite and nonempty set $C$ of infected sites. The main contribution of this paper is to understand the phase transitions of the contact process in our setting. More precisely, our main results give sufficient conditions on the parameter of the model such that we have $\lambda_1(\dG)>0$ (and also for $\lambda_2(\dG)<\infty$)  or $	\lambda_2( \dG)=\lambda_1(\dG)=0$.

\section{Main results}
Our first main theorem is a general result about sufficient conditions on the connection probabilities $p$ and the update speed $v$ for extinction of the contact process. 	An important assumption for our first theorem is that the update speed is bounded away from zero, i.e.
\begin{equation}\label{eq:vbound0}
	v(d_x,d_y)\geq \underline{v}:= \min\{v(d_x, d_y):x,y \in \dV\}%\underline{v}
	>0.
\end{equation}
\begin{theorem}\label{Thm:SubCriticalPhase}
	Consider the CPDG $(\bfC,\bfB)$ on the  graph $\dG$ 
	%started with any finite set $C\subset \dV$ of infected vertices. 
	and assume that condition \eqref{eq:vbound0} holds. We also assume that there exists a function $W:\IN\to [1,\infty)$ and a constant $0<K< \infty$ such that for all $x\in \dV$,
	\begin{equation}\label{Eq:SupermartigaleCondition0}
		\sum_{ \{x,y\}\in E}W(d_y) p(d_y,d_x)\leq K W(d_x),
	\end{equation}
	as well as
	\begin{equation}\label{Eq:SupermartigaleCondition1}
		\sum_{\{x,y\}\in E}\frac{p(d_x,d_y)}{v(d_x, d_y)^{2}} \leq  K.%\frac{K}{\underline{v}^2}.
	\end{equation}
	Then we have $\lambda_1(G)>0$. Furthermore, it follows that the process started with any finite set $C\subset \dV$ of infected vertices does not explode, i.e.\
	\begin{equation*}
		\IP(|\bfC_t|<\infty\ \,\forall\, t\geq 0 )=1
	\end{equation*}
	for any infection rate $\lambda>0$.
\end{theorem}

In the following corollary we discuss the implications of the previous theorem for some specific choices of parameters. 

\begin{corollary}\label{cor:Extinction}
	Consider the CPDG $(\bfC,\bfB)$ on the  graph $\dG$ and assume that $v$ satisfies \eqref{eq:kernelsgen2} with  $\eta\geq 0$ and that one of the following two conditions holds:
	\begin{itemize}
		\item[(i)] $p$ satisfies \eqref{eq:kernelsgen} with $\alpha \geq 1$.
		\item[(ii)] $p = p_{\alpha,\sigma}$ for $p_{\alpha,\sigma}$ as in
		\eqref{eq:kernels} with $\alpha \sigma\geq 1/2$ and $\alpha+2\eta\ge 1$.
	\end{itemize}
	Then,  $\lambda_1(G)>0$ and  the process started with any finite set $C\subset \dV$ of infected vertices does not explode for any infection rate $\lambda > 0$.
\end{corollary}

The second main result is a comparison with a \textit{penalised contact process} which has recently been investigated by Bartha, Komj\'athy and Valesin \cite{zsolt}. This process evolves according to the following rules. For all $x\in \dV$
\begin{equation}\label{PenTransition}
	\begin{aligned}
		C\mapsto C\cup \{x\} \quad &\text{ with rate }\quad    \lambda\sum_{y\in C, \{x,y\} \in E}  p(d_x,d_y)\\
		C\mapsto C\backslash \{x\} \quad  &\text{ with rate } 
		\quad 1,
	\end{aligned}
\end{equation}
where $\lambda>0$ is the infection rate. Intuitively, this penalised contact process arises as the limit model as the update speed tends to infinity, where each potential infection event is immediately preceded by an update event leading to a thinning of the infections with probability $p(d_x,d_y)$.
In the second part of the result, we will consider a sequence of CPDG for which the minimum of the range of $v$ tends to infinity. In order to state the result we thus emphasize the dependence on $v$ of the critical values for extinction-survival and weak-strong survival by writing 
$\lambda_1(v,\dG)$ and
$\lambda_2(v,\dG)$.
Also, let $\lambda_1^p(\dG)$ and $\lambda_2^p(\dG)$ denote the critical values for extinction-survival and weak-strong survival, respectively, of  the penalised contact process.    
\begin{proposition}\label{thm:comparisonpenal}
	Consider the CPDG $(\bfC,\bfB)$ on the  graph $G$. If $\lambda_1^p(\dG)=0$  (resp.\ $\lambda_2^p(\dG)=0$) then we have $\lambda_1(v,\dG)=0$ (resp.\ $\lambda_2(v,\dG)=0$) for all $v$ that satisfy condition \eqref{eq:vbound0}. 
	Furthermore, we have for a sequence of  CPDG $(\bfC^n,\bfB^n)_{n \in \IN}$ with update speeds $(v^n)_{n \in \IN}$ satisfying  \eqref{eq:vbound0} with $\lim_{n \rightarrow \infty} \underline{v}^n = \infty$ (and fixed connection probabilities $p$) that
	\begin{equation*}
		\limsup_{n \rightarrow \infty} \lambda_i(v^n, G)\leq \lambda_i^p(\dG),\qquad \text{for}\qquad i\in \{1,2\}.
	\end{equation*} 
\end{proposition}

\bigskip
Our next main result focuses on the case where the underlying graph is given by a supercritical BGW tree. Let us denote by $\cT\sim \text{BGW}(\zeta)$ the BGW tree rooted at $\rho$ with offspring distribution $\zeta$. Furthermore, we denote the corresponding vertex set by $\cV$ and the edge set by $\cE$, i.e.\ $\cT=(\cV,\cE)$. Note that in this setting $\IP$ denotes the probability measure for which the underlying random tree $\cT$ and the CDPG $(\bfC,\bfB)$ are jointly defined. Sometimes we consider a fixed realization of $\cT$, in these cases we denote the conditional probability by $\mathbb{P}_{\mathcal{T}}(\,\cdot\,):=\mathbb{P}(\ \cdot \ | \ \mathcal{T})$.
The degree of a vertex $x\in \cV$ is a random variable given by
\begin{equation}\label{eq:Dzeta}
	\begin{aligned}
		\de_x=\begin{cases}
			\zeta_{\rho} & \text{ if } x=\rho\\
			\zeta_{x}+1 & \text{ otherwise }
		\end{cases}
	\end{aligned}
\end{equation}
where $ \zeta_{\rho}$ and  $\zeta_{x}$ are i.i.d.\ copies of $\zeta$. We assume that the offspring distribution satisfies
\begin{equation}\label{eq:offconditions}
	\mu:=\mathbb{E}\big[\zeta\big]>1,
\end{equation}
where we do not exclude the case $\mu=\infty.$ In particular, this condition guarantees that $\mathcal{T}$ is infinite with positive probability. Furthermore, under the condition $\mu >1$ the critical values $\lambda_1(\mathcal{T})$ and $\lambda_2(\mathcal{T})$ are constants on the event $\{|\mathcal{T}|=\infty\}$. Note that these two constants do depend on the distribution of $\zeta$. In particular on the event $\{|\cT|=\infty\}$ it holds that
\begin{align*}
	\lambda_1(\cT)&=\lambda_1:= \inf\big\{\lambda>0: \IP^{\{\rho\}}(\bfC_t\neq \emptyset \,\forall\, t\geq 0)>0 \big\} \\
	\lambda_2(\cT)&=\lambda_2:= \inf\big\{\lambda>0:   \IP^{\{\rho\}}(\forall s \geq 0 \,  \exists t \geq s \, : \, \rho\in \bfC_t)>0\big\}
\end{align*}
(see Lemma \ref{lemma:constantscritical} below). This means that the \textit{quenched} critical infection rates $\lambda_1(\cT)$ and $\lambda_2(\cT)$ on the event $\{|\cT|=\infty\}$ are equal to the \textit{annealed} critical infection rates $\lambda_1$ and $\lambda_2$.

The first part of the following theorem gives a sufficient criterion for the process to always survive strongly, so that there is no nontrivial phase transition. The second part gives us a sufficient condition under which the critical values are finite.

\begin{theorem}\label{thm:strongsurvival}
	Consider the CPDG $(\bfC,\bfB)$ on the BGW tree $\cT$ satisfying \eqref{eq:offconditions}. Assume that the connection probability $p$ and the update speed $v$ satisfy \eqref{eq:kernelsgen} and \eqref{eq:kernelsgen2}.
	\begin{itemize}
		\item[(i)]  If we assume that
		\begin{equation}\label{eq:assumtionkernel}
			\limsup_{N\to \infty} \frac{\log \mathbb{P}(\zeta = N)}{N^{1-\alpha- 2(\eta\vee 0)}}=0, \ 
		\end{equation}
		then $\lambda_1=\lambda_2=0$, i.e.\ the process survives strongly for any $\lambda>0$. \item[(ii)] If we assume     \begin{equation}\label{eq:assumtionkernel2}
			\limsup_{N\to \infty} \frac{\log \mathbb{P}(\zeta = N)}{N^{1-\alpha}}=0, \ 
		\end{equation}
		then $\lambda_1\leq\lambda_2<\infty$.
	\end{itemize}
\end{theorem}
Theorem~\ref{thm:strongsurvival}$(i)$ yields a sufficient condition for the lack of a phase transition when the system updates slowly, i.e.~$\eta \leq 0$, or at moderate speed, i.e.~$\eta\in(0,\tfrac{1}{2})$. In case the update speed is slow, we only need $\alpha<1$, but if it is moderate, a stricter assumption on $p$ is required, which is that $\alpha<(1-2\eta)$.

In \cite{zsolt} survival of the penalised contact process is studied in more detail on BGW trees with  offspring distribution given by either a power law or a stretched exponential distribution. Note that they primarily consider the product and maximum kernel. By using their results together with Proposition~\ref{thm:comparisonpenal} we obtain the following result:
\begin{proposition}\label{cor:SurvivalByComparison}
	Consider the CPDG $(\bfC, \bfB)$ on the BGW tree satisfying $\IP(\zeta=0)=0$. Assume that $v$ satisfies \eqref{eq:kernelsgen2} and the connection probability $p$ is chosen to be $p_{\alpha,\sigma}$ as defined in \eqref{eq:kernels} with $\sigma\in[0,1]$,  $\kappa>0$ and $\alpha, \eta\geq 0$. Then if
	\begin{equation}\label{eq:assumptionkernel3}
		\limsup_{N\to \infty} \frac{\log \mathbb{P}(\zeta = N)}{N^{1-2\alpha}}=0, \ 
	\end{equation}
	it follows that $\lambda_1=\lambda_2=0$. 
\end{proposition}
In this proposition the stricter assumption $\IP(\zeta=0)=0$ is necessary since this is an assumption posed for the results
of~\cite{zsolt} for
the penalised contact process. We do believe that again \eqref{eq:offconditions} should be sufficient.

If the update speed of the system is moderate or fast, where by fast we mean that $\eta\geq \tfrac{1}{2}$, and the connection probabilities $p$ is of the form as in \eqref{eq:kernels}, we will see that Proposition~\ref{cor:SurvivalByComparison} gives us a sufficient condition when the subcritical phase is absent in our model by a comparison with a penalised contact process as introduced in \eqref{PenTransition}. Here, the absence of a phase transition is shown if $\alpha<\tfrac{1}{2}$.

Both Theorem~\ref{thm:strongsurvival}$(i)$ and Proposition~\ref{cor:SurvivalByComparison} give criteria guaranteeing that there is no phase transition, but it depends on the value of $\eta$ which result applies for a larger range of $\alpha$. Indeed, which of the two provides the better result changes when the update speed of the system is moderate. This indicates that there might be a change of behaviour where our model is closer to the static case if the update speed is slow and for fast speed it behaves more like the penalised contact process. See Remark~\ref{rem:StarStrategy} for a more detailed discussion of this. Furthermore, we discuss this and other aspects in more depth in Section~\ref{sec:discussion}.

\subsection{Heuristics for the results on general evolving graphs}\label{subsec:HeuristicGenEvolGraph}
In this subsection, we provide a heuristic explanation for the strategy to prove Theorem~\ref{Thm:SubCriticalPhase} and Proposition~\ref{thm:comparisonpenal}, which is applicable to general graphs $G$. Essentially, both results consider the regime where the update speed of an edge is relatively fast. In this case, the system begins to resemble the behaviour of the penalised contact process introduced in~\eqref{PenTransition}.

Let us start with Theorem~\ref{Thm:SubCriticalPhase}. The proof strategy can be considered as an adaptation of the strategy used in \cite[Proposition~6.1]{jacob2017contact}. The key idea is to couple the CPDG with the so-called wait-and-see process, which is formally introduced at the beginning of Section~\ref{sec:thm1}. We couple these processes in such a way that if the wait-and-see process dies out, then so does the CPDG.

The wait-and-see process does not track whether an edge is open or closed, but whether it is \textit{revealed} or not. Every edge starts as unrevealed. If an edge $\{x,y\}$ is unrevealed the next infection event takes place with rate $\lambda p(d_x,d_y)$. If this event causes a successful infection, then this edge is marked as revealed. Here, \textit{successful} means that an infection is actually transmitted from $x$ to $y$. As long as $\{x,y\}$ is revealed further infection events take places with rate $\lambda$. After an update event on $\{x,y\}$ the edge is again marked as unrevealed. 

As already mentioned we are in a situation where the update speed is relatively fast compared to the occurrence of infection events. Thus, most infection events will take place with rate $\lambda p(d_x,d_y)$. This explains the similarity to the penalised contact process, which has exactly these infection rates. Furthermore, the faster the update speed, the closer our dynamics get  to those of the penalised version.

The formal proof proceeds as follows. After coupling %these two processes
the CPDG and the wait-and-see process we use a submartingale argument for a specifically chosen function of the wait-and-see process, which takes the number of revealed edges into account. It is shown that under the assumptions stated in Theorem~\ref{Thm:SubCriticalPhase}, there exists a sufficiently small $\lambda$ such that the wait-and-see process dies out almost surely.

The proof of Proposition~\ref{thm:comparisonpenal} uses a coupling argument originally developed by \cite{broman2007stochastic}. This enables us to couple the CPDG with a contact process with constant but edge dependent infection rates $(a_{\{x,y\}})_{\{x,y\}\in E}$ in such a way that if the latter survives, so does the former CPDG. The explicit form of $a_{\{x,y\}}$ can be found in the proof of Proposition~\ref{thm:comparisonpenal} at the end of Section~\ref{sec:thm1}. If the update rate is fast enough the resulting infection rate $a_{\{x,y\}}$ is reasonably close to the infection rate of the penalised contact process $\lambda p(d_x,d_y)$. This allows us to deduce that if the critical values of the penalised contact process satisfy $\lambda_i^{p}=0$, then the same holds  for the coupled process, and thus also for the CPDG. Furthermore, in the fast speed limit $a_{\{x,y\}}$ converges to $\lambda p(d_x,d_y)$, which is used to obtain the second statement.

\subsection{Heuristic for the evolving tree}\label{subsec:HeuristicEvolTree}

In this subsection we provide a heuristic explanation for the proof strategy of the first part of Theorem~\ref{thm:strongsurvival} and give a non-rigorous explanation why we require Assumption~\eqref{eq:assumtionkernel}. For convenience, we assume that $\eta\ge 0$  (the other cases are similar). The strategy is divided into three key steps, which we explain successively.

\textbf{Survival in a star:}
We consider the star graph induced by a vertex $x\in \cV$ of exceptionally large degree 
$D_x=N$  and its neighbours. 
Furthermore, we assume that the infection rate $\lambda$ is small but fixed. By assumption \eqref{eq:kernelsgen}, at any given time the vertex $x$ is on average connected to $N^{1-\alpha}$ neighbours via open edges (see \eqref{eq:averneig}) and the update speed is of order $\nu N^{\eta}$, for some $\nu>0$. Assume for the moment that only the centre $x$ is infected, while all neighbours are healthy as visualised in Figure~\ref{fig:InfectionStar:a}. Now on average $x$ infects
\begin{equation*}
	\frac{\lambda}{\lambda+1+\nu N^{\eta}}N^{1-\alpha}\approx \frac{\lambda}{\nu}  N^{1-\alpha-\eta},
\end{equation*}
many neighbours before $x$ itself recovers or the connecting open edge is updated. Note that $\lambda(\lambda+1+\nu N^{\eta})^{-1}$ is the probability that an infection  happens along an edge $\{x,y\}$ before $\{x,y\}$ updates or  $x$ recovers. 

Hence, after recovery of $x$ we are in the situation that roughly $\lambda \nu^{-1} N^{1-\alpha-\eta}$ neighbours are infected and the centre $x$ is healthy, as visualised in Figure~\ref{fig:InfectionStar:b}. The probability that all neighbours recover or that their edges update before they could reinfect the centre $x$ is approximately
\begin{equation*}
	\Big(\frac{1+ \nu N^{\eta}}{1+\lambda + \nu N^{\eta}}\Big)^{\frac{\lambda}{\nu} N^{1-\alpha-\eta}}\approx e^{-\left(\frac{\lambda}{\nu}\right)^2N^{1-\alpha-2\eta}}.
\end{equation*}
Thus,  once an order of $N^{1-\alpha-\eta}$ many infected neighbours is reached, the total number of infected neighbours will be kept at this order for a time of order $S=e^{c\lambda^2N^{1-\alpha-2\eta}}$  where $c>0$.
This implies that if $\alpha+2\eta<1$, then the infection persists for a stretched exponentially long time with high probability. This will then be enough time to pass the infection on to the next star. Note that we ignore all neighbours whose shared edge with the centre of the star were updated. These are indeed negligible, in Remark~\ref{rem:StarStrategy} this is explained in more detail.
\begin{figure}[t]
	\centering 
	\subfigure[Only the centre is infected, while all neighbours are healthy]{\label{fig:InfectionStar:a}\includegraphics[width=40mm]{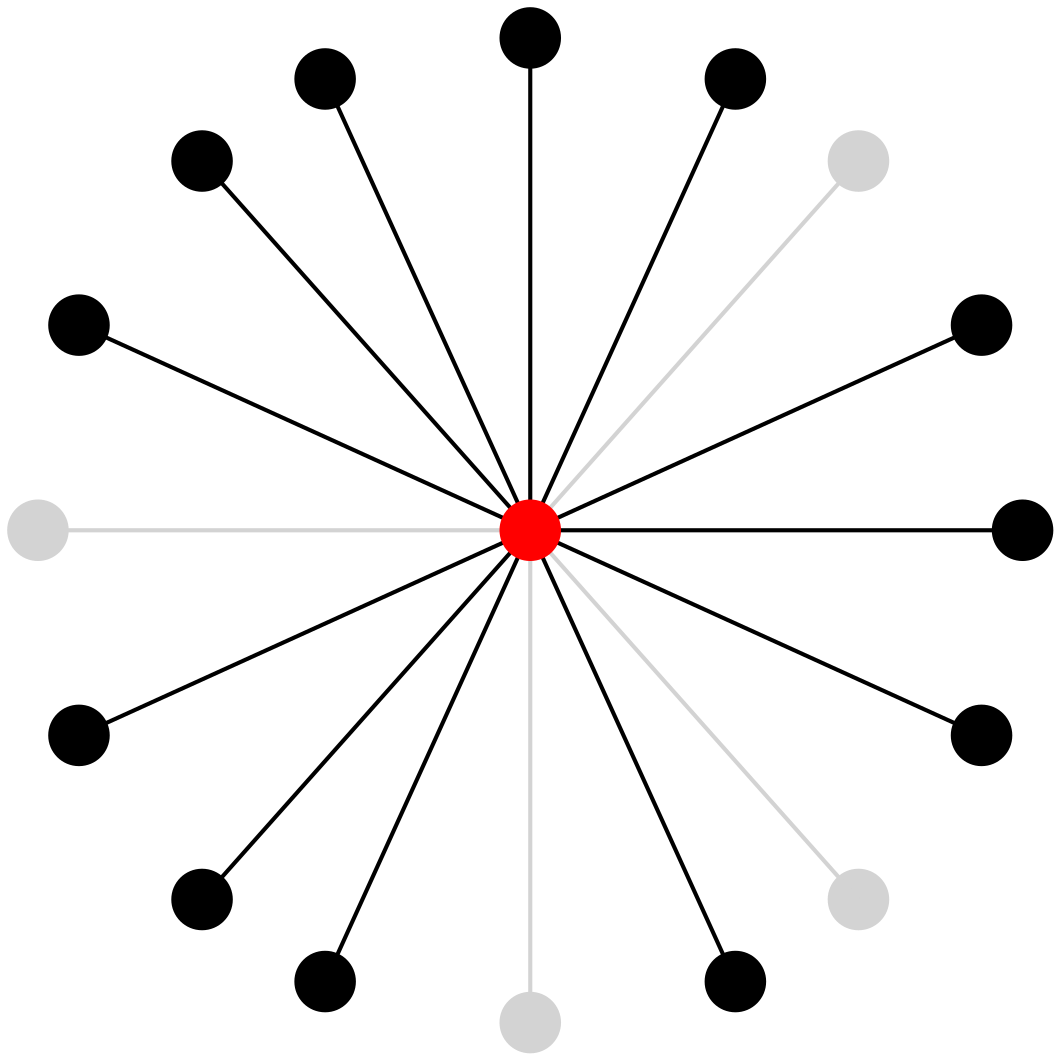}}\hspace{40mm}
	\subfigure[A sufficient number of neighbours is infected, while the centre is healthy]{\label{fig:InfectionStar:b}\includegraphics[width=40mm]{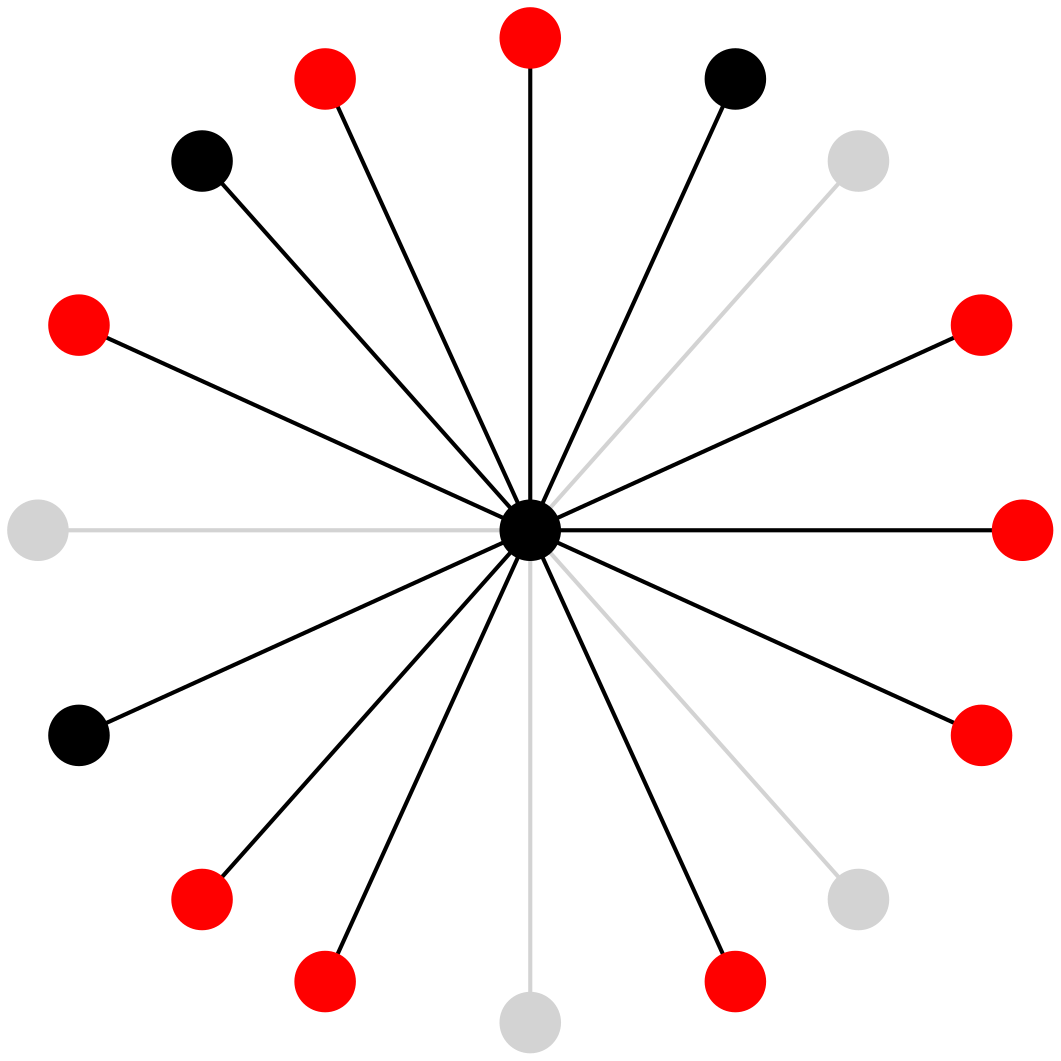}}
	\caption{Black and grey edges indicate open and closed edges, respectively. Red and black vertices indicate that an individual is infected and healthy, respectively.}
	\label{fig:InfectionStar}
\end{figure}

\textbf{Transmitting the infection to the next star:} Let us assume that the root $\rho$ is initially infected and has a large degree $\de_{\rho}=N$. Furthermore, suppose the centre $z$ of the next star that has $N$ neighbours is in generation $r$ of  $\cT$. Then, the probability that the infection originating from the root $\rho$ is transmitted to $z$ in less than $4r$ time units is approximately
\begin{equation}\label{eq:HeuristicProbTransmissionForGenr}
	\approx N^{-2\alpha}K^r,
\end{equation}
for some $0<K<1$ (see Lemma \ref{lem:probapath} for details). By the previous heuristics we know that an infection can survive in a star of size $N$ for a time of order $S=e^{c \lambda^2 N^{1-\alpha-2\eta}}$, which means that for a time period of length $S$ it is possible for the infection to be transmitted from $\rho$ to~$z$ without the infection in the star around $\rho$ becoming extinct. 
Thus, the probability of a successful transmission can be lower bounded by roughly $S/(4r)$ trials of attempting to infect vertex $z$ in one go in a time span of length less than $4r$. Then the probability that at least one trial is successful is approximately
\begin{equation*}
	\approx 1-\big(1-N^{-2\alpha}K^r\big)^{S/4r}\geq 1- \exp\bigg(- \frac{N^{-2\alpha}K^rS}{4r} \bigg).
\end{equation*}
If the generation $r$ is sufficiently small the transmission happens with high probability, i.e.\ we need that
\begin{equation}\label{FirstHeuristicGenBound}
	r \ll \log (S)= c \lambda^2 N^{1-\alpha-2\eta}
\end{equation}

\textbf{Probability to find another star:} In the last step we treat the question, for which offspring distributions it is possible to find a sequence $r_N$, which grows sufficiently slowly. This boils down to the question how far into the tree do we need to look to find a star of size $N$ for a given distribution. Let us first denote by
\begin{equation*}
	M(r,N):=\# \{  \text{vertices with}\ N \ \text{neighbours in generation}\ r\},
\end{equation*}
so that the expected number of stars of size $N$ in generation $r$ is approximately given by
\begin{equation*}
	\IE[M(r,N) \mid \de_{\rho}=N] \approx N \mu^{r-1}\IP(\de_\rho = N).
\end{equation*}
Now it is easy to see that $\IE[M(r,N)\mid \de_{\rho}=N]\geq N$  if 
\begin{equation*}
	r_N\gg-\frac{\log \IP(\de_\rho = N)}{\log \mu }.
\end{equation*}
Thus, together with \eqref{FirstHeuristicGenBound} we found a lower and upper bound on the sequence $r_N$, i.e.\
\begin{equation}\label{FinalHeuristicBounds}
	-\log \IP(\de_\rho = N) \ll r_N \ll c\lambda^2 N^{1-\alpha-2\eta}.
\end{equation}
By this reasoning it follows that if the distribution of $\zeta$ satisfies \eqref{eq:assumtionkernel} with $\eta\ge 0$,
then we can find a sequence $r_N$ such that with  probability bounded away from zero the infection is transmitted to the next star and then back and forth between stars. 
This can then be used to show that that $\rho$ is reinfected infinitely often for any $\lambda>0$ which implies strong survival in case $\alpha+2\eta< 1$. This concludes the heuristic explanation of $(i)$.

On the other hand, if $\alpha\in \big((1-2\eta)\vee 0,1\big)$ then this argument is no longer true for all $\lambda>0$. But it can be modified such that we can conclude strong survival if the infection rate $\lambda$ is chosen large enough. First, the size $N$ of the stars is chosen large but fixed, then by \eqref{FinalHeuristicBounds} if $\lambda$ is chosen to be at least of order $N^{2\eta}$ it still holds true that the infection is transmitted back and forth between stars of size $N$ with a sufficiently high probability. This can then be used to show $(ii)$, i.e.\ strong survival for a sufficiently large $\lambda$ as long as $\alpha<1$.

\begin{remark}\label{rem:StarStrategy}
	Our proof strategy is similar to the strategy proposed for the static case by Pemantle \cite{pemantle1992}.
	But there are some significant differences. The basic idea for the survival of the infection in a star is that if there is a sufficient number of infected neighbours, in our case $\lambda N^{1-\alpha-\eta}$, then the infection can persist on its own for a long time. In comparison to the static case we lose roughly a factor of $N^{-\eta}$ many infections in the first step since additionally to recoveries edges update. Since the star is exceptionally large the edge will be most likely closed after an update. Thus, only $\lambda N^{1-\alpha-\eta}$ many neighbours can be sustained and we lose another factor of $N^{-\eta}$ for the same reason in the probability of reinfection, which is $1-\exp(-\lambda^2N^{1-\alpha-2\eta})$. In the static case,  we would have in both terms $N^{1-\alpha}$ instead.
	
	Clearly as $\eta$ increases this strategy becomes worse and is no longer applicable for $\eta\geq \tfrac{1}{2}$. On the other hand, neighbours $y$ whose shared edge has been updated after receiving the infection from $x$, infect $x$ at rate $\lambda p (\de_x,\de_y)$. This is the same rate as for the penalised contact process. So for $\eta\geq \tfrac{1}{2}$ switching to a strategy in which only updated neighbours are considered, using heuristics similar to the above, the stars of size $N$ should retain the infection for a time of order $e^{c\lambda^2N^{1-2\alpha}}$ with probability $1-\exp(-\lambda^2N^{1-2\alpha})$, which is comparable to the results shown in \cite{zsolt}. 
	
	This suggests that in the parameter regime of moderate update speed there is a change of behaviour, such that the process behaves less like a contact process on a static pruned tree, but rather like a penalised  version on the full tree. The tipping point seems to be at $\eta= \tfrac{1}{4}$.  		
\end{remark}

\bigskip

\textbf{Outline of the article.}
The remaining paper is structured as follows. In Section~\ref{sec:discussion}, we discuss connections to existing results and some open problems. In Section~\ref{sec:preliminaries}, the graphical representation for the contact process is introduced together with the results showing that the critical values do not depend on the particular realisation of an infinite BGW tree. Section \ref{sec:thm1} is devoted to the proof of Theorem \ref{Thm:SubCriticalPhase} and Proposition \ref{thm:comparisonpenal}. In Section \ref{sec:CPstars}, we show Theorem \ref{thm:strongsurvival}. The first step of the heuristic, survival in a star, is treated in Subsection~\ref{sec:survivalStar}, the second step in Subsection~\ref{sec:PushingInfectionPath} and \ref{sec:PushingInfection}, and the final step together with the proof of Theorem~\ref{thm:strongsurvival} is done in Subsection~\ref{sec:thmstrong}.

\section{Discussion}\label{sec:discussion}

In this section we discuss our main results in the context of the CPDG on a BGW tree. We concentrate on our main example, i.e.\ we assume that the update speed $v$ satisfies~\eqref{eq:kernelsgen} and the connection probability is of the form $p = p_{\alpha,\sigma}$ as defined in \eqref{eq:kernels}. This means that for every $m\in \IN$ there exist $\nu_1,\nu_2>0$ such that 
\begin{equation}\label{eq:leadingexample}
	\nu_1 n^{\eta}\leq v(n,m)\leq \nu_2 n^{\eta}\quad \text{ and } \quad p_{\alpha,\sigma}(n, m)=1 \wedge \kappa \big((n\wedge m)^{\sigma}(n\vee m)\big)^{-\alpha},
\end{equation}
for all $m\leq n$, where $\kappa,\alpha>0$, $\eta\in \IR$ and $\sigma\in[0,1]$.

Theorem~\ref{Thm:SubCriticalPhase} and Proposition~\ref{thm:comparisonpenal} are formulated for deterministic graphs. This means that these two results are applied conditionally on the realisation of the BGW tree $\cT$. Thus, the specific choice of the offspring distribution $\zeta$ does not have an impact on these results. On the other hand, Theorem~\ref{thm:strongsurvival} does depend on the choice of the offspring distribution. In the following we discuss the two cases of a power-law  and a stretched exponential offspring distribution.

\subsection{Discussion for power law offspring distributions} 

Consider an offspring distribution $\zeta$ that follows a power law, i.e.\ $\mathbb{P}(\zeta = N) \sim c N^{-{b}}$ for $N$ large with ${b}>1$. An immediate consequence of Corollary~\ref{cor:Extinction}, Theorem~\ref{thm:strongsurvival} and Proposition~\ref{cor:SurvivalByComparison}  is the following result.

\begin{corollary}\label{cor:SummarisationPowerLawResults}
	Consider the CPDG on a BGW tree whose offspring distribution
	follows a power law and has mean larger than $1$. Assume that the connection probability $p$ and the update speed $v$ satisfy~\eqref{eq:leadingexample}. 
	\begin{itemize}
		\item[(i)] If $\eta\ge 0$ and either $\alpha \geq 1$ or if $\alpha+2\eta\ge 1$ and $\alpha\sigma \ge 1/2$
		then $\lambda_1 > 0$. 
		\item[(ii)]  If either $\eta \leq 0$ and $\alpha \in [0,1)$ or if $0 \leq \eta \leq 1/2$ and $\alpha \in (0, 1- 2 \eta)$ then $\lambda_1 = \lambda_2 = 0$.
		\item[(iii)] If $\eta\ge 0$, $1-2\alpha>0$ and $\IP(\zeta=0)=0$ then $\lambda_1=\lambda_2=0$.
	\end{itemize}
	Moreover, if $\alpha < 1$ then $\lambda_1\leq \lambda_2 < \infty$. 
\end{corollary}

We illustrate the resulting phase diagram in Figure~\ref{fig:SigmaKernel}. A first observation is that the choice of  $\kappa$,$\nu_1, \nu_2$ as well as the offspring tail exponent $b$
have no effect on the existence or non-existence of a subcritical phase. 

\begin{figure}[t]
	\centering
	\includegraphics[width=70mm]{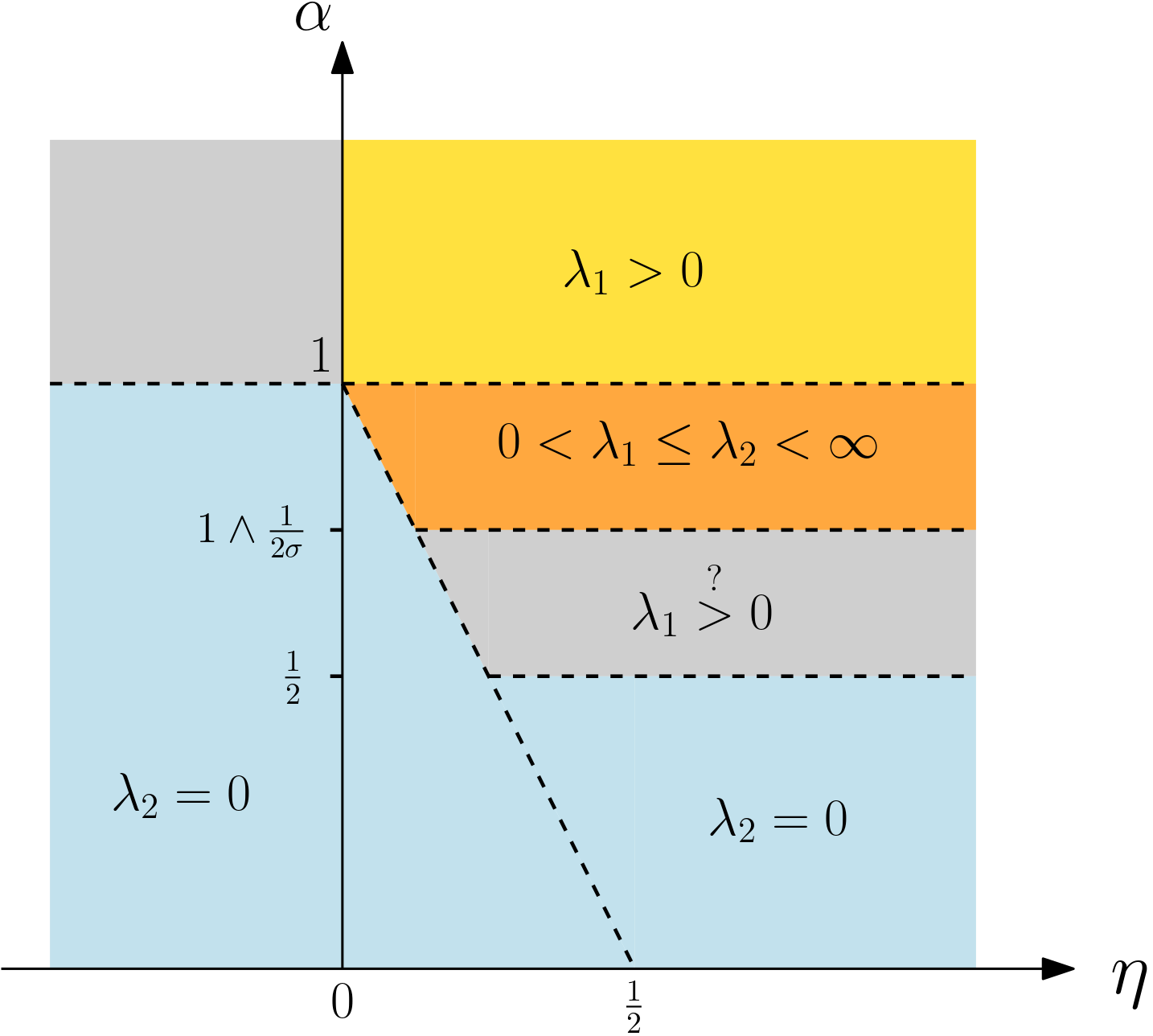}
	\caption{Illustration of the phase diagram when the offspring distribution $\zeta$ follows a power law, $\IP(\zeta=0)=0$ and the connection probability is chosen to be $p_{\alpha,\sigma}$ with $\alpha \geq 0$ and $\sigma\in[0,1]$. }\label{fig:SigmaKernel}
\end{figure}

{\bf Comparison with the static case.} At time $t = 0$ (and in fact at any given fixed time), the distribution of the number of children of the 
root is the same as that of
\begin{equation}\label{eq:def_zeta_p}
	\zeta^p := \sum_{i=1}^\zeta I_i.  
\end{equation}
Here $\zeta$ and $\zeta_i, i \geq 1,$ are independent copies of the offspring number and 
conditionally on these the random variables $I_i, i \geq 1,$ are independent Bernoulli variables with success probabilities
$p(\zeta, \zeta_i)$, respectively.
By calculating the  moment-generating function, one can show that $\zeta^p$ has the same distribution as a mixed
Binomial distribution ${\rm Bin}(\zeta, \mathbb{E}[p(\zeta',\zeta)\, |\, \zeta])$, where $\zeta'$ is an independent copy of $\zeta$.

It is plausible to assume that if we do not allow any updates in the tree, then the 
contact process on this tree behaves in the same way as the contact process on a static BGW($\zeta^p$) tree. 
The difference coming from the correlation of the percolation probabilities of consecutive generations
should be negligible. 
Assume that $\alpha< 1$ and that $\zeta$ has a power law distribution with exponent
$b$. Then, 
using a local CLT for binomial random variables and their strong concentration properties, one can show that 
\[ \begin{aligned} \mathbb{P}( \zeta^p = k) 
	& \asymp \sum_{n = k}^\infty n^{-b} \mathbb{P}( \textrm{Bin} (n, n^{-\alpha}) = k ) \\
	& \asymp \sum_{ n  = k }^\infty n^{-b} \frac{1}{\sqrt{n^{1-\alpha}}} e^{- \frac{(k - n^{1-\alpha})^2}{2 n^{1-\alpha}(1- n^{-\alpha}) }} \\
	& \asymp k^{- \frac{b}{1-\alpha}-\frac 12} \sum_{n: |n^{1-\alpha} -k | \leq \sqrt{k}} e^{- \frac{(k - n^{1-\alpha})^2}{2 n^{1-\alpha}(1- n^{-\alpha}) }}\\
	&\asymp k^{- \frac{b}{1-\alpha}-\frac 12} \Big( \big(k+\sqrt{k}\big)^{\frac{1}{1-\alpha}} - \big(k-\sqrt{k}\big)^{\frac{1}{1-\alpha}}\Big)\\   
	& \asymp k^{- \frac{b -\alpha}{1-\alpha}} . 
\end{aligned} \]
where  $\asymp$ is to be interpreted as upper and lower bounds that agree up to multiplicative constants
and we use that the sum concentrates around those values of $n$ such that 
$n^{1-\alpha}  - \sqrt{n^{1-\alpha}}\leq k \leq n^{1-\alpha}  + \sqrt{n^{1-\alpha}}$. In particular, $\zeta^p$ has a power law distribution with exponent $(b-\alpha) / (1-\alpha)$. The results for the contact process on a static BGW-tree therefore suggest that 
in the zero update speed case, we always have that for the corresponding critical value $\la^{\rm static}_2 = 0$. In the case $\alpha \geq 1$, a standard calculation shows that the percolated offspring distribution $\zeta^p$ has bounded exponential
moments, so on a static BGW$(\zeta^p)$-tree the contact process has a phase transition.

Indeed, comparing to our model, we have  thatwhen $\eta = 0$, then our model exhibits a phase transition in the same cases as  the model on a static BGW($\zeta^p$)-tree, i.e.\ when $\alpha \geq 1$. Moreover we conjecture that the same is true in the full \emph{slow speed regime} $\eta \leq 0$.

However, there is also a major difference between the static  and our dynamic model. In the static case, we can choose $\kappa$ small enough such that the average degree satisfies $\mathbb{E}[ \zeta^p] < 1$. In particular, in that case a BGW($\zeta^p$)-tree  is
finite almost surely and the contact process is always sub-critical, i.e.\  $\lambda_1^{\rm static} = \infty$. However, Corollary~\ref{cor:SummarisationPowerLawResults} also shows that if $\alpha <1$, then $\lambda_1\leq\lambda_2 < \infty$. Thus, the updating helps the contact process by opening up space so that it can survive for large $\lambda$. Furthermore, for $\eta\leq 0$ the difference becomes even more apparent since then it holds that $\lambda_1=\lambda_2 =0$. Heuristically one could argue that by letting $\eta\to -\infty$ the model should behave more similarly to the static case, since edge updates become increasingly rare, but this is not reflected in the critical rates.
We will come back to this issue below.

\textbf{Comparison with the penalised contact process.}
Assume for now that $\sigma = 1$, so that $p$ corresponds to a product kernel.
As discussed previously, for large update speeds the contact process behaves like 
the penalised contact process. From the results of~\cite{zsolt}, we know that 
in this case, the penalised contact process satisfies $\lambda_1^{p} > 0$ if 
and only if $\alpha\ge 1/2$. 
Under the assumption that $\sigma = 1$, our results fully characterise (with the exception of the case $\eta<0$ and $\alpha>1$) the behaviour
of the process. In particular, our model exhibits exactly the same  behaviour in the \emph{fast speed regime} 
as the penalised contact process
when 
$\eta \geq 1/4$. 

In the case $\sigma = 1$, we can thus observe that the dynamical model changes from 
the static behaviour for $\eta = 0$ to the penalised behaviour for $\eta \geq 1/4$. 
Moreover, the critical line $1-2\eta$, $\eta \in [0,1/4]$  interpolates linearly between 
the two extreme cases.

For general $\sigma$, the situation in the fast speed regime is more complicated and indeed Corollary~\ref{cor:SummarisationPowerLawResults} does not fully 
characterise when there is a phase transition.
Nevertheless, we believe that the same critical line as for $\sigma = 1$ also characterises 
the regime between existence and non-existence of a phase transition.

However, it might be that additional moment assumptions on the degree distribution are required.
For the special case that the penalisation kernel is given by $p_{\alpha,0}(n,m)=(n\vee m)^{-\alpha}$ and  for $\alpha\in (1/2,1)$,  \cite[Theorem~2.5]{zsolt}
shows that if $\IE[\zeta^{1-\alpha}]<\infty$ (or equivalently that the tail exponent satisfies $b > 2 - \alpha$), then $\lambda^{p}_2\geq \lambda^{p}_1>0$. 
However, when 
${b}\in(1,2-\alpha)$, then  $\lambda^p_1=0$ but $\lambda^p_2>0$.
Of course, Proposition~\ref{thm:comparisonpenal} implies directly that for this special case also in our model $\lambda_1=0$ holds, but unfortunately we cannot make any statement about $\lambda_2$. 
Furthermore,
\begin{equation*}
	\big((n\wedge m)^{\sigma}(n\vee m)\big)^{-\alpha}\geq\big(n\vee m\big)^{-(\alpha+\alpha\sigma)},
\end{equation*}
and thus by coupling the statement of \cite[Theorem~2.5(b)]{zsolt} can be extended to $p_{\alpha,\sigma}$ for some specific choices of $\sigma$ and $\alpha$, but not the whole grey region in Figure~\ref{fig:SigmaKernel}. 
However, we believe this restriction to be technical and thus think that in the fast speed regime our model is comparable to the penalised contact process.

\begin{conjecture}
	Assume the parametrisation \eqref{eq:leadingexample} and  $\frac{1}{2} <\alpha<1\wedge\frac{1}{2\sigma}$ and $\alpha+2\eta \ge 1$, then $\lambda_1 > 0$ if and only if $\lambda_1^p > 0$ and $\lambda_2 > 0$ if and only if $\lambda^p_2 > 0$.
\end{conjecture}

\textbf{Immunisation.}
The second statement of Theorem~\ref{thm:strongsurvival} implies that if $\alpha<1$, then the critical values are finite, $\lambda_1\le \lambda_2<\infty$, regardless of the other parameter choices. This implies that in this situation no \textit{immunisation} occurs. (We say that immunisation occurs in a system if for a specific parameter choice $\lambda_1=\infty$.) This is a fundamental difference to the contact process 
on a dynamical percolation graph considered on $\mathbb{Z}$ in~\cite{linker2020contact} and in the  long-range setting in \cite[Theorem~2.5]{seiler2022long} where an immunisation phase takes place. 
In these models, immunisation occurs if the average degree of the  (long-range) percolation graph  at every time step $t$ is small enough.
However, as discussed before, by choosing $\kappa$ small, in our setting we can make the average degree at any fixed time arbitrarily small and still have no immunisation phase. 
One essential difference to the dynamical long-range percolation case is that in the case of $\alpha<1$ we retain the occurrence of vertices with exceptionally large degree. If we know that $\de_x=N$ then the average degree in the thinned tree is $N^{1-\alpha}$. Even though the thinned tree (or rather forest) is fairly sparse or might even only consist of finitely many connected components at any given time point, if we reach a vertex with an exceptionally large degree we can again survive exponentially long in the case that $\lambda$ is sufficiently large. This then provides enough time to wait until a connection to another star is established and the infection is transmitted to the next star. On the contrary, if $\alpha\geq 1$ vertices with exceptionally large degree no longer occur. This fact motivates the following conjecture.
\begin{conjecture}
	Assume the parametrisation \eqref{eq:leadingexample}. If $\alpha\geq 1$ then for $\kappa>0$ small enough there exists a $\nu'=\nu'(\kappa, \eta)$ such that $\lambda_1=\infty$ for all $\nu_2<\nu'$.
\end{conjecture}

\textbf{Explosion.} Finally, we want to comment on the possibility that the CPDG started with finitely many infected vertices explodes in finite time with positive probability. If the offspring distribution $\zeta$ has finite mean, i.e.\ $\IE[\zeta]<\infty$, then the process does not explode in finite time almost surely. This follows by a direct coupling with the classical contact process, where it is know that the process surely does not explode if $\IE[\zeta]<\infty$, see e.g.\  \cite[Theorem 3.2]{cardona2021contact} in the case when the fitness is constant equal to one.

On the other hand, if $\IE[\zeta]=\infty$ the classical contact process seems to explode in finite time with positive probability. In the case that $b\in (1,2)$ it appears to us that this can be shown analogously to \cite[Theorem~2.5(b)]{zsolt} or more precisely \cite[Proposition~6.2]{zsolt}. Note that in that proof it is shown that $0=\lambda_1^p<\lambda_2^p$ for the maximum kernel $p_{\alpha,0}$ as penalisation when $\alpha\in (1/2,1)$ and $b\in(1,2-\alpha)$. This is proved by showing that the infection travels infinitely deep into the tree after a finite time, but it does only return to the root finitely often. Thus, in particular this implies explosion in finite time.

Furthermore, the restriction $\alpha>1/2$ seems irrelevant for the proof and is in our opinion only in place since $\alpha\leq 1/2$ is already covered by another result, such that $\alpha=0$ should be possible. For $b=2$, i.e.\ $\IE[\zeta]=\infty$ but $\IE[\zeta^r]<\infty$ for all $r<1$, this does not seem to work and it is not clear if the process explodes in finite time.

In this scenario Theorem~\ref{Thm:SubCriticalPhase} implies for the CPDG that if $\alpha \sigma\geq 1/2$ and $\alpha+2\eta\ge 1$, then the probability of explosion in finite time is zero. But in the other parameter regions it is a priori not clear if the process explodes with a positive probability or not.

Again a comparison with the penalised contact process suggests that it is in fact very much possible that explosion occurs for certain parameter choices.

\subsection{Discussion of stretched exponential offspring distributions} 

In the following, we will
assume that $\zeta$ has stretched exponential tails, i.e.\ that there exists a $\beta\in (0,1)$ such that
\begin{equation}\label{eq:HeavierThanStreched}
	\mathbb{P}(\zeta = N) \sim c\exp(- N^{\beta}) ,
\end{equation}
for some $c > 0$ as $N \rightarrow \infty$.
Combining Corollary~\ref{cor:Extinction}, Theorem~\ref{thm:strongsurvival} and Proposition~\ref{cor:SurvivalByComparison}, 
gives the following result for this special case.
\begin{corollary}\label{cor:SummarisationStrechtedExpResults}
	Consider the CPDG on a BGW tree whose offspring distribution has 
	stretched exponential tails such that
	\eqref{eq:HeavierThanStreched} holds for some $\beta\in (0,1)$ and has mean larger than $1$. Assume that the connection probability $p$ and the update speed $v$ satisfy~\eqref{eq:leadingexample}. 
	\begin{itemize}
		\item[(i)] If $\eta\ge 0$  and either $\alpha \geq 1$ or if $\alpha+2\eta\ge1$ and $\alpha\sigma \ge 1/2$, then $\lambda_1 > 0$. 
		\item[(ii)]     If either $\eta \leq 0$ and $\alpha \in [0,1-\beta)$ or if $\eta \leq (1-\beta)/2$ and $\alpha \in (0, 1- \beta-2 \eta)$, then $\lambda_1 = \lambda_2 = 0$.
		\item[(iii)] If $\eta \ge 0$, $1-\beta -2\alpha>0$ and $\IP(\zeta=0)=0$, then $\lambda_1=\lambda_2=0$.
	\end{itemize}
	Moreover, if $\alpha < 1 - \beta$, then $\lambda_1\leq \lambda_2 < \infty$.
\end{corollary}
\begin{figure}[t]
	\centering
	\includegraphics[width=70mm]{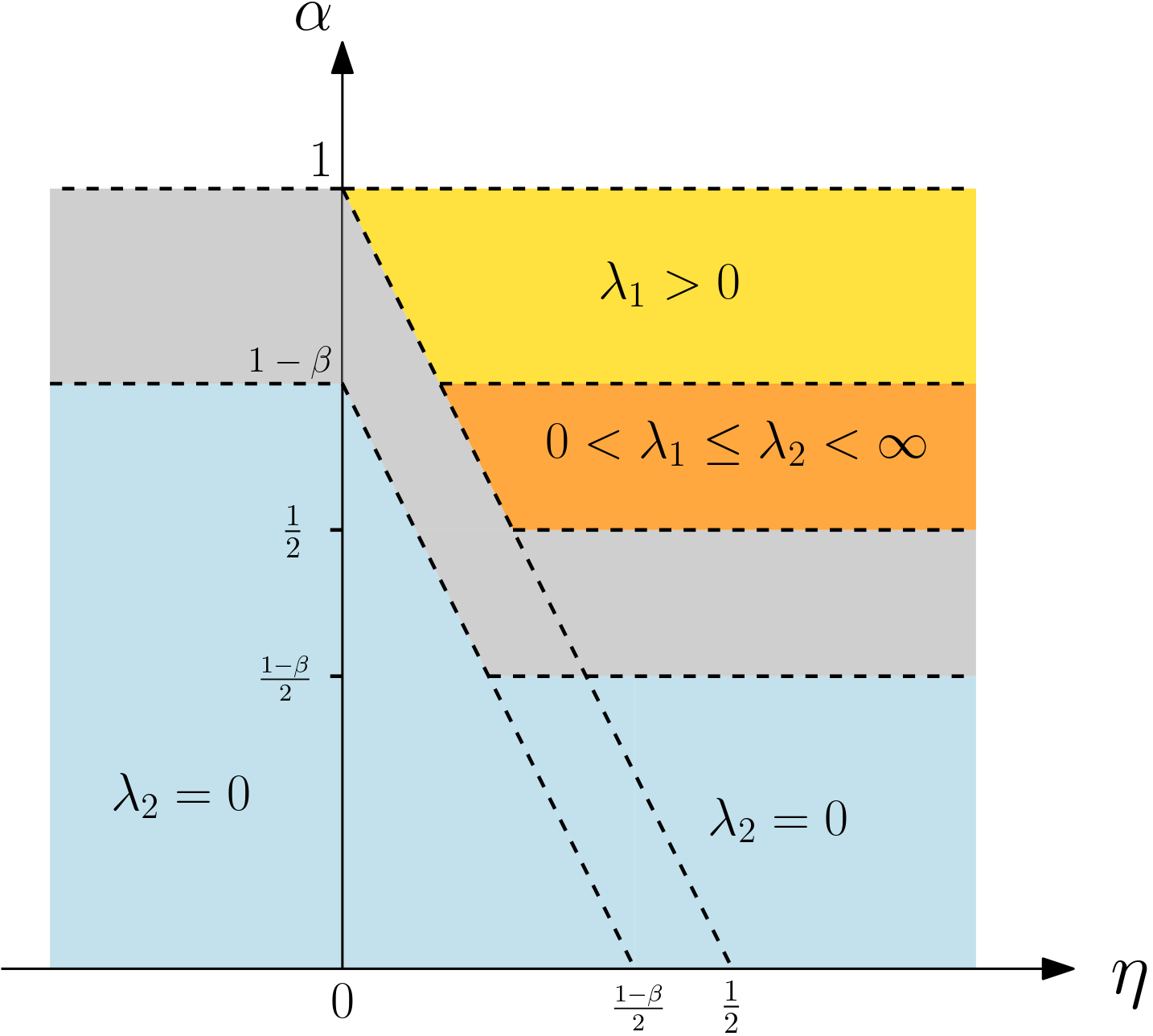}
	\caption{Illustration of the phase diagram when the offspring distribution has a stretched exponential tail, $\IP(\zeta=0)=0$ and the connection probability is chosen to be the product kernel $p_{\alpha,1}$.}\label{fig:stret}
\end{figure}
In the case  $\alpha\geq 1$ our results do not change, which is not surprising since these results do not depend on the specific graph structure, but only how  $p$ depends on the degrees.
But if $\alpha< 1$, choosing this class of distributions with slightly ``lighter" tails compared to the power law distributions has in fact consequences for Theorem~\ref{thm:strongsurvival} and Proposition~\ref{cor:SurvivalByComparison}. In this case we get an additional restriction for our parameter regime, 
in that $\alpha+2\eta<1-\beta$ needs to be satisfied  for~\eqref{eq:assumtionkernel} to hold so that we can deduce $\lambda_2=0$ from Theorem~\ref{thm:strongsurvival}$(i)$.
Similary, the condition $\alpha<1-\beta$ is necessary so that
\eqref{eq:assumtionkernel2} holds and so that Theorem~\ref{thm:strongsurvival}$(ii)$ yields $\lambda_2<\infty$. We illustrate the additional restrictions in Figure~\ref{fig:stret} with the product kernel $p_{\alpha,1}$ as connection probability.

To see that the additional restriction
$\alpha<1-\beta$ is natural, it is helpful to again compare to the static case. 
As discussed before, without updating the model should behave like a contact process
on a static BGW$(\zeta^p$)-tree, where $\zeta^p$ is defined in~\eqref{eq:def_zeta_p}. 
However, if $\zeta$ satisfies~\eqref{eq:HeavierThanStreched} and $\alpha < 1-\beta$, then 
$\zeta^p$ does not have any finite exponential moments, so that the static critical parameter satisfies $\lambda_2^{\rm static} = 0$.
Moreover,
for $\alpha > 1- \beta$, the percolated offspring number $\zeta^p$
has some finite exponential moments again, so that the static model exhibits a phase transition. 
Hence, in the slow updating regime $\eta \leq 0$ the dynamic model shows the same behaviour as 
the static case.

For $\eta > 0$, the comparison with the penalised contact process yields that
for $\alpha < (1-\beta)/2$ we have $\lambda_2 = 0$.
However, for the penalised contact process, as well as for our model,  there is a gap in the fast speed case
as illustrated by the grey area in Figure~\ref{fig:stret}, where it is not clear whether a phase transition exists or not.
Similarly to the case of power-law tails, we would expect a linear transition from the static to the penalised behaviour. 
However, if we consider  the heuristic explanation in Subsection~\ref{subsec:HeuristicEvolTree} and see in the last part that if $\alpha+2\eta > 1-\beta$, then the probability of finding another star of the same size in a reasonable number of steps is not necessarily  bounded away from $0$ such that the strategy of survival through stars might not work.

This raises the following open problem.
\begin{problem}
	If we assume that $\alpha > (1-\beta)/2$ and
	\begin{equation*}
		\limsup_{N\to \infty} \frac{\log \mathbb{P}(\zeta = N)}{N^{1-\alpha- 2(\eta\vee 0)}}\in (-\infty,0) \ 
	\end{equation*}
	does this imply that $\lambda_1>0$? 
\end{problem}
One possible approach to solve this problem might be to adapt techniques used by \cite{bhamidi2021survival} to a  dynamical setting. With regards to the occurrence of immunisation effects a similar question arises.
\begin{problem}
	Assume $p$ and $v$ satisfy \eqref{eq:leadingexample} and that $\alpha \geq 1- \beta$. For $\kappa>0$ small enough, does there exist a $\nu'=\nu'(\kappa,\eta)$ such that $\lambda_1=\infty$ for all $\nu_2<\nu'$?
\end{problem}

\subsection{Comparison with the contact process on finite but large dynamical random graphs}

In this final section, we discuss the connections to the models considered in~\cite{jacob2017contact, jacob2019metastability, jacob2022contact}. These study the contact process with either vertex- or edge-updating on 
inhomogeneous random graphs as the number of vertices tends to infinity. 
In particular, they give criteria for the existence and non-existence of a phase transition between fast and slow extinction that depend on 
the exponent $\eta$ governing the update  speed   as well as the power-law exponent ($\tau$ in their notation) of the limiting degree distribution of a uniformly chosen vertex.
This is in stark contrast to our results, where there is no dependence on the underlying degree exponent (in the power-law case). 
While the local dynamics of a typical vertex seem similar, 
these models are difficult to compare. 
In  their model the `base graph' is really the complete graph  and the inhomogeneous thinning produces the 
sparse graph, whereas in our case the underlying 
tree structure is preserved and we have an additional parameter $\alpha$ that controls the thinning.

\section{Preliminaries}\label{sec:preliminaries}

\subsection{Graphical representation}
\label{sec:GraphRep}
In this section we provide the graphical representation construction for the dynamical contact process on $\dG$.  The general idea is to record the infections and recoveries of the process on $\dG$ on a space-time domain. Let $\Delta^{\text{op}}= (\Delta_{\{x,y\}}^{\text{op}})_{\{x,y\}\in \dE}$ and  $\Delta^{\text{cl}}= (\Delta_{\{x,y\}}^{\text{cl}})_{\{x,y\}\in \dE}$ be two independent
families of Poisson point processes. In other words, for every $\{x,y\}\in \dE$, let $\Delta_{\{x,y\}}^{\text{op}}$ and $\Delta_{\{x,y\}}^{\text{cl}}$ be independent Poisson processes with rates $v(d_x,d_y)p(d_x,d_y)$ and $v(d_x,d_y)(1-p(d_x,d_y))$, respectively. These two processes represent the times at which the edge $\{x,y\}$ opens and closes. Furthermore, we define
\begin{equation*}
	\Delta_{\{x,y\}}^{\text{up}}:=\Delta_{\{x,y\}}^{\text{op}}\cup\Delta_{\{x,y\}}^{\text{cl}},
\end{equation*}
which is a Poisson process with rate $v(d_x, d_y)$ and provides the times of update events on the edge $\{x,y\}$. For every $e = \{x,y\}\in \dE$, let us define the following sequence of stopping times 
\begin{equation*}
	T_0(e):=0\qquad \text{and}\qquad  T_n(e):=\inf\{t>T_{n-1}(e):\ t\in \Delta^{\text{up}}\}, \qquad \text{for all}\quad n\ge 1.
\end{equation*}
Next let $B\subset \dE$ be the initial configuration of the background process $\bfB$, i.e.~$\bfB_0=B$. For every $e\in \dE$, we denote by $(X_t(e))_{t\geq 0}$  a two state Markov process as follows: we set 
\begin{align*}
	X_t(e)=\begin{cases}
		1 &\text{ if }\  T_n(e)\in \Delta^{\text{op}}_e\\
		0 &\text{ if }\  T_n(e)\in \Delta^{\text{cl}}_e
	\end{cases},
\end{align*}
for all  $t\in [T_n(e),T_{n+1}(e))$ with $n\geq 1$. Then we define $\bfB_t:=\{e\in E: X_t(e)=1\}$. 

Finally, we define the infection process $\bfC$. Let $\Delta^{\inf}=(\Delta^{\inf}_{\{x,y\}})_{\{x,y\}\in \dE}$ and $\Delta^{\text{rec}}=(\Delta^{\text{rec}}_x)_{x\in \dV}$  be two independent families of Poisson processes such that for all $\{x,y\}\in \dE$ and $x \in \dV$ fixed, the processes $\Delta^{\inf}_{\{x,y\}}$ and $\Delta^{\text{rec}}_x$ have rate $\lambda$ and $1$, respectively. Now we need to introduce the notion of an \textit{infection path}. Let  $(y, s)$ and $(x, t)$ with $s < t$ be space-time points. We write $(y, s)\stackrel{\bfB}{\longrightarrow} (x, t)$ if there exists  a sequence of times $s = t_0 < t_1 < \dots < t_n \leq t_{n+1} = t$ and space points $y = x_0,x_1,\dots, x_n = x$ such that $(\{x_{k},x_{k+1}\},t_k)\in \Delta^{\inf}$ as well as  $\{x_{k},x_{k+1}\}\in \bfB_{t_k}$
and  $\Delta^{\text{rec}}\cap\big(\{x_k\}\times[t_k , t_{k+1} )\big)=\emptyset$ for all $k \in \{ 0, \dots , n\}$ with $n\in \mathbb{N}$.

Now in order to define the infection process $\bfC$ with initial condition $(C,B)$ we choose a background process $\bfB$ with $\bfB_0=B$ and set $\bfC_0:=C\subset \dV$ as well as
\begin{align}\label{DefinitionCPDLP}
	\bfC_t:=\{x\in \dV:\exists y\in C \text{ such that } (y,0)\stackrel{\bfB}{\longrightarrow} (x,t) \}, \quad \text{ for } t > 0.
\end{align}
The process $(\bfC,\bfB)$ we just defined has the same transition probabilities as in \eqref{def:BackgroundProcess} and \eqref{def:InfectionProcess}.

By now, it is a well-established approach to define variations of the contact process via a graphical representation. If one considers graphs with certain properties, for example finite graphs or infinite graphs with uniformly bounded degrees, it follows by standard results that the constructed process is a Feller process, see for example \cite{liggett05} or \cite{swart2022course}. However, in the more general case, where we only assume that the graph is connected and locally finite it is not clear that the resulting process is Feller.

Since in our setting, we cannot ensure that $(\bfC,\bfB)$ is itself a Feller process, one option is to consider an approximating sequences of processes defined on a sequence of finite subgraphs $G_n$ of $G$ by restricting the graphical representation to $G_n$. To ensure that this sequence of processes, which are then Feller processes, indeed approximates $(\bfC,\bfB)$ we needed to explicitly assume that we only allow infection paths of finite length in the definition of our CPDG. In fact, this is not a real restriction if the infection process $\bfC$ does not explode in finite time. For more details on this see the discussion in Section~\ref{sec:thm1} right before Proposition~\ref{prop:Non-Explosion&Extinction}.

\begin{remark}\label{BasicProperties}
	One advantage of the graphical construction is that it provides a joint coupling of the processes with different infection rules or different initial states. The CPDG is \textit{monotone increasing} with respect to the initial conditions, i.e.
	\[\bfC_0 \subset \bfC_0' \quad \text{and}\quad \bfB_0 \subset \bfB_0' \qquad \Rightarrow  \qquad \bfC_t\subset \bfC_t'\quad  \text{and}\quad \bfB_t \subset \bfB_t', \quad \text{for all}\quad t\geq 0, \]
	where  $(\bfC, \bfB)$ and $(\bfC',\bfB')$ are CPDG. Further, the process is also monotone increasing with respect to the infection rate $\lambda$.
\end{remark}

\subsection{Properties of $(\bfC,\bfB)$ on $\mathcal{T}$}

Let $\mathcal{T}\sim$ BGW($\zeta$) and 
recall that we denote by $\mathbb{P}_{\mathcal{T}}(\cdot)=\mathbb{P}(\ \cdot \ | \ \mathcal{T}) $ the law of the CPDG on the tree $\mathcal{T}$.
\begin{lemma}\label{lemma:constantscritical}
	The critical values $ \lambda_1(\mathcal{T})$ and $ \lambda_2(\mathcal{T})$ are constant $\mathbb{P}$-almost surely conditioned on $|\mathcal{T}|=\infty$. In particular it holds that
	\begin{align*}
		\begin{aligned}
			\lambda_1(\cT)&=\lambda_1 := \inf\big\{\lambda>0: \IP^{\{\rho\}}(\bfC_t\neq \emptyset \,\forall\, t\geq 0)>0 \big\} \\
			\lambda_2(\cT)&=\lambda_2:= \inf\big\{\lambda>0:   
			\IP^{\{\rho\}}(\forall s \geq 0 \ \exists t \geq s \,  : \, \rho\in \bfC_t)>0\big\}
		\end{aligned}
	\end{align*}
	almost surely on the event $\{|\cT|=\infty\}$.
\end{lemma}

The proof follows similar ideas as those used first in \cite[Proposition 3.1]{pemantle1992} and  afterwards in \cite[Proposition 1]{huang2020}.

\begin{proof}
	First we deal with the case of~$ \lambda_1(\mathcal{T})$. Denote by~$\theta_{\mathcal{T}}(\lambda, v, \{\rho\})$ the survival probability in a given BGW tree $\mathcal{T}$ with the root initially infected, i.e.
	\begin{equation*}\label{eq:survivalproba}
		\theta_{\mathcal{T}}(\lambda,v, \{\rho\}):=\IP^{\{\rho\}}_{\mathcal{T}}\left(\bfC_t\neq\emptyset,\,     \forall t\geq0\right),
	\end{equation*}
	where we are using the notation $\IP_{\cT}^{\{\rho\}}(\cdot):= \IP(\ \cdot \ | \ \mathcal{T}, \bfC_0=\{\rho\})$.
	We shall prove that 
	\begin{equation}
		\label{eq:extinctionclaim}
		\mathbb{P}\big(\theta_{\mathcal{T}}(\lambda, v, \{\rho\})=0\big) \in \{\mathbb{P}(|\mathcal{T}|<\infty),1\},
	\end{equation}
	which implies that 
	\begin{equation}\label{eq:cond_zero_one} \mathbb{P}\big(\theta_{\mathcal{T}}(\lambda, v, \{\rho\})=0 \mid |\mathcal{T}|= \infty\big)\in \{0, 1\}\end{equation}
	because if $|\mathcal{T}|<\infty$ then $\theta_{\mathcal{T}}(\lambda, v, \{\rho\})=0$, and so 
	$\mathbb{P}\big(\theta_{\mathcal{T}}(\lambda, v, \{\rho\})=0, |\mathcal{T}|< \infty\big)=\mathbb{P}\big( |\mathcal{T}|< \infty\big)$.
	Now, assume that the root $\rho$ has $k$ children, namely $x_1, \dots, x_k$. Then the subtrees rooted in these children are independent identically distributed BGW trees with again offspring distribution $\zeta$, which we denote by $\mathcal{T}_1, \dots, \mathcal{T}_k$. 
	Let $b_{\rho,x_i}$ be the probability that $\rho$ infects $x_i$ conditional on $\{\rho,x_i\}\notin \bfB_0$. 	Note that by monotonicity 
	\begin{equation*}
		\theta_{\mathcal{T}}(\lambda, v, \{\rho\}) \geq    b_{\rho, x_i} \theta_{\mathcal{T}_i}(\lambda, v, \{x_i\}).
	\end{equation*}
	As $b_{\rho,x_i}> 0$, it follows that if~$\theta_{\mathcal{T}_i}(\lambda, v, \{x_i\})>0$  for some $i\in \{1,\dots,k\}$, then~$ \theta_{\mathcal{T}}(\lambda, v, \{\rho\}) >0$. In other words, if the contact process dies out on $\mathcal{T}$ then it has to die out on all subtrees $\mathcal{T}_1, \dots, \mathcal{T}_k$. Denote by $g$ the generating function of the random variable $\zeta$. Now we observe that
	\begin{equation*}
		\begin{aligned}
			\mathbb{P}\big(\theta_{\mathcal{T}}(\lambda, v, \{\rho\})=0\big) &= \sum_{k=1}^\infty   \mathbb{P}\big(\theta_{\mathcal{T}}(\lambda, v, \{\rho\})=0\ |\  \zeta = k\big) \mathbb{P}(\zeta =k )\\ & \leq  \sum_{k=1}^\infty   \mathbb{P}\Big(\bigcap_{i=1}^k \{\theta_{\mathcal{T}_i}(\lambda, v, \{x_i\})=0\} \Big) \mathbb{P}(\zeta =k ) \\ & =  \sum_{k=1}^\infty   \Big(\mathbb{P}\big( \theta_{\mathcal{T}}(\lambda, v, \{\rho\})=0\big)\Big)^k \mathbb{P}(\zeta =k )\\ & = g\Big(\mathbb{P}\big( \theta_{\mathcal{T}}(\lambda, v, \{\rho\})=0\big)\Big),
		\end{aligned} 
	\end{equation*}
	where in the second equality we have used that the subtrees $\mathcal{T}_1,\dots, \mathcal{T}_k$ are independent and identically distributed as $\mathcal{T}$. This means that 
	\[\mathbb{P}\big( \theta_{\mathcal{T}}(\lambda, v, \{\rho\})=0\big)\in \{s\in [0,1]: s \leq g(s)\}.\]
	But we know that  $g$ is a convex function on $[0,1]$ with at most two solutions of $g(s)=s$ on $[0,1]$, namely $s=1$ and $s=\mathbb{P}(|\mathcal{T}|<\infty).$ If $\mathbb{P}(|\mathcal{T}|<\infty)<1$ then 
	by convexity  for $s \in (0,1)$  we have $s \leq g(s)$ iff $s\leq \mathbb{P}(|\mathcal{T}|<\infty).$ That means that either~$  \mathbb{P}\big(\theta_{\mathcal{T}}(\lambda, v, \{\rho\})=0\big)=1$ or~$  \mathbb{P}\big(\theta_{\mathcal{T}}(\lambda, v, \{\rho\})=0\big)\leq \mathbb{P}(|\mathcal{T}|<\infty).$ 
	Further, note that as mentioned previously we have 
	\begin{equation*}
		\mathbb{P}(|\mathcal{T}|<\infty) \leq \mathbb{P}\big(\theta_{\mathcal{T}}(\lambda, v, \{\rho\})=0\big),
	\end{equation*}
	which shows the claim \eqref{eq:extinctionclaim}. This proves that $\lambda_1(\cT)$ is constant on the event $\{|\cT|=\infty\}$.

	Let us now prove the statement 	for $\lambda_2(\mathcal{T})$. 	Denote by 
	\begin{equation*}
		\beta_{\mathcal{T}}(\lambda, v, \{\rho\}):=\IP_{\cT}^{\{\rho\}} (\forall s \geq 0 \, \exists t \geq s \, : \, \rho\in \bfC_t).
	\end{equation*} 
	The same reasoning as in the argument for $\lambda_1(\mathcal{T})$ applies, with the only difference 
	that we need to argue that  conditionally on $\mathcal{T}$, if \ $\beta_{\mathcal{T}_i}(\lambda, v, \{x_i\})>0$ for some $i\in \{1,\dots,k\}$, then $\beta_{\mathcal{T}}(\lambda, v, \{\rho\})>0$. First of all, $\rho$ infects $x_i$ in finite time with positive probability. Then, by assumption (with positive probability)
	$x_i$ is infected at times $t_j \rightarrow \infty$ as $j\to \infty$ in the process restricted to $\mathcal{T}_i$. By moving to a subsequence (and the Borel-Cantelli lemma), we may assume that $|t_{j+1}-t_j| \geq 1$ and also that $x_i$ remains infected during $[t_j,t_j+1]$. Then, let $E_j$ be the event that there is an infection event in $[t_j,t_{j}+1]$ along the edge $\{x_i,\rho\}$ in the graphical construction (restricted to this edge) when the edge is open. Conditionally on the process restricted to $\mathcal{T}_i$, the events $E_j$ are independent and have positive probability. Thus, the conditional probability that infinitely many  of the events $E_j$ happen is equal to $1$ by the Borel-Cantelli lemma. 
	If this happens then $\rho \in \bfC_{s_i}$ for some $s_i \rightarrow \infty$. Therefore, $\beta_{\mathcal{T}}(\lambda, v, \{\rho\})>0$.     
	
	Note that as the CPDG dies out almost surely on a finite tree, we have that 
	\[ \begin{aligned}
		\IP^{\{\rho\}}(\bfC_t\neq \emptyset \,\forall\, t\geq 0) 
		&=\IE\big[\IP^{\{\rho\}}_{\cT}(\bfC_t\neq \emptyset \,\forall\, t\geq 0)\big]
		=\IE\big[\IP^{\{\rho\}}_{\cT}(\bfC_t\neq \emptyset \,\forall\, t\geq 0)\1_{\{|\cT|=\infty\}}\big],
	\end{aligned}\]
	and similarly
	\[ \begin{aligned}
		\IP^{\{\rho\}}(\forall s \geq 0 \, : \, \exists t \geq s \, : \, \rho\in \bfC_t) 
		&=\IE\big[\IP^{\{\rho\}}_{\cT}(\forall s \geq 0 \, : \, \exists t \geq s \, : \, \rho\in \bfC_t)\1_{\{|\cT|=\infty\}}\big].
	\end{aligned} \]
	Thus, by~\eqref{eq:cond_zero_one}  (and the analogous expression for $\beta_{\mathcal{T}}$)
	\begin{align*}
		\IP^{\{\rho\}}(\bfC_t\neq \emptyset \,\forall\, t\geq 0)>0  
		\quad &\Leftrightarrow \quad
		\IP^{\{\rho\}}_{\cT}(\bfC_t\neq \emptyset \,\forall\, t\geq 0)>0  \,\, \text{almost surely on }  \,\, \{|\cT|=\infty\}, \\
		\IP^{\{\rho\}}(\forall s \geq 0 \, : \, \exists t \geq s \, : \, \rho\in \bfC_t)>0 \quad &\Leftrightarrow \quad 
		\IP^{\{\rho\}}_{\cT}(\forall s \geq 0 \, : \, \exists t \geq s \, : \, \rho\in \bfC_t)>0 \,\, \text{almost surely on } \{|\cT|=\infty\}.
	\end{align*}
	This proves the statement of the proposition.
\end{proof}

\section{Proofs in the general graph setting}\label{sec:thm1}
In this section we prove Theorem \ref{Thm:SubCriticalPhase}, Corollary \ref{thm:comparisonpenal}, and Propositions \ref{thm:comparisonpenal} and \ref{cor:SurvivalByComparison}. In order to show Theorem \ref{Thm:SubCriticalPhase},
we will adapt the technique used in the proof of \cite[Theorem~5]{jacob2022contact}. Note that \cite{jacob2022contact} always work on finite graphs while we are dealing with an infinite graph. Thus, we need to adjust the proof by using an approximation argument via finite subgraphs. Let $\bfX=(\bfC^\bfX, \bfB^\bfX)= (\bfC_t^\bfX, \bfB_t^\bfX)_{t\ge 0}$  be a process with values in $\cP(\dV)\times \cP(\dE)$, where we again call a vertex infected or healthy if $x\in \bfC^\bfX$ or respectively $x\notin \bfC^\bfX $. Furthermore we call an edge $\{x,y\}$ revealed if $\{x,y\}\in \bfB^\bfX$ and otherwise unrevealed. This auxiliary process obeys the following rules:
\begin{itemize}
	\item[(i)] If an infected  vertex $x$ shares an unrevealed edge with a vertex $y$ then
	the edge is revealed at rate $\lambda p(d_x,d_y)$ and $y$ is infected if it was previously uninfected. 
	\item[(ii)] If an infected  vertex $x$ shares a revealed edge with an uninfected vertex $y$  then $x$ infects $y$ at rate $\lambda.$
	\item[(iii)] If $\{x,y\}$ is revealed it becomes unrevealed at rate $v(d_x,d_y).$
	\item[(iv)] Any infected vertex recovers at rate $1$.
\end{itemize}
We have the following coupling result that is taken from the proof of \cite[Theorem~5]{jacob2022contact} but has its roots in  \cite[Proposition~6.1]{jacob2017contact} and \cite{linker2020contact}.
Since the coupled process is an upper bound it can be used to show extinction for the coupled contact process. 

\begin{lemma}
	\label{lem:coupling}
	Let $\bfC_0^\bfX=\bfC_0$ and $\bfB_0^\bfX=\emptyset$. Then there exists a coupling such that  $\bfC_t\subset \bfC_t^\bfX$ for all $t\geq 0$. In particular, if there exists a $t\geq 0$ such that $\bfC_t^\bfX=\emptyset$ this implies that $\bfC_t=\emptyset$. 
\end{lemma}

\begin{proof}
	We will define a process $(\bfC^\bfX, \bfB^\bfX)$ that is coupled to the processes $(\bfC, \bfB)$ via the graphical representation using the same Poisson point processes $\Delta^{\inf}$, $\Delta^{\text{rec}}$, $\Delta^{\text{op}}$, $\Delta^{\text{cl}}$, as well as some additional independent randomness, and then show that the process has the correct dynamics. In order to prepare for the coupling we first define a thinned version of $\Delta^{\inf}$ that will provide us with the needed additional independent randomness. For each $e=\{x,y\}$ we use an independent sequence of i.i.d.\ $\text{Ber}(p(d_x, d_y))$ random variables to thin $\Delta^{\inf}_e$ so that we obtain a Poisson point process $\Delta^{\inf,p}_e$ with rate $\lambda p(d_x, d_y).$
	
	The evolution of $\bfC_t^\bfX$ is defined as follows: If $e \in \bfB_{t-}^\bfX$ such that $e$ is a revealed edge right before time $t$ then a potential infection event $t \in \Delta^{\inf}_e$ can be used for infection.	If $e \notin \bfB_{t-}^\bfX$ such that $e$ is an unrevealed edge right before time $t$ then at a potential infection event $t \in \Delta^{\inf}_e$ 
	the edge $e=\{x,y\}$ can be revealed (and infection can occur along the edge) if  $x \in \bfC_{t-}^{\bfX}$ or $y \in \bfC_{t-}^\bfX$. Whether this revealment and infection actually occurs (otherwise nothing happens) is decided by independent $\text{Ber}(p(d_x, d_y))$ random variables. As a consequence the overall rate of revealment and infection in this setting is then $\lambda p(d_x, d_y).$ Generally, in order to assure that $\bfC_t \subset \bfC_t^\bfX$ for all $t \geq 0$ we would like to use $\1_{\{e \in \bfB_{t-}\}}$, thus revealment and infection happens if the edge $e$ is open. However, if the edge is closed, and there is no revealment then at the next potential revealment (if it occurs before the next update event) we cannot use the same Bernoulli random variable but need another {independent} Bernoulli random variable. 
	In this case we use the independent Bernoulli random variables that defined the thinning of $\Delta^{\inf}_e$ (formally we use $\1_{\{t \in \Delta^{\inf,p}_e\}}$). This then assures independence of the $\text{Ber}(p(d_x, d_y))$ random variables. Notice that using this independent randomness to decide whether revealment and infection happens does not disturb the monotone coupling as we can in this situation only have additional infections for $\bfC_t^\bfX$ since the edge is closed and no infections occur for $\bfC_t$ until the next update of the edge. (It is in fact only due to such revealment and infection events that we may obtain more infected sites in $\bfC^\bfX$ than in $\bfC$.)
	Furthermore, an edge will be unrevealed whenever it is updated, so at times $t \in \Delta^{\text{op}} \cup \Delta^{\text{cl}}$.
	Also, the recovery events in $\Delta^{\text{rec}}$ are used by $\bfC$ and $\bfC^\bfX$ in the same way.

	From this construction, it is clear that we have $\bfC_t\subset \bfC_t^\bfX$ since the initial states are identical and all infection paths that are used by the process $\bfC$ can also be used by $\bfC^\bfX$. For this note that at any time $t \in \Delta^{\inf}$ an edge that is unrevealed but open at $t-$ will be revealed at time $t$ whenever it is used by an infection path for $\bfC$ and that we then also have  an infection path for $\bfC^\bfX$. Also, revealed edges (whether open or closed) can at any $t \in \Delta^{\inf}$ be used by an infection path for $\bfC^\bfX$. Lastly, both processes use the same recovery events.
	
	Thus, it just remains to verify that the process $(\bfC^\bfX, \bfB^\bfX)$ we defined has the dynamics that was described in (i) to (iv) before the statement of Lemma \ref{lem:coupling}. We have already argued in the construction above that revealment happens for edges $e=\{x,y\}$ at rate $\lambda p(d_x, d_y)$ while  $x \in \bfC^\bfX$ or $y \in \bfC^\bfX$. The points (ii)-(iv) follow immediately: For (ii) we simply observe that the rate $\lambda$ is a consequence of the fact that $\Delta^{\inf}$ is for every edge constructed from a Poisson process at rate $\lambda$. Likewise, (iii) follows since update events  in $\Delta^{\text{op}} \cup \Delta^{\text{cl}}$ for any edge coincide with unrevealment events and happen at rate $v(d_x, d_y)$. Finally, recoveries use $\Delta^{\text{rec}}$, which comes from a rate~1 Poisson process at every vertex, and this implies (iv). 
\end{proof}

\noindent
Next let us assume that condition \eqref{eq:vbound0} holds and set for every $x\in \dV$ and $X:=(C^X, B^X) \in \cP(\dV) \times \cP(\dE)$,
\begin{equation*}
	R_x(X):=\sum_{\{x,y\}\in B^X}\frac{\lambda}{v(d_x,d_y)} \quad \text{ and } \quad Q_x(X):=\sum_{\{x,y\}\in B^X}\frac{\lambda}{v(d
		_x,d_y)^2}.
\end{equation*}
In words $R_x(X)$ and $Q_x(X)$ are weighted sums of edges around $x$ which are contained in the configuration $X$. Note that due to $v(d_x, d_y)\geq \underline v$ we have $R_x(X) \geq \underline{v} Q_x(X).$
We define for every $x$ a  score 
\begin{equation*}
	H_x(X):=\begin{cases}
		1+2Q_x(X)& \text{ if } x\in C^X\\
		R_x(X) +2Q_x(X)& \text{ if } x\notin C^X.
	\end{cases}
\end{equation*}
We also define 
\begin{equation*}
	S_x(X):=|\{\{x,y\}\in B^X\}|,
\end{equation*}
where $S_x(X)$ is the number of neighbors which share a revealed edge with $x$.  Next, for a given function $W:\IN\to [1,\infty)$ we define 
\begin{equation}\label{function_f}
	f(X):=\sum_{x\in \dV}W(d_x)H_x(X),\quad X\subset \dV\times \dE. 
\end{equation}
In order to be able to apply the generator to the function $f$, we define the following approximating process. For $n \in \mathbb{N}$, set $\dA_n:=\dV_n \times \dE_n$	with 		
\begin{equation*}
	\dV_n:=\{x\in \dV: d(x, \rho)\leq n\} \quad \text{ and } \quad \dE_n:=\{\{x,y\}\in \dE: d(x,\rho)\leq n \  \text{ or } \  d(y, \rho)\leq n\},
\end{equation*}
where $d(\cdot, \cdot)$ is the graph distance between two vertices in $\dG$. We consider a new process $\bfX^n :=(\bfC^{\bfX^n}, \bfB^{\bfX^n}), $ with state space $\cP(V)\times\cP(E)$, where  we declare all vertices in $V \setminus V_n$ to be permanently uninfected and all edges in $E \setminus E_n$ to be permanently non-revealed.
For all other vertices and edges we use same graphical construction as in the original process to describe the infection, recovery and revealment events. Now we set $M^n= (M_t^n)_{t\ge 0}$ where
\begin{equation}\label{eq:Mnt}
	M^n_t:=f(\bfX_t^n).
\end{equation}
Note that since all edges in $E \setminus E_n$ are unrevealed, the corresponding edges do not contribute to the sum defining $f$ and so the sum is finite.

Our goal is to show that $M^n$  is a supermartingale for every $n \in \IN$. For this
let us denote by $(T^n_t)_{t\geq0}$ Markov semigroup and by $\gen_n$ the generator corresponding  to $\bfX^n$. 	
\begin{lemma}\label{Lem:GeneratorBound}
	Let $f$ be defined as in \eqref{function_f}.
	Suppose that $v$ satisfies condition \eqref{eq:vbound0} and that for $W$ condition \eqref{Eq:SupermartigaleCondition0} of Theorem~\ref{Thm:SubCriticalPhase} holds. 
	Then for $X\subset \dA_n$ we have
	\begin{equation*}
		\begin{aligned}
			\gen_n f(X)\leq&\sum_{x\in \dV_n}W(d_x)\Big(\lambda K \Big( 1+\frac{2\lambda}{\underline{v}^2}\Big)+\sum_{\{x,y\}\notin B^X} \frac{4\lambda^2}{v(d_x,d_y)^2}p(d_x,d_y)-\Big(\frac{\underline{v}}{2}\wedge 1\Big) \Big)H_x(X).
		\end{aligned}
	\end{equation*}
\end{lemma}
\begin{proof}
	Suppose that the current state of $\mathbf{X}^n_t$ is $X\subset \dA_n$. We first list all the possible infinitesimal changes of~$H_x(X)$ for $x\in \dV_n$.
	\begin{enumerate}
		\item Let us first assume that the vertex $x$ is infected, i.e.\ $x\in C^X$. Then the following changes  can occur:
		\begin{enumerate}
			\item Either $x$ attempts to infect one of its neighbours $y$ and $\{x,y\}$ changes from being unrevealed to revealed or one of the infected neighbours $y$ of $x$ attempts to infect $x$ along a previously unrevealed edge that now becomes revealed.   
			Each of these events causes a change of $\tfrac{2\lambda}{v(d_x,d_y)^2}$, so that the cumulative change is
			\begin{equation*}
				\sum_{y \in V_n \, : \, \{x,y\}\notin B^X} \lambda p(d_x,d_y) \frac{2\lambda}{v(d_x,d_y)^2} + \sum_{y \in C^X \, : \, \{x,y\}\notin B^X} \lambda p(d_x,d_y) \frac{2\lambda}{v(d_x,d_y)^2}.
			\end{equation*}
			\item The cumulative change caused by update events at revealed edges between $x$ and its neighbours is
			\begin{equation*}
				\sum_{\{x,y\}\in B^X} v(d_x,d_y) \frac{-2\lambda}{v(d_x,d_y)^2}=-2R_x(X).
			\end{equation*}
			\item A recovery event at $x$ results in the change $R_x(X)-1$.
		\end{enumerate}
		\item Next we consider $x$ to be healthy, i.e.~$x\notin C^X$.
		\begin{enumerate}
			\item The cumulative change caused by infection events along unrevealed edges between $x$ and its neighbours is
			\begin{equation*}
				\sum_{y\in C^X,\{x,y\}\notin B^X} \lambda p(d_x,d_y) \Big(1-R_x(X)+\frac{2\lambda}{v(d_x,d_y)^2}\Big).
			\end{equation*}
			\item If the infection event between $x$ and its neighbours happens at a revealed edge the cumulative change is
			\begin{equation*}
				\sum_{y\in C^X,\{x,y\}\in B^X} \lambda \big(1-R_x(X)\big)\leq \lambda S_x(X).
			\end{equation*}
			\item The cumulative change of all update events at revealed edges between $x$ and its neighbours is
			\begin{equation*}
				\sum_{\{x,y\}\in B^X} -v(d_x,d_y)\Big(\frac{\lambda}{v(d_x,d_y)}  +  \frac{2\lambda}{v(d_x,d_y)^2}\Big)=-\lambda S_x(X)-2R_x(X).
			\end{equation*}
		\end{enumerate}
	\end{enumerate}
	Now by adding up the individual contributions we listed before we get that
	\begin{equation*}
		\begin{aligned}
			\gen_nf(X)=&\sum_{x\in C^X} W(d_x) \bigg(\sum_{\{x,y\}\notin B^X}  \frac{2\lambda^2 p(d_x,d_y)}{v(d_x,d_y)^2}+\sum_{y\in C^X:\{x,y\}\notin B^X}  \frac{2\lambda^2 p(d_x,d_y)}{v(d_x,d_y)^2}-2R_x(X)+\big(R_x(X)-1\big)\bigg)\\
			&+\sum_{x\notin C^X}W(d_x)\bigg( \sum_{y \in C^X,\{x,y\}\in B^X} \lambda \big(1-R_x(X)\big)-\lambda S_x(X)\\
			&\phantom{XXXXXXXX}-2R_x(X)+\sum_{y \in C^X,\{x,y\}\notin B^X} \lambda p(d_x,d_y) \bigg(1-R_x(X)+\frac{2\lambda}{v(d_x,d_y)^2}\bigg)\bigg).
		\end{aligned}
	\end{equation*}
	Note that the sum over $x\notin C^X$ is only over $x\in \dV_n$ which are not infected. 
	We can simplify and ignore the negative contributions from $\big(\sum_{y \in C^X,\{x,y\}\in B^X} \lambda \big(1-R_x(X)\big)\big)-\lambda S_x(X)$ in the second and $-R_x(X)$ in the very last sum in order to obtain
	\begin{equation}
		\begin{aligned}\label{Eq:GeneratorBound}
			\gen_nf(X)\leq&\sum_{x\in C^X} W(d_x) \bigg(\sum_{\{x,y\}\notin B^X} \frac{4\lambda^2}{v(d_x,d_y)^2} p(d_x,d_y) -1-R_x(X)\bigg)\\
			&+\sum_{x\notin C^X} W(d_x) \bigg( \bigg( \sum_{y \in C^X,\{x,y\}\notin B^X} \lambda p(d_x,d_y) \Big(1+\frac{2\lambda}{v(d_x,d_y)^2}\Big)\bigg)-2R_x(X)\bigg).
		\end{aligned}
	\end{equation}
	Now the right hand side of \eqref{Eq:GeneratorBound} can be split in four terms. Let us first handle the terms with negative sign. Since $v(d_x,d_y)\geq \underline{v}$ for all $\{x,y\} \in \dE$ and so $R_x(X)\geq \underline{v} Q_x(X)$ it follows that
	\begin{equation}\label{Eq:SeperateTerms1}
		\begin{aligned}
			&-\sum_{x\in C^X}W(d_x)\Big(1+R_x(X)\Big)\leq -\Big(\frac{\underline{v}}{2}\wedge 1\Big)\sum_{x\in C^X}W(d_x)\big(1+2Q_x(X)\big),\\
			&-\sum_{x\notin C^X} W(d_x)\Big(R_x(X)+R_x(X)\Big)\leq -\Big(\frac{\underline{v}}{2}\wedge 1\Big)\sum_{x\notin C^X} W(d_x)\big(R_x(X)+2Q_x(X)\big).
		\end{aligned}
	\end{equation}
	Finally we deal with the remaining term on the second line of \eqref{Eq:GeneratorBound} and change the order of summation so that we get 
	\begin{equation*}
		\begin{aligned}
			\sum_{y \in C^X} \sum_{\{x,y\}\notin B^X,x\notin B^X}\lambda \Big(1+\frac{2\lambda}{v(d_x,d_y)^2}\Big)W(d_x) p(d_x,d_y)
			\leq \sum_{y \in C^X}\lambda K\Big(1+\frac{2\lambda}{\underline{v}^2}\Big)W(d_y),
		\end{aligned}
	\end{equation*}
	where for the inequality we have once again used $v(d_x, d_y) \geq \underline{v}$ as well as assumption \eqref{Eq:SupermartigaleCondition0} on the function $W$.
	With this fact and \eqref{Eq:SeperateTerms1}  
	we get by plugging back into \eqref{Eq:GeneratorBound} that
	\begin{equation*}
		\begin{aligned}
			\gen_nf(X)\leq&\sum_{x\in C^X}  \lambda K\Big(1+\frac{2\lambda}{\underline{v}^2}\Big)W(d_x)
			+\sum_{x\in C^X}W(d_x)\sum_{\{x,y\}\notin B^X}\frac{4\lambda^2}{v(d_x,d_y)^2} p(d_x,d_y)
			\\
			&-\Big(\frac{\underline{v}}{2}\wedge 1\Big)\Big(\sum_{x\in C^X}W(d_x)\big(1+2Q_x(X)\big)+\sum_{x\notin C^X} W(d_x)\big(R_x(X)
			+2Q_x(X)\big)\Big).
		\end{aligned}
	\end{equation*}
	Now by adding in some positive terms (involving $x \notin C^X$), we obtain that
	\begin{equation*}
		\begin{aligned}
			\gen_n f(X)\leq&\sum_{x\in \dV_n}W(d_x)\Big(\lambda K \Big( 1+\frac{2\lambda}{\underline{v}^2}\Big)+\sum_{\{x,y\}\notin B^X} \frac{4\lambda^2}{v(d_x,d_y)^2}p(d_x,d_y)-\Big(\frac{\underline{v}}{2}\wedge 1\Big) \Big)H_x(X),
		\end{aligned}
	\end{equation*}
	where we have also used that $H_x(X)\ge 1$ if $x\in C^X$. 
\end{proof}

We can now use Markov semigroup theory which yields that $\partial_t T^n_t=\gen_nT^n_tf=T^n_t\gen_nf$ and use the bound on the derivative found in Lemma~\ref{Lem:GeneratorBound} to find a suitable bound on $T^n_t$. Then 
we take the limit as $n\rightarrow \infty$ to deduce that 
$\bfC$ goes extinct.

For this purpose, we define an approximation process analogously to  $\bfX^n$. We define $(\bfC^n,\bfB^n)$ by restricting the graphical representation introduced in Section~\ref{sec:GraphRep} to the finite subgraph $G_n=(V_n,E_n)$, where we declare all vertices in $V \setminus V_n$ to be permanently uninfected and all edges in $E \setminus E_n$ to be permanently closed. By definition it follows that the sequences of (random) sets $(\bfC_t^n,\bfB_t^n)$ are monotone in $n$, i.e.\ $(\bfC_t^n,\bfB_t^n)\subset (\bfC_t^{n+1},\bfB_t^{n+1})$ for all $n\geq 0$, which implies existences of the pointwise limit and also that 
\begin{equation}\label{eq:defC}
	\lim_{n\to \infty} \bfC_t^n=\bigcup_{n\in\IN}\bfC_t^n \quad \text{and} \quad  \lim_{n\to \infty} \bfB_t^n=\bigcup_{n\in\IN}\bfB_t^n\qquad  \text{for every}\quad t\ge 0.
\end{equation}
Now we only need to ensure that $(\bfC_t,\bfB_t)=(\lim_{n\to \infty} \bfC_t^n,\lim_{n\to \infty} \bfB_t^n)$ holds true for all $t\geq 0$. It is clear $\bfB_t=\lim_{n\to \infty} \bfB_t^n$, since edge updates happen independently of each other. In the case of the infection process the equality could only fail if there exists an $x\in \bfC_t$ such that $x\notin\lim_{n\to\infty} \bfC_t^n$. But in the definition of $\bfC_t$ we explicitly assumed that infection paths are of finite length, and thus there must exist an $n\in \IN$ such that $x\in \bfC_t^n$, which implies that $x\in\lim_{n\to\infty} \bfC_t^n$.

\begin{proposition}\label{prop:Non-Explosion&Extinction}
	Let $(\bfC_0^\bfX, \bfB_0^\bfX)= X$ such that $X=(C^X, B^X)$ and $C^X$ is a finite set. Suppose  that condition \eqref{eq:vbound0} and that \eqref{Eq:SupermartigaleCondition0} and \eqref{Eq:SupermartigaleCondition1} of Theorem \ref{Thm:SubCriticalPhase} hold. Then it follows that
	\begin{equation*}
		\IP(|\bfC_t|<\infty\ \,\forall\, t\geq 0 )=1
	\end{equation*}
	for every infection rate $\lambda>0$. Assume further that 
	\begin{equation}\label{Eq:SupermartigaleCondtion3}
		\vartheta:=\lambda K\Big(1+\frac{2\lambda}{\underline{v}^2}\Big)+4\lambda^2 K
		%\frac{2\lambda^2}{\underline{v}^2} K %\widetilde{K}
		-\Big(\frac{\underline{v}}{2}\wedge 1\Big)<0.
	\end{equation}
	Then it follows that the process $M^n=(M_t^n)_{t\ge 0}$ defined as in \eqref{function_f} is a supermartingale for every $n \in \IN$, which converges almost surely to $0$ as $t\to \infty$. This implies in particular that $\bfC$ goes extinct almost surely. 
\end{proposition}
\begin{proof}
	We first consider the processes $\bfX^n$ and $M^n=f(\bfX^n).$ By  Lemma~\ref{Lem:GeneratorBound} and \eqref{Eq:SupermartigaleCondition1} we get for $X\subset \dA_n$  that 
	\begin{equation*}
		\partial_t T^n_t f(X)= T^n_t \gen_n f(X)\leq T^n_t \Big(\lambda K \Big( 1+\frac{2\lambda}{\underline{v}^2}\Big)+4\lambda^2 K
		-\Big(\frac{\underline{v}}{2}\wedge 1\Big) \Big)
		f(X)=\vartheta T^n_tf(X).
	\end{equation*}
	Therefore, for $n$ sufficiently large such that $M_0^n = f(X)$,
	\begin{equation}\label{eq:Tn}
		T^n_tf(X)\leq e^{\vartheta t}T^n_0f(X)=e^{\vartheta t} f(X).
	\end{equation}
	Moreover, since $W\geq 1$ and $H_x(\bfX_t^n)\geq 1$ when $x\in \bfC_t^{\bfX^n}$, we obtain that $f(\bfX_t^n) \geq  |\bfC_t^{\bfX^n}|$. It follows by Markov's inequality that for any $L\in \mathbb{N}$
	\[ \mathbb{P} ( |\bfC_t^{\bfX^n}| \geq L) \leq \frac{T^n_t f(X)}{L}  \leq \frac{e^{\vartheta t}f(X)}{L}  . \]
	Further, using Lemma~\ref{lem:coupling} with the processes $(\bfC^n,\bfB^n)$ and $\bfX^n$, we see that these processes can be coupled such that $\bfC_t^n \subset \bfC_t^{\bfX^n}$ for all $n\in \mathbb{N}$ and $t\ge 0$, which together with the previous inequality implies that
	\[ \mathbb{P} ( |\bfC_t^n| \geq L) \leq \frac{e^{\vartheta t} f(X) }{L}. \]
	But we have that $\mathbb{P}(|\bfC_t| \geq L) = \lim_{n \rightarrow \infty}
	\mathbb{P}(|\bfC_t^n| \geq L)$ since
	$\{|\bfC_t| \geq L\} = \lim_{n \rightarrow \infty} \{|\bfC_t^n| \geq L\}.$ As for any $n$ we have $\bfC_t^n \subset \bfC_t$ it is clear that the right-hand side is smaller. In order to see the reverse inequality we note that if $|\bfC_t| \geq L$ holds then there exists an $x \in V$ such that $x \in \bfC_t$ and we will by the convergence $\bfC_t^n \uparrow \bfC_t$ also have that $x \in \bfC_t^n$ for large enough $n$, and thus $|\bfC_t^n| \geq L$. 
	
	This implies 
	\begin{equation}\label{eq:Cl}
		\mathbb{P}(|\bfC_t| \geq L) \leq \frac{e^{\vartheta t} f(X)}{L}.
	\end{equation}
	Now using that that $\{|\bfC_t|\geq L\}$ is decreasing in $L$ and continuity from below of the probability measures, we get that 
	\begin{equation*}
		\IP(|\bfC_t|=\infty)=\lim_{L\to \infty} \IP(|\bfC_t|\geq L)\leq \lim_{L\to \infty}\frac{e^{\vartheta t} f(X)}{L}=0.
	\end{equation*}
	Thus, it follows that $\IP(|\bfC_s|<\infty \,\forall \, s\leq t)=1$ for all $t\geq 0$ which in turn implies that 
	\begin{equation*}
		1=\lim_{t\to \infty}\IP(|\bfC_s|<\infty \,\forall \, s\leq t)=\IP\bigg(\bigcap_{t\geq0}\{|\bfC_s|<\infty \  \,\forall \, s\leq t\}\bigg)=\IP(|\bfC_t|<\infty \,\forall \, t\geq0).
	\end{equation*}    
	For the last statements, we assume that $\vartheta<0$. From \eqref{eq:Tn} and Markov's property it follows that for each $n\in \mathbb{N}$ the process $M^n$ is a supermartingale such that $M_t^n\to 0$ as $t\to\infty$ almost surely. In addition, by \eqref{eq:Cl} and the continuity of probability measures
	\[ \mathbb{P} (|\bfC_t| = 0  \mbox{ for some } t \geq 0) = \lim_{t \rightarrow \infty} \mathbb{P}(|\bfC_t| = 0) 
	\geq \lim_{t \rightarrow \infty} 1 - e^{\vartheta t}f(X) = 1, \]
	as $\vartheta < 0$.
\end{proof}

\begin{proof}[Proof of Theorem~\ref{Thm:SubCriticalPhase}]
	The claim of the theorem follows directly from Proposition \ref{prop:Non-Explosion&Extinction}. We just need to pick a $\lambda>0$ small enough so that \eqref{Eq:SupermartigaleCondtion3} is satisfied.
\end{proof}

\begin{proof}[Proof of Corollary~\ref{cor:Extinction}]
	Assume that $p$ and $v$ satisfy \eqref{eq:kernelsgen} and \eqref{eq:kernelsgen2} with $\eta, \alpha \geq 0$. Since $p$ is decreasing in both parameters and $d_x\ge 1$ for any $x\in V$, we have
	\[p(d_x, d_y) \le p(d_x, 1) \le \kappa_2 d_x^{-\alpha} \qquad \text{and}\qquad p(d_x, d_y) \le p(1, d_y) \le \kappa_2 d_y^{-\alpha}.\]
	Using the same arguments we also have the corresponding bounds for $v$. In other words, we deduce that 
	\begin{equation}\label{eq:BoundedByMaximum}
		p(d_x,d_y)\leq \kappa_2(d_x\vee d_y)^{-\alpha} \qquad \text{ and } \qquad v(d_x,d_y) \geq \nu_1(d_x\vee d_y)^{\eta}.
	\end{equation}
	With this in hand, we can check conditions \eqref{Eq:SupermartigaleCondition0} and \eqref{Eq:SupermartigaleCondition1} of Theorem~\ref{Thm:SubCriticalPhase}. We first prove the statement under condition $(i)$. Therefore, we assume $\alpha\geq 1$. Then we choose $W(d_x)=d_x$, then it follows that for all $x \in V$
	\begin{equation*}
		\sum_{\{x,y\}\in \dE}W(d_y)p(d_x,d_y)\leq \kappa_2 \sum_{\{x,y\}\in \dE}d_y (d_x\vee d_y)^{-\alpha}\leq \kappa_2 d_x= \kappa_2 W(d_x),
	\end{equation*}
	where in the second inequality we used that $d_y(d_x\vee d_y)^{-\alpha}\leq 1$ since  $\alpha\geq 1$. On the other hand, we observe that \eqref{eq:BoundedByMaximum} implies that for all $x \in V$
	\begin{equation*}
		\sum_{\{x,y\}\in \dE}\frac{p(d_x,d_y)}{v(d_x, d_y)^{2}} \leq \frac{\kappa_2}{\nu^2_1}\sum_{\{x,y\}\in \dE}(d_x\vee d_y)^{-(\alpha+2\eta)} \le \frac{\kappa_2}{\nu^2_1}\sum_{\{x,y\}\in \dE} d_x^{-(\alpha+2\eta)}\le \frac{\kappa_2}{\nu^2_1} d_x^{1-(\alpha+2\eta)}\le \frac{\kappa_2}{\nu^2_1},
	\end{equation*}
	where in the last inequality we used that $\alpha \geq 1$ and $\eta\ge 0$. Therefore,  
	the conditions of Theorem \ref{Thm:SubCriticalPhase} are satisfied, and thus $\lambda_1(G)>0$ and the infection process does not explode for any infection rate $\lambda>0$. 
	
	Now we consider the condition $(ii)$. Thus, we assume that the connection probability is of the form $p_{\alpha, \sigma}$ as defined in \eqref{eq:kernels}, where $\sigma\in(0,1]$, $\alpha\geq 0$  and $\kappa>0$, and $v$ still satisfies \eqref{eq:kernelsgen}. Note that it suffices to show the statement for $\alpha\in \big[(2\sigma)^{-1},1\big)$, since $\alpha\geq 1$ is already covered by the first case $(i)$.   
	Let  $W(d_x)=d_x^{\beta}$ for some $\beta\in[1-\alpha\sigma,\alpha\sigma]$, which is well-defined since $\alpha \sigma\geq 1/2$.
	Now, we observe  
	\begin{align*}
		d_y^{\beta} \big((d_x\wedge d_y)^{\sigma}(d_x\vee d_y)\big)^{-\alpha}    = d_y^{\beta-\alpha}d_x^{-\alpha\sigma}\1_{\{d_x\leq d_y\}}+d_y^{\beta-\alpha\sigma}d_x^{-\alpha}\1_{\{d_x> d_y\}} \le  d_x^{-\alpha \sigma},
	\end{align*}
	where in the inequality we used that $\beta - \sigma\alpha\leq 0$ and that $\beta -\alpha\le 0$, which follows since $\sigma\leq 1$. Hence, we see that for any $x \in V$
	\begin{eqnarray*}
		\sum_{\{x,y\}\in \dE}W(d_y)	p_{\alpha,\sigma}(d_x,d_y)&\leq& \kappa \sum_{\{x,y\}\in \dE}d_y^{\beta} \big((d_x\wedge d_y)^{\sigma}(d_x\vee d_y)\big)^{-\alpha}\\
		&\le&   \kappa \sum_{\{x,y\}\in \dE}d_x^{-\alpha \sigma}\leq  \kappa d_x^{1-\alpha\sigma}\leq  \kappa W(d_x),
	\end{eqnarray*}
	where we used that $1-\sigma\alpha\leq \beta$. For condition \eqref{Eq:SupermartigaleCondition1}, we observe that  
	\begin{equation*}
		\sum_{\{x,y\}\in \dE}\frac{p(d_x,d_y)}{v(d_x, d_y)^{2}} \leq \frac{\kappa}{\nu^2_1}\sum_{\{x,y\}\in \dE}(d_x\vee d_y)^{-(\alpha+2\eta)} (d_x \wedge d_y)^{-\sigma \alpha} \le \frac{\kappa}{\nu^2_1}d_x^{1-(\alpha+2\eta)} \le \frac{\kappa}{\nu^2_1},
	\end{equation*}
	where in the last inequality we used that $\alpha+2\eta\ge 1$. Therefore, again the conditions of Theorem \ref{Thm:SubCriticalPhase} are satisfied, which yields the claim.
\end{proof}

\begin{proof}[Proof of Proposition \ref{thm:comparisonpenal}]
	Let $v$ be such that \eqref{eq:vbound0} holds. Following similar arguments to those used in the proof of Theorem~2.2 in \cite{seiler2022long} (based on the coupling of~\cite{broman2007stochastic}), we can couple the process $(\bfC,\bfB)$ with a contact process $\underline{\bfX}$  such that $\underline{\bfX}_t\subset\bfC_t$ for all $t\ge 0$, with an infection rates $(a_{\{x,y\}})_{\{x,y\}\in \dE}$ defined as follows
	\begin{equation*}
		a_{\{x,y\}}(\lambda, v)
		:=\frac{1}{2}\Big(\lambda+v(d_x,d_y)-\sqrt{(\lambda+ v(d_x,d_y))^2- 4\lambda v(d_x,d_y) p(d_x,d_y)}\,\Big),
	\end{equation*}
	for all $\{x,y\} \in \dE$. Note that we explicitly indicate the dependence on $v(\cdot, \cdot)$ since in the second part of the proof we consider $\underline{v}\to\infty$. Denote by $\lambda_i^{\underline{\bfX}}(v,
	\dG)$ for $i=1,2$ the critical values for weak and strong survival of the process $\underline{\bfX}$ on $\dG$. Then by the coupling we immediately have $\lambda_i(v,\dG) \leq \lambda_i^{\underline{\bfX}}(v,
	\dG)$ for $i=1,2$. Note that by rearranging the terms we see that
	\begin{equation*}
		a_{\{x,y\}}(\lambda,v)
		=\frac{\lambda+v(d_x,d_y)}{2}\Big(1-\sqrt{1- 4 \lambda v(d_x,d_y) p(d_x,d_y)\big(\lambda+v(d_x,d_y)\big)^{-2}}\,\Big).
	\end{equation*}
	Now, using the inequalities $1-x\leq \sqrt{1-x}\leq 1-\frac{x}{2}$, which hold for all $x \in [0,1]$, we deduce that
	\begin{equation}\label{ratebounds}
		\frac{\lambda v(d_x,d_y) p(d_x,d_y)}{\lambda+v(d_x,d_y)} \leq  a_{\{x,y\}}(\lambda,v)
		\leq 2\frac{\lambda v(d_x,d_y)p(d_x,d_y)}{\lambda+v(d_x,d_y)}.
	\end{equation}
	Since we assume in \eqref{eq:vbound0} that $\underline{v}>0$ we obtain that 
	\begin{equation*}
		a_{\{x,y\}}(\lambda, v)
		\ge \lambda'(\underline{v}) p(d_x,d_y)\qquad \text{where} \qquad \lambda'(\underline{v}):=\frac{\lambda \underline{v}}{\lambda+\underline{v}}.
	\end{equation*}
	Denote by $\bfX'$ a penalised contact process with infection rate $\lambda'(\underline{v})$, i.e.\ the infection rate along an edge $\{x,y\}\in \dE$ is  $\lambda'(\underline{v})p(d_x,d_y)$. Now we can obviously couple $\underline{\bfX}$ to $\bfX'$ such that $\bfX'_t\subset \underline{\bfX}_t$ for all $t\ge 0$. In other words, $\lambda_i^{\underline{\bfX}}(v, \dG)\leq \lambda'_i(\underline{v},\dG)$ for $i\in\{1,2\}$, where $\lambda'_1(\underline{v}, \dG)$ and $\lambda'_2(\underline{v},\dG)$ are the critical values for weak and strong survival of $\textbf{X}'$.
	Now we assume that $\lambda^{p}_1=0$ (resp.\ $\lambda^{p}_2=0$), i.e.\ the penalised contact process with rates $\lambda p(d_x,d_y)$ survives weakly (resp.\ strongly) with positive probability for all $\lambda>0$. In particular, the process survives weakly for $\lambda = \lambda'(\underline{v})$, corresponding to the transition rates of $\bfX'$. Then $\lambda_1'(\underline{v}, \dG)=0$ (resp.\ $\lambda_2'(\underline{v}, \dG)=0$). This provides the first claim. 
	
	Now we consider a sequence $(v^n)_{n \in \IN}$ of update speeds. Then by using a Taylor expansion one can show that $a_{\{x,y\}}(\lambda,v^n)\nearrow \lambda p(d_x,d_y)$ as $\lim_{n\to \infty} \underline{v}^n=\infty$, see for details the proof of \cite[Lemma~5.9]{seiler2022long}. It turns out that $\lambda_i^{\underline{\bfX}}(v^n, \dG)\searrow \lambda^p_i$ as $\lim_{n\to \infty} \underline{v}^n=\infty$ for $i=1,2$, which can be shown exactly as in the proof of \cite[Corollary~2.4]{seiler2022long} (we point out that condition $(3)$ in \cite[Corollary~2.4]{seiler2022long} is only used to get the uniform control over $v$, but here it follows from \eqref{ratebounds} and the definition of $\underline{v}$). This yields the second claim, i.e.\ 
	\begin{equation*}
		\limsup_{n\to \infty} \lambda_i(v^n, G)\leq \limsup_{n\to \infty} \lambda_i^{\underline{\bfX}}(v^n, G)=\lambda^p_i.\qedhere
	\end{equation*}
\end{proof}

We finish this section with the proof of Proposition~\ref{cor:SurvivalByComparison}.
\begin{proof}[Proof of Proposition~\ref{cor:SurvivalByComparison}]
	In this proof we denote by $\lambda_2^{p}=\lambda_2^p(\kappa,\sigma)$ the critical infection rate for strong survival for the penalised contact process on the BGW tree with transitions as in \eqref{PenTransition} and with penalisation chosen to be $p_{\alpha,\sigma}$ as in \eqref{eq:kernels} with $\kappa>0$, $\sigma\in[0,1]$ and $\alpha$ such that 
	\eqref{eq:assumptionkernel3} holds. Note that for $\kappa =1$, $p_{\alpha,1}=(d_xd_y)^{-\alpha}$ and by \cite[Theorem~2.1]{zsolt} it follows that $\lambda_2^p(1,1)=0$. Now if $\kappa<1$, then  $p_{\alpha,1}=\kappa(d_xd_y)^{-\alpha}$, and thus it follows directly that $\lambda_2^{p}(\kappa,1)=\frac{1}{\kappa}\lambda_2^{p}(1,1)= 0$. Furthermore, we see that 
	$p_{\alpha,1}\leq p_{\alpha,\sigma}$ for all $\sigma\in[0,1]$, and this is also true if we replace $\kappa$ in $p_{\alpha,\sigma}$ with $\kappa'$ such that $\kappa'\geq \kappa$. Therefore, by monotonicity of the penalised contact process in the infection rates it follows that $0=\lambda_2^p(\kappa,1)\geq\lambda_2^p(\kappa',\sigma) \geq 0.$ This implies $\lambda_2^p(\kappa,\sigma) =0$ for all $\kappa>0$ and $\sigma\in[0,1]$. Finally, the claim that the critical infection rates of the CPDG with connection probability $p_{\alpha,\sigma}$ equal $0$, i.e.\ $\lambda_1=\lambda_2=0$  follows by Proposition~\ref{thm:comparisonpenal}.
\end{proof}

\section{Proofs in the tree setting}\label{sec:CPstars}
The goal of this section is to prove Theorem~\ref{thm:strongsurvival}. As described in Subsection~\ref{subsec:HeuristicEvolTree}, the proof strategy relies on two key arguments. The first argument is that the infection can survive in the neighbourhood of a vertex $x$ of exceptionally high degree $N$ for a time of exponential order. We call such vertices $x$ \textit{stars}. The second argument is that we will consider relatively heavy-tailed offspring distributions, allowing the infection to reach the next star of at least size $N$ fairly quickly. We present the technical auxiliary results in the subsequent Subsections \ref{sec:StableStar}-\ref{sec:PushingInfection} in detail. Finally, in Subsection~\ref{sec:thmstrong}, we bring them all together to prove Theorem~\ref{thm:strongsurvival}.

Throughout this section, in most cases, we consider the root $\rho\in \cV$ to be a star. Thus, for the sake of readability we introduce some notation. If we condition on $\de_\rho=N\in\IN$, where $N$ will be chosen to be sufficiently large, we write $\IP_{\cT, N}(\,\cdot\,)$ (resp.\ $\IP_{N}(\,\cdot\,)$) instead of $\IP_{\cT}(\,\cdot\mid \de_\rho=N)$ (resp.\ $\IP(\,\cdot\mid \de_\rho=N)$).

For the rest of this section, we fix an $L\in\IN$ such that $\mu_L:=\IE[\zeta\1_{\{\zeta < L\}}]>1$, where $\zeta$ is a given offspring distribution with $\IE[\zeta]>1$. The condition $\mu_L>1$ ensures that if we prune the random tree by only keeping vertices of degree less than $L$ from generation $1$ onward, then this random tree is still supercritical.

Due to Conditions~\eqref{eq:kernelsgen} and \eqref{eq:kernelsgen2}, for given 
$\alpha>0$ and $\eta\in \IR$, we find constants $\kappa_1,\kappa_2,\nu_1,\nu_2>0$ which only depend on $L$ such that
\begin{equation}\label{eq:KernelCondition}
	\kappa_1 N^{-\alpha} \leq p(N,m)\leq \kappa_2 N^{-\alpha}\ \qquad \text{and}\qquad 
	\nu_1 N^{\eta} \leq v(N,m)\leq \nu_2 N^{\eta} 
\end{equation}
for $N\ge L\ge m$. 
Throughout this section, we assume that $\alpha\in (0,1)$.

\subsection{Stars and stable stars}\label{sec:StableStar}

We will first focus on the first argument, which is to understand how long the infection can survive restricted to the neighbourhood of a star $x$ with sufficiently high probability. 
The results in the next two subsections are very similar to what is discussed in \cite[Section~3.3]{jacob2022contact}. Thus, we will omit some of the proofs, which can be proved analogously.

Since we consider a dynamical tree we need to find a space-time structure which can keep an infection alive for a time of exponential order. For technical reasons we need to only consider neighbours of $x$ which are of bounded degree. We denote by
\begin{equation*}
	\cN=\cN_{x}:=\{y \in \cV :\  y \in \cV \text{ is offspring of } x \text{ and } \de_{y} \leq L\},
\end{equation*}
the set of all offspring of $x$ with degree less than or equal to $L$. Next we define
\begin{equation*}
	T=T_N:=\frac{1}{1+\nu_2N^{\eta}}  \qquad \text{and} \qquad  J_k:=[(k-2)T,(k+2)T)\cap \IR_+\quad  \text{for all}\quad k\geq 0.
\end{equation*} 
Note that $|J_k|=4T$ for $k \geq 2$ and the overlap of consecutive intervals is of length $3T$. Recall that $\Delta^{\text{rec}}_y$ and $\Delta^{\text{up}}_{\{x,y\}}$ are the Poisson point processes describing the recovery times of $y$ and the updating times of the edge $\{x,y\}$, respectively. Now we define the set of \textit{good neighbours of $x$} in the interval $J_k$ as follows
\begin{equation*}
	\cG_{k}=\cG_{x,k}:=\{y\in \cN_x: \{x,y\}\in \bfB_{kT},\ (\Delta^{\text{rec}}_y\cup\Delta^{\text{up}}_{\{x,y\}})\cap J_k=\emptyset\}, \quad k\ge 0.
\end{equation*}
In words a good neighbour in $J_k$ is a neighbouring vertex $y$ of $x$ which itself has degree less than or equal to $L$ and is connected to $x$ at time $kT$. Additionally, the edge $\{x,y\}$ is not updated within $J_k$, and therefore $y$ is connected to $x$ for the whole time interval $J_k$. Furthermore, $y$ does not recover which implies that if $y$ is infected, it will remain infected until at least $(k+2)T$.

As described in Subsection~\ref{subsec:HeuristicEvolTree}, a star can be thought of as reservoir for the infection, where it can survive for an exceptionally long time without input from outside. Furthermore, during this period we want to push the infection to the next star. In the next result we show that with high probability we have enough good neighbours around the root $\rho$ for a sufficiently long time. This lemma can be proved analogously to \cite[Proposition~3.7]{jacob2022contact}. We will nevertheless give the whole proof for two reasons. First we adapted the statement slightly and secondly since we consider a different model compared to \cite{jacob2022contact} there are minor differences in the proof, and thus we would like to highlight this once. Before stating the result, we say that $x$ is a \textit{stable star} if the following event holds
\begin{equation}\label{eq:setS}
	\cS=\cS_{x,N}:=\{|\mathcal{G}_{x,k}|>c_L Np(N,L) \ \text{for all}\ k \leq e^{c_L Np(N,L)}\}.
\end{equation}
Since we only consider stable stars in the subsequent subsection we mostly refer to $x$ just as a star. 
\begin{lemma} 
	\label{lemma:boundgoodnb}
	There exists $c_L\in (0,1)$ independent of $N, p(\cdot,\cdot)$ and $\lambda$ such that
	for $N$ large enough
	\begin{equation*}
		\IP_{N}\big(\cS_{x,N}\big)
		\geq 1-e^{-c_L N p(N,L)}.
	\end{equation*}
\end{lemma}
\begin{proof} As $\rho$ is fixed, we write $\cG_{k}$ for $\cG_{\rho,k}$ throughout.
	Note that for $y\in \cN$ the random variable $|(\Delta^{\text{rec}}_y\cup\Delta^{\text{up}}_{\{\rho,y\}})\cap J_k|$ has a Poisson distribution with mean bounded by  $4T(1+\nu_2 N^\eta)=4$, where we used that we conditioned on $\de_{\rho}=N$, \eqref{eq:KernelCondition} and that by definition $1\le \de_y\leq  L$. Now, by definition of the set $\mathcal{G}_{k}$ and $p(\cdot, \cdot)$ being a decreasing function, we have for every $y\in \cV$ with $\{\rho,y\}\in\cE$ that
	\begin{equation*}
		\IP_{\cT,N}\big(y\in \cG_{k}\big)\geq p(N,\de_y)e^{-4T(1+v(N,\de_y))}\1_{\{\de_y\leq L\}}\geq p(N,L)e^{-4}\1_{\{\de_y\leq L\}},
	\end{equation*}
	where we also used that we condition on $\de_{\rho}=N$. Thus, we get that $\IP_{N}\big(y\in \cG_{k}\big)\geq \phi_L e^{-4}p(N,L),$ where $\phi_{L}:=\IP(\zeta_y\leq L-1)$, and we recall that $\zeta_y$ is the number of offspring of vertex $y$ and by \eqref{eq:Dzeta} we have $D_y= \zeta_y+1$. Since the number of offspring are sampled independently for every parent in a BGW tree, this implies that $|\cG_{k}|$ stochastically dominates a binomial random variable $B\sim \text{Bin}(N,e^{-4}\phi_{L}p(N,L)))$.
	By the Chernoff bound for a binomial distribution (see e.g. Theorem 2.21 in \cite{van2016random}), we get for $0<c<1$ that
	\begin{equation}\label{eq:BinomChernoffBound1}
		\IP\big(B\leq c e^{-4}\phi_LNp(N,L)\big)\leq \exp\Big(-\frac{(1-c)^2 e^{-4}\phi_LNp(N,L)}{2}\Big).
	\end{equation}
	Next set 
	$S:=\lceil\exp\big(c e^{-4}\phi_L Np(N,L))\rceil$ 
	and observe that
	\begin{equation}\label{eq:lowerG}
		\IP_{N}\big(|\cG_{k}|\geq c e^{-4}\phi_L Np(N,L)\   \text{ for all }\   k\leq e^{c e^{-4}\phi_L Np(N,L)}
		\big)\geq 1-\sum_{k=0}^{S}\IP_{N}\big(|\cG_{k}|< c e^{-4}\phi_L Np(N,L)\big).
	\end{equation}
	Thus, by using the fact that $|\cG_{k}|$  stochastically dominates  $B$ and \eqref{eq:BinomChernoffBound1} we get that
	\begin{equation*}
		\sum_{k=0}^{S}\IP_{N}\big(|\cG_{k}|< c e^{-4}\phi_L Np(N,L)\big)\leq \exp\Big(-e^{-4}\phi_L Np(N,L)\Big( \frac{(1-c)^2-2c}{2}\Big)\Big).
	\end{equation*}
	Note that $ (1-c)^2-2c> 2c$ if  $c <  3 - 2 \sqrt{2} $, and thus if we choose $c \in (0, 3 - 2 \sqrt{2})$  and set $c_L:=%c e^{-4}\phi_L/2
	c e^{-4}\phi_L$ we have 
	\begin{equation*}
		\sum_{k=0}^{S}\IP_{N}(|\cG_{k}|< c_L Np(N,L))\leq e^{-c_L Np(N,L)}.\qedhere
	\end{equation*}
\end{proof}

\subsection{Survival on a star}\label{sec:survivalStar}
In this section we study survival of the contact process on a (stable) star and for notational convenience formulate the results for $\rho$ being a star. So
let us introduce the process  $(\mathbf{C}_{t}^{*})_{t\geq 0}$ defined via the same graphical representation as $(\mathbf{C}_{t})_{t\geq 0}$ but restricted to $\cN_{\rho}\cup \{\rho\}$, i.e.
\begin{enumerate}
	\item we set
	$\mathbf{C}_{0}^{*} := \mathbf{C}_{0}\cap (\cN_{\rho}\cup \{\rho\})$ to be the initial configuration and
	\item an infection event at time $t\in \Delta^{\text{inf}}_{\{\rho,y\}} \cap J_k$ is only valid if $y\in \mathcal{{G}}_{k}$.
\end{enumerate}
This means that only $\rho$ and its good neighbours can participate in infection events. Now we define the set of \textit{good infected neighbours} of $\rho$ as follows
\begin{equation}\label{eq:setGxk}
	\mathcal{G}_{k}'= \mathcal{G}'_{\rho,k}:=\{y\in \mathcal{G}_{k}:\  y\in \bfC^{*}_{kT} \},
\end{equation}
In words this is the set of good neighbours in the interval $J_k$, which are infected at time $kT$, and thus are infected at any time during $[kT, (k+2)T]$ since good neighbours do not recover during $J_k$. Furthermore, for any $0<s<t$, we denote by $\cF_{s,t}$, the $\sigma$-algebra generated by the graphical construction up to time $s$, and the processes $\Delta^{\text{rec}}_y$ and $\Delta^{\text{up}}_{\{\rho,y\}}$ for all $y\in\cN_{\rho}$ up to time $t$.

The next lemma shows that with high probability the infection will persist for a long time in a stable star if we already have a sufficient number of infected good neighbours. Before we start we introduce some notation, set $\delta:=c_L/8$ and define the following event
\begin{equation*}
	\cW_{k}:= \left\{ |\mathcal{G}'_{k}|\geq \delta \lambda Np(N,L) T \quad \text{and}\quad  \int_{kT}^{(k+1)T} \1_{\{\rho \notin \mathbf{C}_s^{*}\}} \mathrm{d}s \leq \frac{T}{2} \right\}.
\end{equation*}
\begin{lemma}(\textbf{Local survival})\label{lem:probastar}
	Let $\lambda>0$ and $N\in \IN$ be chosen such that $\tfrac{3}{2}\lambda T_N<1$ and set $\bar{k}= \lfloor e^{\delta \lambda^2 T^2 Np(N,L)/4} \rfloor $. Then there exists a large universal constant $C$ such that if	$\lambda^2 T^2 N >C$, then, for any $k\leq e^{c_L Np(N,L)/2}$, we have
	\begin{equation*}
		\IP_{\cT,N}\bigg(\bigcap_{i=k}^{k+\bar{k}}\cW_{i} \ \Big|\Big. \  \mathcal{F}_{kT, (k+\bar{k})T}\bigg) \geq 1- C e^{-\delta \lambda^2 T^2  Np(N,L)/4},
	\end{equation*}
	on the event $\cS_{\rho, N}\cap \{|\mathcal{G}_{k}'|\geq \delta \lambda T  Np(N,L)\}$.
\end{lemma}
\begin{proof}
	The result follows analogously to \cite[Proposition 3.10]{jacob2022contact}. Note the condition \(\tfrac{3}{2}\lambda T_N < 1\), which does not appear in the statement of \cite{jacob2022contact}. However, while reproducing the proof, we found that the argument only holds for sufficiently small $\lambda$, as noted in the proof of \cite[Proposition 3.10]{jacob2022contact}. In particular, in our setting, we require the condition $\tfrac{3}{2}\lambda T_N < 1$.
\end{proof}

With this result we have control over the event that an infection persists for an exponential time if we already have a sufficiently large pool of infected good neighbours which can sustain the centre to be infected most of the time. It remains to estimate the probability to reach this amount of good infected neighbours in case that only the centre is initially infected.
\begin{lemma}\label{lemma:kickstart}
	Let $\lambda>0$ and $N\in \IN$ be chosen such that $2\lambda T_N<1$. Furthermore, let $s\leq e^{c_L Np(N,L)/2}$ and $k=\lfloor s/T \rfloor$. Then it holds that
	\begin{equation*}
		\IP_{\cT,N}(|\cG'_{k+2}|> \delta \lambda T Np(N,L) \mid \cF_{s,(k+4)T})>e^{-2T}(1-e^{-\delta\lambda T Np(N,L)}),
	\end{equation*}
	on the event $\cS_{\rho,N}\cap \{\rho\in \bfC^{*}_s\}$. 
\end{lemma}
\begin{proof}
	By the choice of $k$ we see that $kT\leq s\leq (k+1)T$, and thus the probability that $\rho$ does not recover in $[s, (k+2)T]$ is bounded from below by $e^{-2T}$. On the other hand, the probability that an infection event happens between a good neighbour $y\in\cG_{k+2}$ and $\rho$ in the time interval $[s,(k+2)T]$ is bound from below by $1-e^{-\lambda T} \ge \frac{\lambda T}{2}$. Note that the last inequality holds since $\lambda T<1$. Now given that $\rho$ does not recover in $[s, (k+2)T]$ we conclude that on the event $\cS_{\rho, N}\cap \{\rho\in \bfC^{*}_s\}$  the random variable $|\cG'_{k+2}|$ stochastically dominates a binomial random variable $B\sim \text{Bin}(\lceil 8\delta Np(N,L)\rceil, \tfrac{\lambda T}{2})$. Recall that $\delta=c_L/8$, where $c_L$ is the constant from Lemma~\ref{lemma:boundgoodnb}. Therefore, we get that
	\begin{equation*}
		\begin{aligned}
			\IP_{\cT,N}\big(|\cG'_{k+2}|>\delta \lambda T Np(N,L)\mid \cF_{s,(k+4)T}\big)
			& \geq  e^{-2T}\IP\big(B>\delta \lambda T Np(N,L)\big).
		\end{aligned}
	\end{equation*}
	Next we can again use a Chernoff bound for the binomial distribution (see e.g. \cite[Theorem 2.21]{van2016random}) to obtain
	\begin{equation*}
		\IP_{\cT,N}\Big(B
		>\delta \lambda T Np(N,L)\Big) \ge 1- e^{-\delta \lambda T Np(N,L)}.
	\end{equation*}
	This completes the proof.
\end{proof}

\subsection{Pushing the infection along a path}
\label{sec:PushingInfectionPath}

Given a single path of length $r$ on the vertices $\rho=x_0,x_1\dots,x_r$ such that $\{x_{i-1},x_{i}\}\in \cE$ for $1\leq i\leq r$, the goal in this subsection is to find a lower bound on the probability $\IP_{\cT,N}^{\{\rho\}}(x_r\in \bfC_{4r})$ where we let $\IP^{\{\rho\}}_{\cT,N}(\cdot)=\mathbb{P}_{\cT,N}(\ \cdot \ | \  \bfC_0= \{\rho\})$.
We begin by finding the probability to pass the infection along a single edge. In order to do so, we first make the following simple observation.

\begin{lemma}\label{lem:sumexp} Let $T_i^{(1)}, i \geq 1,$ be independent exponential random variables with parameter~$\alpha$ and 
	let $T_i^{(2)}, i \geq 1,$ be independent exponential random variables with parameter $\beta > 0$. 
	Let $N \sim {\rm Geom}(q)$ be independent of everything else, then 
	\[ T := \sum_{i=1}^N \left(T_i^{(1)} + T_i^{(2)}\right), \]
	has the Laplace transform
	\begin{equation}\label{eq:Laplacetran}
		\mathbb{E} \left[ e^{-\theta T}\right] = \frac{q \alpha \beta}{\theta^2 + \theta (\alpha + \beta) + q \alpha \beta},\qquad \theta >0.
	\end{equation}
	Moreover, $T$ has the same distribution as the sum of two exponential random variables with parameters  
	\[\frac{1}{2}(\alpha +\beta)+\sqrt{ \frac 14 ( \alpha + \beta)^2 - q \alpha \beta } \qquad \text{and}\qquad   \frac 12 (\alpha + \beta ) - \sqrt{ \frac 14 ( \alpha + \beta)^2 - q \alpha \beta }. \]
\end{lemma}

\begin{proof}
	We calculate for $\theta >0$,
	\[ \begin{aligned} 
		\mathbb{E}\left[ e^{- \theta T}\right] 
		& = \sum_{n=1}^\infty \mathbb{E}\left[ \prod_{i=1}^n e^{- \theta(T_i^{(1)} + T_i^{(2)})} \right] \mathbb{P} (N = n) \\&= \sum_{n=1}^\infty\left(\mathbb{E}\left[ e^{-\theta T_i^{(1)}}\right] \mathbb{E}\left[e^{- \theta T_i^{(2)}}\right] \right)^n\mathbb{P} (N = n)\\&
		= \sum_{n=1}^\infty \Big( \frac{\alpha}{\alpha +\theta}\Big)^n\Big( \frac{\beta}{\beta +\theta}\Big)^n  (1-q)^{n-1} q = q \frac{\alpha}{\alpha +\theta}\frac{\beta}{\beta +\theta} \frac{1}{ 1 - (1-q)   \frac{\alpha}{\alpha +\theta}\frac{\beta}{\beta +\theta} }\\
		& = \frac{q \alpha \beta}{(\alpha + \theta)(\beta + \theta) - (1-q) \alpha \beta } = \frac{q \alpha \beta}{\theta^2 + \theta (\alpha + \beta) + q \alpha \beta}.
	\end{aligned}
	\]
	For the second claim, we recall that the Laplace transform of the sum 
	of two independent exponential random variables with parameter $a$ and $b$ has the form $a(a+\theta)^{-1}b(b+\theta)^{-1}$.
	%\[ \frac{a}{a+\theta}\frac{b}{b+\theta}.\]
	By comparing factors, we see that $a$ and $b$ have to satisfy 
	$ a + b = \alpha + \beta$  and  $a b = q\alpha \beta,$
	which is true by taking $a$ and $b$ as in the statement of the lemma.
\end{proof}

For the next lemma we consider some arbitrary but fixed edge $\{x,y\}$. We denote by $T^{\inf}$ and $T^{\rm rec}$ the first \textit{true} infection time on $\{x,y\}$ and the first recovery time of $x$, i.e.\@
\begin{equation*}
	T^{\inf}:= \inf\big\{t\geq 0: \ t\in \Delta^{\rm inf}_{\{x,y\}} \ \text{and} \  \{x,y\} \in \bfB_t \big\} \quad \text{and} \quad  T^{\text{rec}}:= \inf\big\{t\geq 0: \ t\in \Delta^{\rm rec}_{x}\big\}.
\end{equation*}

\begin{lemma}\label{lem:InfBeforeRec}
	It holds that
	\begin{equation}\label{eq:InfBeforeRec}
		\IP(T^{\inf}<T^{\rm rec} \mid \{x,y\}\notin \bfB_0)
		=\frac{\lambda v(\de_x,\de_y) p(\de_x,\de_y)}{\lambda+v(\de_x,\de_y)+\lambda v(\de_x,\de_y) p(\de_x,\de_y)+1}.
	\end{equation}
	Furthermore, for any $t>0$
	\begin{align*}
		&  \IP(T^{\inf} > t\mid T^{\inf} < T^{\rm rec}, \{x,y\}\notin \bfB_0 )=\frac{b+1}{b-a} e^{-(a+1)t} + \frac{a+1}{a-b} e^{-(b+1)t},
	\end{align*}
	where
	\begin{equation}\label{eq:ab}
		\begin{aligned}
			a&:=\frac{\lambda +v(\de_x, \de_y)}{2}+\sqrt{ \frac 14 ( \lambda +v(\de_x, \de_y))^2 -  \lambda v(\de_x, \de_y)p(\de_x, \de_y) },\\ & b:=\frac{\lambda +v(\de_x, \de_y)}{2}-\sqrt{ \frac 14 ( \lambda +v(\de_x, \de_y))^2 - \lambda v(\de_x, \de_y)p(\de_x, \de_y) }.
		\end{aligned}
	\end{equation}
\end{lemma}

\begin{proof}
	For simplicity we drop the argument  $(\de_x,\de_y)$ and write $v := v(\de_x,\de_y)$ and $p= p(\de_x,\de_y)$.
	Suppose that initially the edge $\{x,y\}$ is closed and we denote by $\IP(\, \cdot\, )$ referring to $\IP(\, \cdot \mid \{x,y\}\notin \bfB_0)$. Set $S_0^{\rm cl/inf} = 0$ and for $i \geq 1$, iteratively define 
	\[ S_i^{\rm op}  
	= \inf\left\{ t \geq S_{i-1}^{\text{cl/inf}} \, : \, t \in \Delta^{\rm op}_{\{x,y\}} \right\}\quad \text{and}\quad
	S_i^{\rm cl/inf} = \inf\Big\{ t \geq S_{i}^{\rm op}  \, : \, t \in \Delta^{\rm cl}_{\{x,y\}} \cup \Delta^{\rm inf}_{\{x,y\}} \Big\},
	\]
	as the times at which  an edge  opens and then the times, once the edge is open, for the edge to either close or for an infection to be passed along.
	Denote the corresponding waiting times as 
	\[ T_i^{\rm op} = S_i^{\rm op } - S_{i-1}^{\rm cl/inf}  \qquad \mbox{and}\qquad
	T_i^{\rm cl/inf} = S_i^{\rm cl/inf} - S_i^{\rm op} , \quad i\ge 1.\]
	Now, define 
	\[ I_i = \1_{\left\{ S_i^{\rm cl/inf} \in \Delta^{\rm inf}_{\{x,y\}} \right\}}, \]
	as the indicator for the event that on the $i$-th trial an infection event takes place rather than a closing event.
	
	By the properties of Poisson processes and their exponential waiting times, 
	the random variables 
	$T_i^{\rm op}, T_i^{\rm cl/inf}, I_i, i \geq 1$,
	are all independent. Moreover, 
	\begin{equation}\label{eq:TopTcl}
		T_i^{\rm op} \sim {\rm Exp}(pv), \quad 
		T_i^{\rm cl/inf} \sim {\rm Exp}((1-p)v + \lambda)\quad  \text{and}\quad  \IP(I_i = 1) = \frac{\lambda}{\lambda + (1-p) v}.
	\end{equation}
	Then, if we define $N = \inf\{ n \geq 1, \: \, I_n = 1\}$ to count the total number of trials to pass the infection along the edge, we have that $N$ is a geometric random variable with parameter $q = \lambda(\lambda + (1-p)v)^{-1}$ that is 
	independent of 
	$T_i^{\rm op}, T_i^{\rm cl/inf}, i \geq 1$. Thus the first time when we have a true infection can be written in terms of these waiting times as follows,
	\begin{equation}\label{eq:Tinf} T^{\rm inf} = \sum_{i=1}^N \left( T_i^{\rm op} + T_i^{\rm cl/inf}\right). \end{equation}
	Denote by $f_{T^{\rm inf}}$ the density of $T^{\rm inf}$. We can now calculate by independence and then Fubini's theorem the following probability
	\[ \begin{aligned} \mathbb{P} ( T^{\rm inf} < T^{\rm rec} )
		& = \int_0^\infty e^{-x} \mathbb{P} (T^{\rm inf} < x ) \, \mathrm{d} x = \int_0^\infty e^{-x} \int_0^x f_{T^{\rm inf}} (y) \, \mathrm{d}y \, \mathrm{d}x \\ &  = \int_0^\infty f_{T^{\rm inf}}(y) \int_y^\infty e^{-x} \, \mathrm{d}x \, \mathrm{d}y  = \int_0^\infty f_{T^{\rm inf}}(y) e^{-y} \, \mathrm{d}y
		= \mathbb{E}\left[ e^{- T^{\rm inf}} \right].
	\end{aligned} \]
	Using the representation in~\eqref{eq:Tinf} and Lemma \ref{lem:sumexp} with $\theta =1, \alpha=pv, \beta = (1-p)v+\lambda$ and $q= \lambda(\lambda + (1-p)v)^{-1}$ we get the first claim. For the second claim, we first see that 
	\begin{align}\label{eq:lawTinf}
		\IP(T^{\inf} > t,\, T^{\inf} < T^{\rm rec} )&=\int_{t}^{\infty}\int_{s}^{\infty} f_{T^{\inf}}(s)e^{-u}\, \mathrm{d}u\, \mathrm{d}s
		=\int_{t}^{\infty} f_{T^{\inf}}(s)e^{-s}\, \mathrm{d}s.
	\end{align}
	By the second part of Lemma \ref{lem:sumexp} we know that $T^{\inf}$ has the same law as the sum of two exponential random variables  $X$ and $Y$ with parameters $a$ and $b$ as defined in \eqref{eq:ab}. Then it is straightforward to see that
	\begin{equation*}
		\IP(X+Y>s)=\frac{b}{b-a} e^{-a s} + \frac{a}{a-b} e^{-b s}
	\end{equation*}
	which yields
	\[f_{T^{\inf}}(s) =  \frac{ba}{b-a}e^{-as} + \frac{ab}{a-b}e^{-bt}.\]
	Plugging this back into \eqref{eq:lawTinf}, we get	
	\begin{align*}
		\IP(T^{\inf} > t,\, T^{\inf} < T^{\rm rec} )=   \frac{ba}{(b-a)(a+1)} e^{-(a+1)t} + \frac{ab}{(a-b)(b+1)} e^{-(b+1)t}.
	\end{align*}
	Using \eqref{eq:InfBeforeRec} together with the identities $\lambda+v+\lambda p v +1 = (1+a)(1+b)$ and $\lambda pv = ab$, we deduce the desired result.
\end{proof}

Finally we will show a lower bound for the probability of pushing the infection along a \textit{path of bounded degree} to a \textit{stable star}. The proof uses a similar line of arguments as in \cite[Lemma 5.1]{cardona2021contact}, however the presence of the dynamical graph structure leads to some changes. In order to do so, we first introduce some notation.

Let $r, N\in \mathbb{N}$ and denote by  $\cV_r = \{x \in \cV : d(\rho, x) = r\}$ the set of vertices in generation~$r$. Note that due to the tree structure of $\cT$ we find a unique path $\rho=x_0,x_1,\dots,x_r=x$ from $\rho$ to $x$, which means that $\{x_{i-1},x_{i}\}\in \cE$ for every $1\leq i\leq r$.

For any vertex $x\in \cV_r$, we define the subgraph $\cT^{\rho-x}$ where the vertex set is given through $\cV^{\rho-x}:=\{\rho,x_1,\dots,x_{r-1},x\}$ and we consider all edges $\{x,y\}\in \cE$ such that $x,y\in \cV^{\rho-x}$. Now we introduce the process $(\bfC_{t}^{\rho-x})_{t\geq 0}$, which is defined via the same graphical representation as $(\mathbf{C}_{t})_{t\geq 0}$ but the infection is restricted to $\cT^{\rho-x}$, i.e.\
\begin{enumerate}
	\item we set 
	$\mathbf{C}_{0}^{\rho-x} = \bfC_{0}\cap\cV^{\rho-x}$ to be the initial configuration and
	\item only infection events contained in the subgraph $\cT^{\rho-x}$ are used.
\end{enumerate}
This means that only vertices contained in $\cV^{\rho-x}$ participate in infection events, and therefore $\bfC^{\rho-x}_{t}\subset \cV^{\rho-x}$ for all $t\geq 0$.

Now we define $\mathbf{P}_r$ as the set of all vertices in generation $r$ with degree $N+1$ which are connected to $\rho$ via a path consisting of vertices with degree bounded by $L$, i.e.
\begin{equation}\label{eq:paths}
	\mathbf{P}_r:= \{x\in \cV_r:\,  \de_x =N+1,  \de_{x_i} \le L \ \text{ for all } 1\leq i\leq r-1 \}.
\end{equation}
We emphasize here that the notation $\mathbb{P}^{\{\rho\}}_{\mathcal{T},N}(\cdot)$ in what follows corresponds to the conditional probability on the tree and on the event that the root $\rho$ is infected and the background process starts in any distribution.

\begin{lemma}\label{lem:probapath}
	Let $r\in\IN$.  Suppose the background process is started in an arbitrary distribution. Let  $x \in \cV_r$ and $\rho=x_0,x_1,\dots,x_r=x$ be the unique path from $\rho$ to $x$ then there exists a constant $\gamma>0$ independent of $x$ and $r$ such that
	\begin{equation*}
		\IP^{\{\rho\}}_{\mathcal{T}}\Big(x\in \bigcup_{s \leq 4r} \bfC^{\rho-x}_{s}\Big)\geq(1-e^{-\gamma r})\prod_{i=1}^{r}  \frac{\lambda v(\de_{x_{i-1}}, \de_{x_i}) p(\de_{x_{i-1}},\de_{x_{i}})}{\lambda+v(\de_{x_{i-1}}, \de_{x_i})+\lambda v(\de_{x_{i-1}}, \de_{x_i}) p(\de_{x_{i-1}},\de_{x_{i}})+1}.
	\end{equation*}
	Moreover, let $N, L\in \mathbb{N}$ with $N\geq L$ and suppose that $p(\cdot, \cdot)$ and $v(\cdot, \cdot)$ satisfy \eqref{eq:kernelsgen} and \eqref{eq:kernelsgen2}, respectively. Then for $x\in\mathbf{P}_r$ it follows that
	\begin{equation*}
		\IP^{\{\rho\}}_{\mathcal{T},N}\Big(x\in \bigcup_{s \leq 4r} \bfC^{\rho-x}_{s}\Big) \geq  (1-e^{-\gamma})\bigg(\frac{\lambda \nu_1 \kappa_1 c_p }{\lambda+\lambda \nu_1+\nu_1+1}\bigg)^{r} C_p(N),
	\end{equation*}
	where 
	\begin{equation}\label{eq:constantCp}
		C_p(N):= 
		\Big(\frac{\kappa_1}{\kappa_2 c_p} (N+1)^{(\eta\wedge 0)-\alpha} \Big)^2\, \quad \text{ and }\, \quad  c_p :=   L^{(\eta\wedge 0)-\alpha}.
	\end{equation}
\end{lemma}

\begin{proof}
	Let us consider the \textit{true infection point process} on the edge $\{x,y\}$. By this we mean only the times $t\in \Delta^{\inf}_{\{x,y\}}$ such that $\{x,y\}\in \bfB_t$, which we denote by
	\begin{equation*}
		\Delta^{\bfB}_{\{x,y\}}:=\{t\in \Delta^{\inf}_{\{x,y\}}: \{x,y\}\in \bfB_t\}.
	\end{equation*}
	Let us define the sequence of random times as $s_0=0$ and for $i\in \{1,\dots,r\}$ set
	\begin{equation*}
		s_{i}:=\inf\big\{s>s_{i-1}: s\in \Delta^{\bfB}_{\{x_{i-1},x_{i}\}}\cup\Delta^{\text{rec}}_{x_{i-1}}\big\}.
	\end{equation*}
	Furthermore, define 
	%$\tau:=\sum_{i=1}^{r}t_i$ with 
	$t_i=s_i-s_{i-1}$ for $i\in \{1,\dots, r\}$. 
	Also denote by
	\begin{equation*}
		B_i:=\big\{s_i\in\Delta^{\bfB}_{\{x_{i-1},x_{i}\}}\big \} \qquad \text{ and }\qquad B:=\bigcap_{i=1}^{r} B_i,
	\end{equation*}
	where $B_i$ is the event that $x_{i-1}$ infects $x_{i}$ before it recovers. Now we see that
	\begin{equation}\label{eq:boundxr}
		\IP^{\{\rho\}}_{\cT}\Big(x_r\in \bigcup_{s \leq 4r} \bfC^{\rho-x}_{s}\Big)\geq \IP^{\{\rho\}}_{\cT}(B\cap \{s_r\leq 4r\})= \IP^{\{\rho\}}_{\cT}(s_r\leq 4r \mid B )\IP^{\{\rho\}}_{\cT}(B).
	\end{equation}
	Recall that from Lemma \ref{lem:InfBeforeRec}, we have 
	\[\mathbb{P}_{\cT}^{\{\rho\}}(B_i\ | \ \{x_{i-1}, x_i\} \not\in \bfB_{s_{i-1}}) = \frac{\lambda v(\de_{x_{i-1}}, \de_{x_i}) p(\de_{x_{i-1}},\de_{x_i})}{\lambda+v(\de_{x_{i-1}}, \de_{x_i})+\lambda v(\de_{x_{i-1}},\de_{x_i}) p(\de_{x_{i-1}},\de_{x_i})+1}.\] 
	By monotonicity with respect to the initial condition in the background process, we have
	\[\mathbb{P}_{\cT}^{\{\rho\}}(B_i) \geq \mathbb{P}_{\cT}^{\{\rho\}}(B_i\ |  \ \{x_{i-1}, x_i\} \not\in \bfB_{s_{i-1}}),\]
	which implies due to the Markov property that 
	\begin{align}\label{eq:B}
		\IP_{\cT}^{\{\rho\}}(B)&\geq\prod_{i=1}^r \mathbb{P}^{\{\rho\}}(B_i\ | \ \{x_{i-1}, x_i\} \not\in \bfB_{s_{i-1}})\nonumber\\ &= \prod_{i=1}^{r}\frac{\lambda v(\de_{x_{i-1}}, \de_{x_i}) p(\de_{x_{i-1}},\de_{x_{i}})}{\lambda+v(\de_{x_{i-1}}, \de_{x_i})+\lambda v(\de_{x_{i-1}}, \de_{x_i}) p(\de_{x_{i-1}},\de_{x_{i}})+1}.
	\end{align}
	On the other hand, appealing to Markov’s inequality and the definition of $B$ and $s_r$, we have, for any $\theta>0$
	\begin{align*}
		\IP_{\cT}^{\{\rho\}}(s_r\geq 4r \ | \ B)&=\IP_{\cT}^{\{\rho\}}(e^{\theta s_r}\geq e^{4\theta r}\ | \ B)\leq e^{-4\theta r}\IE_{\cT}^{\{\rho\}}\left[e^{\theta s_r} \ \Big|\Big.\ B\right]\\ &
		\leq e^{-4\theta r}\IE_{\cT}^{\{\rho\}}\bigg[ \prod_{i=1}^{r}e^{\theta t_i}\ \Big|\Big. \ \bigcap_{i=1}^{r} B_i \cap \{\{x_{i-1}, x_i\}\not \in \bfB_{s_{i-1}}\} \bigg]\\ &
		\leq e^{-4\theta r}\prod_{i=1}^{r}\IE_{\cT}^{\{\rho\}}\Big[e^{\theta t_i}\ \Big|\Big. \ B_i\cap \{\{x_{i-1}, x_i\}\not \in \bfB_{s_{i-1}}\}\Big],
	\end{align*}
	where in the second inequality we have used that the waiting time $t_i$ is larger when we start with  closed edges and  in the third inequality we have used that $t_i$ is independent of $B_j$ for all $j\not=i$ and also the times $t_i$ are independent of each other.
	
	Now appealing to Lemma \ref{lem:InfBeforeRec}, we have that
	\begin{equation*}
		\IP_{\cT}^{\{\rho\}}(t_i> s \ | \ B_i\cap \{\{x_{i-1}, x_i\}\not \in \bfB_{s_{i-1}}\}) =\frac{b_i+1}{b_i-a_i} e^{-(a_i+1)s} + \frac{a_i+1}{a_i-b_i} e^{-(b_i+1)s},
	\end{equation*}
	where $a_i$ and $b_i$ are defined as in \eqref{eq:ab} for the corresponding edge $\{x_{i-1}, x_i\}$.
	This is the distribution of the sum of two independent exponential random variables with parameters $a_i+1$ and $b_i+1$. Since $a_i, b_i>0$, by a coupling argument one can show that the law of the sum of two standard exponential random variables stochastically dominates the law of $t_i$. We also know that the distribution of the sum of two standard exponential random variables is a Gamma distribution with parameter $2$ and $1$. Then the law of $t_i$ is stochastically dominated by the law of $r_i$ where $r_i \sim \Gamma(2,1)$, which implies 
	\begin{align*}
		\IP_{\cT}^{\{\rho\}}\left(s_r\geq 4r \mid  B\right)&\leq e^{-4\theta r+ r \log \phi(\theta)},
	\end{align*}
	where $\phi(\theta):= \mathbb{E}_{\cT}[e^{\theta r_i}] = (1-\theta)^{-2}$. 
	Now, note that
	%\begin{equation*}
	$\lim_{\theta \to 0}\frac{\log((1-\theta)^{-2})}{\theta}=2.$
	%\end{equation*}
	Therefore, by choosing $\theta>0$ small enough, we can deduce that there exists $\gamma>0$ (independent of $r$) such that
	\begin{equation}\label{eq:boundtau}
		\IP_{\cT}^{\{\rho\}}(s_r\geq 4r \ | \ B) \leq e^{-r\gamma}.
	\end{equation}
	Plugging \eqref{eq:B} and \eqref{eq:boundtau} back into \eqref{eq:boundxr}, we get the first statement. For the second claim, we assume that $p(\cdot, \cdot)$ and $v(\cdot, \cdot)$ satisfy \eqref{eq:kernelsgen} and \eqref{eq:kernelsgen2}, respectively, and that $x\in \mathbf{P}_r$.
	First, we deal with the case $\eta\ge 0$. Since $\de_{x_i}\ge 1$ for all $0\le i\le r$ in this case we have $v(\cdot,\cdot)\ge \nu_1$ and using the fact that the mapping $v\mapsto \lambda vp(\lambda +v +\lambda vp +1)^{-1}$ is increasing, we deduce
	\begin{equation*}
		\frac{\lambda v (\de_{x_{i-1}},\de_{x_i}) p(\de_{x_{i-1}},\de_{x_i})}{\lambda+v (\de_{x_{i-1}},\de_{x_i})+\lambda v (\de_{x_{i-1}},\de_{x_i}) p(\de_{x_{i-1}},\de_{x_i})+1}
		\geq\frac{\lambda \nu_1 p(\de_{x_{i-1}},\de_{x_i})}{\lambda+\nu_1+\lambda \nu_1 p(\de_{x_{i-1}},\de_{x_i})+1}.
	\end{equation*}
	Furthermore using that $p(\de_{x_{i-1}},\de_{x_i})\leq 1$, we get that
	\begin{equation*}
		\frac{\lambda v(\de_{x_{i-1}},\de_{x_i}) p(\de_{x_{i-1}},\de_{x_i})}{\lambda+v(\de_{x_{i-1}},\de_{x_i})+\lambda v(\de_{x_{i-1}},\de_{x_i}) p(\de_{x_{i-1}},\de_{x_i})+1}\geq p(\de_{x_{i-1}},\de_{x_i})\frac{\lambda \nu_1}{\lambda+\lambda \nu_1+\nu_1+1}.
	\end{equation*}
	Therefore, using the first statement of the lemma together with the inequality  $1-e^{-\gamma r} \ge  1-e^{-\gamma}$ which holds for $r \ge 1$ and the fact that $p(\cdot, \cdot)$ is decreasing in both arguments, we get that
	\[\begin{split}
		\IP^{\{\rho\}}_{\mathcal{T},N}\Big(x\in \bigcup_{s \leq 4r} \bfC^{\rho-x}_{s}\Big)
		&\ge  (1-e^{-r}) \left(\frac{\lambda \nu_1}{\lambda+\lambda \nu_1+\nu_1+1}\right)^r p(N,L)^2 p(L,L)^{r-2}.
	\end{split}\]
	For the case $\eta<0$ we have that $v(n,m) \ge \nu_1 n^{\eta}$ for all $m\leq L$ and $n\geq L$, and furthermore that $n^\eta\leq 1$ for all $n\geq 1$. Now using that $v$ is a monotone function we get that
	\begin{equation*}
		\frac{\lambda v(\de_{x_{i-1}},\de_{x_i}) p(\de_{x_{i-1}},\de_{x_{i}})}{\lambda+v(\de_{x_{i-1}},\de_{x_i})+\lambda v(\de_{x_{i-1}},\de_{x_i}) p(\de_{x_{i-1}},\de_{x_{i}})+1}
		\geq\frac{\lambda  \nu_1 L^{\eta}  p(\de_{x_{i-1}},\de_{x_{i}})}{\lambda+\nu_1 +\lambda \nu_1 +1},
	\end{equation*}
	for all $1\leq i\leq r-1$ and for $i\in \{0,r\}$ it follows that
	\begin{equation*}
		\begin{aligned}
			\frac{\lambda v(\de_{x_{i-1}},\de_{x_i}) p(\de_{x_{i-1}},\de_{x_{i}})}{\lambda+v(\de_{x_{i-1}},\de_{x_i})+\lambda v(\de_{x_{i-1}},\de_{x_i}) p(\de_{x_{i-1}},\de_{x_{i}})+1}
			\geq\frac{\lambda \nu_1(N+1)^{\eta} p(\de_{x_{i-1}},\de_{x_{i}})}{\lambda+\nu_1 +\lambda \nu_1 +1}.
		\end{aligned}
	\end{equation*}
	Since $p(\cdot, \cdot)$ is a positive decreasing function in both arguments, we deduce that
	\begin{equation*}
		\begin{aligned}
			\IP^{\{\rho\}}_{\mathcal{T},N}\Big(x\in \bigcup_{s \leq 4r}  \bfC^{\rho-x}_{s} \Big) &\geq (1-e^{-\gamma})\bigg(\frac{\lambda \nu_1 L^{\eta}p(L,L) }{\lambda+\lambda \nu_1+\nu_1+1}\bigg)^r \left(\frac{(N+1)^\eta p(N,L)}{L^{\eta} p(L,L)}\right)^2.
		\end{aligned}
	\end{equation*}
	Furthermore, we know that $\kappa_1 N^{-\alpha}\leq p(L,N)\leq \kappa_2 N^{-\alpha}$ for $N\geq L$, and thus we get that for all $\eta\in\IR$ that
	\begin{align*}
		\IP^{\{\rho\}}_{\mathcal{T},N}\Big(x\in \bigcup_{s \leq 4r} \bfC^{\rho-x}_{s} \Big) &\geq (1-e^{-\gamma})\bigg(\frac{\lambda \nu_1 \kappa_1 L^{(\eta\wedge 0)-\alpha}}{\lambda+\lambda \nu_1+\nu_1+1}\bigg)^r \left(\frac{ \kappa_1 (N+1)^{(\eta\wedge 0)-\alpha} }{ \kappa_2 L^{(\eta\wedge 0)-\alpha} }\right)^2.
	\end{align*}
	This concludes the proof.   
\end{proof}

\subsection{Transmitting the infection from one star to another}\label{sec:PushingInfection}
In this section, we finally want to bring together the main results of the previous two subsections. We will now define an infection process $\bfC^{*-x}$, which is restricted to the neighbourhood of the root $\rho$ and a single path leading to a star $x$ in generation $r$.

For any vertex $x\in \cV_r$ we define the subgraph $\cT^{*-x}$, where the vertex set is given through $\cV^{*-x}:=\cN_{\rho}\cup \{\rho,x_1,x_2,\dots,x_{r-1},x\}$ and we consider all edges $\{x,y\}\in \cE$ such that $x,y\in \cV^{*-x}$. Now we introduce the process $(\bfC_{t}^{*-x})_{t\geq 0}$, which is defined via the same graphical representation as $(\mathbf{C}_{t})_{t\geq 0}$ but the infection is restricted to $\cT^{*-x}$ , i.e.\
\begin{enumerate}
	\item we set 
	$\mathbf{C}_{0}^{*-x} = \bfC_{0}\cap\cV^{*-x}$ to be the initial configuration and
	\item only infection events contained in the subgraph $\cT^{*-x}$ are used.
\end{enumerate}
This means that only vertices contained in $\cV^{*-x}$ participate in infection events, and therefore $\bfC^{*-x}_{t}\subset \cV^{*-x}$ for all $t\geq 0$.

If we start with only the centre of a star infected, we estimate the probability to push the infection to the next star. We use Lemmas \ref{lem:probastar} and \ref{lem:probapath} to find a lower bound for the probability that the star $x$, which is at distance $r$ from the root and has the same degree as the root, is infected before the infection dies out around the root. 

Before we proceed, we recall some notation. We defined the set of good infected neighbours by $\mathcal{G}_{k}'$ in \eqref{eq:setGxk},	$\mathcal{F}_{s,t}$ the $\sigma$-algebra generated by the graphical construction up to time $s$, and the processes $\Delta^{\text{rec}}_y$ and $\Delta^{\text{up}}_{\{\rho,y\}}$ for all $y\in \cN_{\rho}$ up to time $t$. 
Lastly at the beginning of Subsection~\ref{sec:StableStar} we defined $T=T_N=(1+v_2N^\eta)^{-1}$. 

The next lemma shows that, if the root has enough good infected neighbours then the root is infected with sufficiently high probability at a given time.
\begin{lemma}\label{InfectedContant}
	Let $s>0$ and $N,M\in \IN$ with $M\leq N$. Now set $k:=\lfloor s/T\rfloor$, then it holds that  
	\begin{equation*}\label{eq:probrootinfected}
		\mathbb{P}_{\cT,N}\big(\rho \in \bfC^{*-x}_{(k+2)T} \ \big|\big. \ \cF_{s,(k+2)T}\ ,\  \rho \notin \bfC^{*-x}_{s}\big)
		\geq \frac{ \lambda M T}{( \lambda M +1)T+1},
	\end{equation*}
	on the event $\{|\mathcal{G}_{\rho,k}'|\geq M\}$.
\end{lemma}

\begin{proof}
	Let us first consider a simpler setting, that is we consider a star graph with centre~$\rho$ and $M$ leaves denoted by $x_1,\dots, x_M$. We consider the following situation: we start with every leaf infected and the root healthy. Furthermore, the leaves are unable to recover, but the root can. We want to find a lower bound on the probability that the root is infected at time $t$. This problem is described by a two-state Markov process $X=(X_s)_{s\in[0,t]}$ on the state space $\{0,1\}$, where $\rho$ is either infected (state 1) or healthy (state 0) and we use the Poisson processes $\Delta^{\inf}_{\{\rho, x_i\}}$ and $\Delta^{\text{rec}}_{\rho}$ to construct it. Set
	\begin{equation*}
		\Delta^{M,\inf}:=\bigcup_{i=1}^{M}\Delta^{\inf}_{\{\rho,x_i\}} \qquad  \text{ and } \qquad \Delta^M:=\big(\Delta^{M,\inf}\cup \Delta^{\text{rec}}_{\rho}\big).
	\end{equation*}
	We set initially $X_0=0$. Now at any time point in $\Delta^{M,\inf}$ the process $X$ jumps to state~$1$ and at any time point in $\Delta^{\text{rec}}_{\rho}$ it jumps into state $0$, i.e.\ we get a Markov process with transitions
	\begin{equation*}
		\begin{aligned}
			0 \to 1 &\quad \text{ with rate } \lambda M\\
			1 \to 0 &\quad \text{ with rate } 1.
		\end{aligned}
	\end{equation*}
	The distribution of this two-state Markov process $X$ is well-known,
	\begin{align*}
		\IP(X_t=1 \mid X_0=0)=\frac{\lambda M}{\lambda M +1} (1-e^{-(\lambda M+1)t}),
	\end{align*}
	(see e.g.\ \cite[Example~2.6]{liggett2010continuous}). Using the inequality $1-e^{-x}\geq\frac{x}{x+1}$ which holds for all $x>-1$, we see that
	\begin{equation*}
		\frac{\lambda M}{\lambda M+1}\big(1 -e^{-( \lambda M+1)t} \big)\geq \frac{\lambda M}{ \lambda M+1}\frac{ (\lambda M+1)t}{ ( \lambda M+1)t+1}=\frac{ \lambda M t}{ (\lambda M+1)t+1},
	\end{equation*}
	and thus
	\begin{equation*}
		\IP(X_t=1 \mid X_0=0)\geq \frac{ \lambda M t}{( \lambda M +1)t+1}.
	\end{equation*}
	Now we come back to our original goal. Since we are on the event $\{|\mathcal{G}_{\rho,k}'|\geq M\}$ we know that there exist at least $M$ good neighbours which are infected. By the definition of good neighbours we know that they do not recover and they are connected to the root $\rho$ until time $(k+2)T$. 
	But this means we are again in the same setting as above, since we have a star graph with at least $M$ infected leaves which are unable to recover and we want to know what the probability is that the root is infected after a time step of length $(k+2)T-s\geq T$.  Hence, by monotonicity we get the claim.
\end{proof}

The next lemma provides an upper bound for the probability that, at some point in the time interval $[0, (\bar{k}+2)T ]$, the number of infected good neighbors is less than or equal to $\delta \lambda TNp(N,L)$ given that the root is initially infected and has degree $N$, where $\bar{k}= \lfloor e^{\delta \lambda^2 T^2 Np(N,L)/4} \rfloor $ as defined in Lemma \ref{lem:probastar} for $\delta=c_L/8$ with $c_L$ chosen as in Lemma~\ref{lemma:boundgoodnb}.

\begin{lemma}\label{lem:goodneighbors}
	Let $z\in\cN_{\rho}$, $\lambda>0$ and $N\in \IN$ be chosen such that $\tfrac{3}{2}\lambda T<1$ and $ \lambda^2 T^2 N >C$, where $C$ denotes the universal constant given in Lemma \ref{lem:probastar}. Suppose the background process $\bfB$ is started in stationarity. Then, for $N$ large enough it holds that
	\begin{equation*}
		\IP_{\cT, N}^{\{\rho\}}\Big(\inf_{2\leq k \leq \bar{k}} |\mathcal{G}_{k}'| \leq \delta \lambda T  Np(N,L)\,  \Big| \, \cF_{0,S}\Big)\leq R_{\lambda}(N),
	\end{equation*}
	on the event $\cS$ (defined in \eqref{eq:setS}), where 
	\begin{equation}\label{eq:S}
		S=S_N:= T \lfloor e^{\delta \lambda^2 T^2 Np(N,L)/4} \rfloor,
	\end{equation} and
	\begin{equation}\label{eq:functionR}
		R_{\lambda}(N):=1-(1- C e^{-\delta \lambda^2 T^2  Np(N,L)})e^{-2T}(1-e^{-\delta\lambda T Np(N,L)}).
	\end{equation}
\end{lemma}

\begin{proof}
	We begin by observing
	\begin{align*}
		\IP^{\{\rho\}}_{\cT,N}\Big(& \inf_{2\leq k \leq \bar{k}}  |\mathcal{G}_{k}'| \geq \delta \lambda T  Np(N,L)\ \Big| \ \mathcal{F}_{0,S}\Big)\\
		& \geq \IE_{\cT,N}^{\{\rho\}}\bigg[\IP_{\cT,N}\bigg(\bigcap_{k=2}^{\bar{k}} \{|\mathcal{G}'_{k}|\geq \delta \lambda T Np(N,L)   \} \ \Big|\Big. \  \mathcal{F}_{2T,S}\bigg) \1_{ \{|\mathcal{G}_{2}'|\geq \delta \lambda T  Np(N,L)\}}\Big| \ \mathcal{F}_{0,S}\bigg].
	\end{align*}
	Now note that the conditional probability of $\{|\mathcal{G}_{2}'|\leq \delta \lambda T  Np(N,L) \}$ on $\mathcal{F}_{0,S}$ is the same as on $\mathcal{F}_{0,4T}$, since $|\mathcal{G}_{2}'|$ only depends on events in the graphical representation contained in the time interval $[0,4T]$. Therefore, it is independent of everything happening after the time point $4T$. By Lemma~\ref{lem:probastar}, we deduce that the inequality
	\begin{equation*}
		\IP_{\cT,N}\bigg(\bigcap_{k=2}^{\bar{k}} \{|\mathcal{G}'_{k}|\geq \delta \lambda T Np(N,L)   \} \ \Big|\Big. \  \mathcal{F}_{2T,S}\bigg) \geq 1- C e^{-\delta \lambda^2 T^2  Np(N,L)}
	\end{equation*}
	holds on the event $\cS\cap \{|\mathcal{G}_{2}'|\geq \delta \lambda T  Np(N,L)\}$ and with Lemma~\ref{lemma:kickstart} it follows that
	\begin{equation*}
		\IP_{\cT,N}^{\{\rho\}}\big( |\mathcal{G}_{2}'|\geq \delta \lambda T  Np(N,L)\big| \ \mathcal{F}_{0,S}\big)
		\geq e^{-2T}(1-e^{-\delta\lambda T Np(N,L)})
	\end{equation*}
	holds on the event $\cS$. This completes the proof.
\end{proof}

\begin{remark}\label{rem:excludingvertex}
	For technical reasons, we need to ensure that the reservoir around $\rho$ and the path towards the next star share only the centre $\rho$, and are otherwise defined on disjoint parts of the graph. This technicality poses no real issue, since it is not difficult to realise that all the results in Section~\ref{sec:StableStar}, as well as Lemmas~\ref{InfectedContant} and~\ref{lem:goodneighbors}, still hold when  a single neighbour $x_1 \in \mathcal{N}_{\rho}$ is excluded. To avoid introducing additional notation, we continue to use the symbols $\mathcal{G}_{\rho,k}$, $\mathcal{G}'_{\rho,k}$, $\mathcal{F}_{s,t}$, and $\mathcal{S}_{\rho, N}$ throughout the remainder of the paper, with the convention that these objects are now defined excluding the vertex $x_1 \in \mathcal{N}_{\rho}$.
\end{remark}

The next lemma gives a lower bound for the probability that a vertex in a given generation is infected if we start with the root infected and condition on the event that the root and such vertices have the same degree. For the next result recall the definition of $C_p(N)$ and $c_p$ from \eqref{eq:constantCp}.

\begin{lemma}\label{lem:pathstar}
	Let $\lambda>0$ and $N\in \IN$ be chosen such that $\tfrac{3}{2}\lambda T<1$ and 
	$ \lambda^2 T^2 N >C$, where $C$ denotes the universal constant given in Lemma \ref{lem:probastar}. Suppose the background process $\bfB$ is started in stationarity.
	Then for every $x\in\mathbf{P}_r$ it follows that 
	\begin{equation*}
		\mathbb{P}^{\{\rho\}}_{\mathcal{T},N}\left(x\not \in \bfC^{*-x}_{s}\ \text{for all}\  s\in [0,\tfrac{S}{2}] \ |  \ \cF_{0,S}\right)\leq  F_{\lambda}(N, r)+R_\lambda(N),
	\end{equation*}
	%for all $t\in [\tfrac{S}{2},S]$ 
	on the event $\cS_{\rho,N}$, where $S$ is defined as in \eqref{eq:S}, $R_\lambda(N)$ as in \eqref{eq:functionR},
	\begin{equation}\label{eq:functionFp}
		\begin{aligned}
			F_{\lambda}(N, r):= 
			\Bigg(1- b_{\lambda}(N)C_p(N) \bigg(\frac{\lambda \nu_1\kappa_1 c_p }{\lambda+\lambda \nu_1+\nu_1+1}\bigg)^{r}\Bigg)^{\left\lfloor S(8r+4T)^{-1} \right\rfloor}
		\end{aligned}
	\end{equation}
	and
	\begin{equation}\label{eq:blambdaN}
		b_{\lambda}(N) := \frac{\lambda \lfloor \delta \lambda T  Np(N,L) \rfloor T}{ \lambda \lfloor \delta \lambda T  Np(N,L)\rfloor + 1)T+1} (1-e^{-\gamma}).
	\end{equation}
\end{lemma}

\begin{proof}
	Let $r\geq 1$ and set $m:=\lfloor S (8r+4T)^{-1}\rfloor$. Recall that we denote the unique path from $\rho$ to $x$ by $\rho=x_0,x_1,\dots,x_r=x$. We begin by noting that
	\begin{align}\label{eq:probastarpath}
		& \IP^{\{\rho\}}_{\mathcal{T},N} \left(x\not \in \bfC_s^{*-x}\ \text{for all}\  s\in [0,\tfrac{S}{2}] \ |  \ \cF_{0,S}\right)  \leq \IP^{\{\rho\}}_{\mathcal{T},N}\Big(\inf_{2\leq k \leq \bar{k}} |\mathcal{G}_{k}'| \leq \delta \lambda T  Np(N,L)\ |  \ \cF_{0,S}\Big) \\
		& \hspace{1mm}\quad + \IP^{\{\rho\}}_{\mathcal{T},N}\Big(\big\{x\not \in \bfC_s^{*-x}\ \text{for all}\  s\in [0,m(8r+4T)]\big\} \cap \Big\{ \inf_{2\leq k \leq \bar{k}} |\mathcal{G}_{k}'|\geq \delta \lambda T  Np(N,L)\Big\} \big|   \cF_{0,S}\Big).\nonumber
	\end{align}
	Thanks to Lemma \ref{lem:goodneighbors} and Remark \ref{rem:excludingvertex} we know that the first probability on the right-hand side is bounded from above by $R_{\lambda}(N)$ on the event $\cS$. Thus it is enough to derive the upper bound for the second probability.  Define the sequence of times $\{t_i,\ 1 \leq i \leq  m+1\}$  such that $t_1=0$ and $ t_{i+1}-t_{i} =  (4r + 2T)$ for $i \in \{1,\dots,m\}$. Now set $D_i:=\big\{ |\mathcal{G}_{\lfloor t_{i}/T\rfloor}'(x_1)|\geq \delta \lambda T  Np(N,L) \}$ for $i\in\{2,\dots ,m\}$, then we see that
	\begin{equation*}
		\left\{ \inf_{2\leq k \leq \bar{k}} |\mathcal{G}_{k}'|\geq \delta \lambda T Np(N,L)\right\}\subset \bigcap_{i=2}^{m} D_i.
	\end{equation*}
	Now set 
	\begin{equation*}
		A_i:=\big\{x_r\not \in \bfC^{*-x}_s\ \text{for all}\  s\in [t_{i},t_{i+1}]\big\}\quad \text{and}\quad t^{*}_i:=\lfloor t_i/T\rfloor T,\quad \text{for} \quad i\in\{1,\dots ,m\}.
	\end{equation*}
	Note that by definition it holds that $t_i^{*}\leq t_i\leq t_i^{*}+T$. The tower property implies that
	\begin{equation*}
		\IP_{\mathcal{T},N}\bigg(\bigcap_{i=1}^{m} A_i \cap \bigcap_{j=2}^{m} D_j \Big |  \cF_{0,S}\bigg)\leq 
		\IE_{\mathcal{T},N}\bigg[\1_{A_1}\prod_{i=2}^{m-1} \1_{A_i\cap D_i} \IP_{\mathcal{T}}(A_m|\cF_{t_m, S})\1_{D_m}\Big |  \cF_{0,S}\bigg],
	\end{equation*}
	where we used that by definition of $\cG'_{\rho,k}$ we have $D_m\in \cF_{t_{m},t^{*}_{m}+2T}\subset \cF_{t_{m},S}$ and also that $t_{m+1}\ge t^*_{m}+2T$. Now, define the following two events 
	\begin{equation*}
		B^1_i:=\{\rho\in \bfC^{*-x}_{t^{*}_i+2T}\} \quad \text{ and } \quad B^2_i:=\big\{x\in \bfC^{*-x}_s\ \text{for some}\  s\in [t^{*}_i+2T,t_{i+1}]\big\}
	\end{equation*}
	with $i\in\{1,\dots ,m\}$, then we see that $B_i^2\cap B^1_i\subset A_i^c$. Therefore,
	\begin{align*}
		\IP_{\mathcal{T},N}(A_m \mid  \cF_{t_m, S})
		&=1-\IP_{\mathcal{T},N}(A_m^c\mid \cF_{t_m, S})
		\leq 1-\IP_{\mathcal{T},N}(B_m^2\cap B_m^{1}\mid \cF_{t_m, S})
		\\
		& \leq 1-\IE_{\mathcal{T},N}\Big[\IP_{\mathcal{T}}(B_m^{2}|\cF_{t^{*}_{m}+2T,S}) \1_{B_m^1}\mid \cF_{t_m, S}\Big],
	\end{align*}
	where we used the tower property, since $\cF_{t_m,t^{*}_{m}+2T,S}\subset \cF_{t^{*}_{m}+2T,S}$, and that $B^1_m\in \cF_{t^{*}_{m}+2T,S}$ by definition. Recall that before Lemma~\ref{lem:probapath} we defined  the infection process $\bfC^{\rho-x}$, which is restricted to the path from $\rho$ to $x$. By definition it follows  that $\bfC^{\rho-x}_t\subset \bfC^{*-x}_t$ for all $t\geq 0$. Let us emphasize at this point that by excluding $x_1$ from the set of good neighbours the path leading from $\rho$ to $x$ and the set $\cG'_{k}$ are defined on disjoint parts of the graphical representation for all $k$ (see Remark \ref{rem:excludingvertex}). This enables us to use the Markov property and monotonicity to obtain that
	\begin{equation*}
		\IP_{\mathcal{T},N}\big(B_m^{2}\mid \cF_{t^{*}_{m}+2T,S}\big) \1_{B_m^1}\geq  \IP^{\{\rho\}}_{\mathcal{T},N}(x_r\in \bfC^{\rho-x}_s \ \text{for some}\ s\in [0,4r])\1_{B_m^1},
	\end{equation*}
	where we have used that $t_{m+1}-t^*_{m}-2T\geq t_{m+1}- t_m-2T= 4r$. Thus, we obtain that
	\begin{align*}
		\IP_{\mathcal{T},N}\big(A_m \mid  \cF_{t_m, S}\big) \leq 1- \IP^{\{\rho\}}_{\mathcal{T},N}(x_r\in \bfC^{\rho-x}_s \ \text{for some}\ s\in [0,4r])\IP_{\cT}\big( B_m^1\mid \cF_{t_m,S}\big).
	\end{align*}
	Moreover, appealing to Lemmas~\ref{lem:probapath} and \ref{InfectedContant} we get 
	\begin{equation*}
		\IP^{\{\rho\}}_{\mathcal{T},N}(x_r\in \bfC_s^{\rho-x} \ \text{for some}\ s\in [0,4r])\geq (1-e^{-\gamma}) \bigg(\frac{\lambda \nu_1\kappa_1 c_p }{\lambda+\lambda \nu_1+\nu_1+1}\bigg)^{r} C_p(N),
	\end{equation*}
	and 
	\begin{equation*}
		\IP_{\mathcal{T},N}(B_m^1 \mid \cF_{t_m,S})\1_{D_m} \ge \IP_{\mathcal{T},N}(B_m^1 \mid \cF_{t_m,S}\,,\, \rho \notin \bfC^{*-x}_{t_m})\1_{D_m} \ge \frac{ \lambda \lfloor \delta \lambda T  Np(N,L)\rfloor T }{ (\lambda \lfloor \delta \lambda T  Np(N,L) \rfloor + 1)T+1}\1_{D_m}.
	\end{equation*}
	These two inequalities now yield that
	\begin{align*}
		&\IE^{\{\rho\}}_{\mathcal{T},N}\bigg[\1_{A_1}\prod_{i=2}^{m-1} \1_{A_i\cap D_i} \IP_{\mathcal{T},N}(A_m\mid\cF_{t_m,S})\1_{D_m}\Big|\cF_{0,S}\bigg]\\
		&\hspace{2cm}\leq \IE_{\mathcal{T},N}\bigg[\1_{A_1}\prod_{i=2}^{m-1}\1_{A_i\cap D_i}\Big|\cF_{0,S}\bigg]\bigg(1-b_{\lambda}(N) C_p^L(N)\bigg(\frac{\lambda \nu_1\kappa_1 c_p }{\lambda+\lambda \nu_1 +\nu_1 +1}\bigg)^{r}\bigg).
	\end{align*}
	Now we can recursively conclude
	\begin{align*}
		\IE^{\{\rho\}}_{\mathcal{T},N}\bigg[\1_{A_1}\prod_{i=2}^{m-1} & \1_{A_i\cap D_i} \IP_{\mathcal{T}}(A_m \mid \cF_{t_m,S})\1_{D_m}\Big|\cF_{0,S}\bigg]\\
		& \leq \IP^{\{\rho\}}_{\mathcal{T},N}(A_1|\cF_{0,S})\bigg(1-b_{\lambda}(N) C_p(N) \bigg(\frac{\lambda \nu_1\kappa_1 c_p }{\lambda+\lambda \nu_1 +\nu_1 +1}\bigg)^{r}\bigg)^{m-1}\\
		& \leq \bigg(1-b_{\lambda}(N)C_p^L(N)\bigg(\frac{\lambda \nu_1\kappa_1 c_p }{\lambda+\lambda \nu_1 +\nu_1 +1}\bigg)^{r}\bigg)^{m},
	\end{align*}
	where in the last step we treated the first time interval $[0,t_2]$ differently, since we know that at time $0$ the root $\rho$ is infected, and thus we do not need to reinfect it in the first time period. Recalling that we chose $m=\lfloor S (8r+4T)^{-1}\rfloor$, we can conclude
	\begin{align*}
		\IP^{\{\rho\}}_{\mathcal{T},N} \Big(\big\{x_r & \not \in \bfC^{*-x}_s\ \text{for all}\  s\in [0,m(8r+4T)]\big\} \cap \Big\{ \inf_{2\leq k \leq \bar{k}} |\mathcal{G}_{k}'|\geq \delta \lambda T Np(N,L)\Big\}\Big|\cF_{0,S}\Big)\\
		&  \leq  \IP_{\mathcal{T},N}\bigg(\bigcap_{i=1}^{m} A_i \cap \bigcap_{j=2}^{m} D_j\Big|\cF_{0,S}\bigg) \leq  \bigg(1-b_{\lambda}(N) C_p(N)\bigg(\frac{\lambda \nu_1\kappa_1 c_p }{\lambda+\lambda \nu_1 +\nu_1 +1}\bigg)^{r}\bigg)^{\lfloor S (8r+4T)^{-1}\rfloor},
	\end{align*}
	and thus the claim is proved.
\end{proof}

\subsection{Proof of Theorem \ref{thm:strongsurvival}}\label{sec:thmstrong}
The aim of this section is to combine the results of the previous sections and show the two statements of Theorem \ref{thm:strongsurvival}. Throughout the section we suppose that $(\bfC,\bfB)$ is a CPDG with $\rho$ initially infected and that the background process is started stationary. Let $L\in \mathbb{N}$ be a fixed constant large enough such that 
\begin{equation}\label{eq:muL}
	\mu_{L}:= \mathbb{E}[\zeta\1_{\{\zeta< L\}}]>1.
\end{equation}

We recall that the first claim of Theorem \ref{thm:strongsurvival} is that if $\alpha\in[0,1)$, $\eta\in (-\infty,\tfrac{1-\alpha}{2})$
and \eqref{eq:assumtionkernel} holds then $\lambda_2=0$. In other words, for every $\lambda>0$ %local
strong survival  takes place with positive probability. The second statement is that if $\alpha \in [0,1)$ and \eqref{eq:assumtionkernel2} holds then $\lambda_2<\infty$, i.e.\ for $\lambda$ large enough strong survival takes place with positive probability. Note that it is sufficient to consider $\eta\geq 0$ in order to show this claim because the case $\eta<0$ is covered by the first claim.
Therefore, we need to consider two parameter regimes. First let
\begin{equation}\label{ass:parameteregime1}
	\alpha\in[0,1), \, \eta\in \big(-\infty,\tfrac{1-\alpha}{2}\big),\, \lambda\in \big(0,\tfrac{1}{2}\big) \, \text{ and }\, N\in \IN\,  \text{ such that }\,  N>C\vee L,\tag{\textbf{R}}
\end{equation}
where $C$ is the universal constant given in Lemma \ref{lem:probastar}.
Note that if we show that there is strong survival for all  $\lambda\in \big(0,\tfrac{1}{2}\big)$, then it follows by monotonicity for all $\lambda> 0$. The second parameter regime is the following
\begin{equation}\label{ass:parameteregime2}
	\alpha\in [0,1),\, \eta\geq 0,\, N\in \IN,\, \lambda=\lambda_{N}=(3T_N)^{-1} \text{ such that }  N>C\vee L.\tag{\textbf{R'}}
\end{equation}
Observe that in both cases we have that $2\lambda T_N<1$ for all $N$ large enough, which is necessary for several results in the previous sections. 

We recall from Section \ref{subsec:HeuristicEvolTree} that our strategy is to apply  \cite[Lemma 6.2]{pemantle1992} to the function	$g(t):=\mathbb{P}^{\{\rho\}}_N(\rho \in \bfC_t),\ t\ge 0.$ In order to do so, we need to find a non-decreasing function $H$ on $[0,\infty)$ such that $H(x)\geq x$ in a neighbourhood of $0$ and show that $g$ and $H$ satisfy
\begin{equation}\label{eq:pemantle}
	\inf_{0\leq t\leq S}g(t) >0 \qquad \text{and}\qquad    g(t)\geq H\Big(\inf_{0\leq s\leq t-S} g(s)\Big), \qquad \text{for all}\quad t>S,
\end{equation}
where $S$ is chosen as in \eqref{eq:S}. The next lemma will be useful in our proof, its proof can be found in \cite[Lemma 2.3]{pemantle1992}.
\begin{lemma}\label{lem:Pemantle}
	Let $M$ be a positive integer-valued random variable and pick $p<\mathbb{E}[M]$. For any $x>0$, let $M_x$ be a $\text{Bin}(M,x)$ random variable. Then there exists an $\epsilon>0$ such that $\mathbb{P}(M_x\ge 1) \ge px \wedge \epsilon$.
\end{lemma}
For simplicity of exposition, we split the proof of Theorem~\ref{thm:strongsurvival} into several lemmas. The first lemma gives us a good candidate for the function $H$.

\begin{lemma}\label{lem:gboundgH}
	Let $\lambda,N,\eta, \alpha$ satisfy either \eqref{ass:parameteregime1} or \eqref{ass:parameteregime2}. Furthermore, assume that there exists $r=r_N$ such that
	\begin{equation}\label{eq:restrikernel}\tag{$\star$}
		\left\lfloor\frac{S}{8r+4T}\right\rfloor C_p(N)> {4} b_{\lambda}^{-1}(N) \bigg(\frac{\lambda \nu_1\kappa_1 c_p }{\lambda+\lambda \nu_1+\nu_1+1}\bigg)^{-r}.
	\end{equation}
	There exists $c\in(0,1]$ and $\epsilon>0$ such that for $N$ large enough, the function
	\begin{equation*}
		H(x):=c N\mu^{r-1}_L\mathbb{P}(\zeta = N) x\wedge\epsilon, \quad \text{for}\quad x\geq 0,
	\end{equation*}
	satisfies that
	\begin{equation*}
		g(t)\ge H\Big(\inf_{0\leq s\leq t-S} g(s)\Big), \quad \text{for}\quad t>S.
	\end{equation*}
\end{lemma}

\begin{proof}
	Let $x\in \mathbf{P}_r$ and $\rho=x_0,x_1,\dots,x_r=x$ be the unique path from $\rho$ to $x$ and also let
	\begin{equation*}
		p_{in}:=	\mathbb{P}_{\mathcal{T}, N}^{\{\rho\}}\big(x\in \bfC_s \ \text{for some}\ s\in [0,S] \mid  \cF_{0,S}\big).
	\end{equation*}
	Recall that the process $(\bfC_t^{*-x})_{t\ge 0}$ is defined via the same graphical representation as $(\bfC_t)_{t\ge 0}$ but the infection is restricted to $\cV^{*-x}=\cN_{\rho}\cup \{\rho,x_1,\dots,x_{r-1},x\}$. We begin by noting that due to the monotonicity of the CPDG, we have
	\[p_{in}\ge \mathbb{P}_{\mathcal{T}, N}^{\{\rho\}}\big(x\in \bfC_s^{*-x} \ \text{for some}\ s\in [0,S] \mid  \cF_{0,S}\big).\]
	Now, Lemma~\ref{lem:pathstar} and monotonicity tells us that 
	\begin{align*}
		\mathbb{P}_{\mathcal{T}, N}^{\{\rho\}}\big(x\in \bfC_s^{*-x} \ \text{for some}\ s\in [0,S] \mid \cF_{0,S}\big)\geq 1- (F_{\lambda}(N, r)+R_{\lambda}(N)),
	\end{align*}
	on the event $\cS_{\rho,N}$, where $S$ is defined as in Lemma~\ref{lem:goodneighbors}, the functions $R_{\lambda}(N)$ and $F_{\lambda}(N, r)$ are defined in \eqref{eq:functionR} and \eqref{eq:functionFp}, respectively. Then, using the inequality $1-x\leq e^{-x}$, observe that \eqref{eq:restrikernel} forces $F_{\lambda}(N, r)$ to be at most $e^{-4}$, i.e.
	\begin{align*}
		& F_{\lambda}(N, r)=\left(1-b_{\lambda}(N)C_p(N)\left(\frac{\lambda \nu_1 \kappa_1 c_p }{\lambda+\lambda \nu_1+\nu_1+1}\right)^r \right)^{\lfloor S(8r+4T)^{-1}\rfloor}\\ & \hspace{3.5cm} \leq  \exp\left\{-b_{\lambda}(N) C_p(N)\left(\frac{\lambda \nu_1 \kappa_1 c_p }{\lambda+\lambda \nu_1+\nu_1+1}\right)^r  \left\lfloor \frac{S}{8r+4T}\right\rfloor \right\}  < e^{-4}.
	\end{align*}
	Furthermore, the asymptotic result derived in Lemma~\ref{lemma:kickstart} implies that $R_{\lambda}(N)\to 0$ if $\eta>0$,  $R_{\lambda}(N)\to 1-e^{-2}$ if $\eta<0$ and $R_{\lambda}(N)\to 1-e^{-2(1+\nu_2)^{-1}}$ if $\eta=0$, as $N\to \infty$. Hence, there exists an $N_1\in \mathbb{N}$ sufficiently large such that $R_{\lambda}(N)\leq 1-e^{-3}$ for all $N\geq N_1$. Thus, it follows that
	\begin{equation}\label{eq:pin}
		p_{in} \geq b_1,\qquad \text{where}\qquad b_1:= e^{-3}-e^{-4}>0.
	\end{equation}
	Before we proceed we need to introduce some further notation. Given a realisation of the underlying tree $\cT$ we denote by $\cS_{x,N}^{t-S}$ the event that $x\in \mathbf{P}_r$ is a stable star starting from time $t-S$, where $t\geq S$. Furthermore, we explicitly exclude the neighbour in generation $r-1$, but we also omit this in the notation to avoid clutter.
	In the following, we will condition on the event 
	\begin{equation*}
		A_{t-S}:=\bigcup_{x\in \mathbf{P}_r}\{x\in \bfC_{t-S}\}\cap \cS_{x,N}^{t-S},
	\end{equation*}
	i.e.\ that $x\in \bfC_{t-S}$ for some $x\in \mathbf{P}_r$ which is a stable star starting at time $t-S$. We deduce for $t>S$
	\begin{equation*}
		g(t)=\mathbb{P}_N^{\{\rho\}}(\rho \in \bfC_t)\geq H_1(t)H_2(t),
	\end{equation*}
	where
	\begin{equation*}
		H_1(t)=\mathbb{P}_N^{\{\rho\}}(A_{t-S}) \quad \text{ and } \quad H_2(t)= \mathbb{P}_N^{\{\rho\}}\big(\rho \in \bfC_t \ | \ A_{t-S}  \big).
	\end{equation*}
	Hence, the next goal is to establish lower bounds for the functions $H_1$ and $H_2$.\\
	
	\textbf{ Lower bound for} $H_1$. Recall that we denote by $\zeta$ a generic random variable distributed according to the offspring distribution of $\cT$. Note that the expected number of vertices in $\mathbf{P}_r$ can be bounded from below as follows
	\begin{equation}\label{eq:expPr}
		\mathbb{E}_N\big[|\mathbf{P}_r|\big] \ge \IP(\zeta< L)N \mu^{r-1}_{L}  \mathbb{P}(\zeta = N) ,
	\end{equation} 
	where we used symmetry and monotonicity. Let $M_r^S$ be the number of vertices in $\mathbf{P}_r$ infected before time $S$, that is 
	\[M_r^S= \sum_{x\in \mathbf{P}_r}\1_{\{x \in \bfC_s \ \text{for some}\ s\in [0,S]\}}.\]
	Then, by using the fact that $\bfC_s^{*-x} \subset \bfC_s$ for all $s\ge 0$, followed by the tower property, \eqref{eq:pin} and Remark \ref{rem:excludingvertex}, we deduce
	\begin{align*}
		\IE_{N}\big[M_r^S\big] \ge   \IE_{N}\left[\sum_{x\in \mathbf{P}_r} \mathbb{P}_{\mathcal{T},N}^{\{\rho\}}(x \in \bfC_s^{*-x} \ \text{for some}\ s\in [0,S] \mid \cF_{0,S})\1_{\cS_{\rho,N}}\right]  
		\ge b_1 \mathbb{E}_{N}\bigg[\sum_{x\in \mathbf{P}_r} \1_{\cS_{\rho,N}}\bigg].
	\end{align*}
	Now we observe that from \eqref{eq:expPr}, we get 
	\begin{align*}
		\IE_{N}\bigg[\sum_{x\in \mathbf{P}_r}\1_{\cS_{\rho,N}}\bigg] &\ge   \mathbb{E}_{N}\big[|\mathbf{P}_r|\big] \mathbb{P}_{N}(\cS_{\rho,N})   \ge  \IP(\zeta< L) N \mu^{r-1}_L  \mathbb{P}(\zeta= N) \mathbb{P}_{N}( \cS_{\rho, N}),
	\end{align*}
	In addition, by Lemma \ref{lemma:boundgoodnb} we have that $\mathbb{P}_{N}(\cS_{\rho, N}) \ge 1-e^{-c_L N p(N,L)}$. Thus there exists $b_2>0$ such that for $N$ large enough $\mathbb{P}_{N}(\cS_{\rho, N})\ge b_2$. In other words, we deduce
	\begin{equation*}
		\IE_{N}\big[M_r^S\big] \ge b_1 b_2 \IP(\zeta < L) N \mu^{r-1}_L  \mathbb{P}(\zeta = N).
	\end{equation*}
	Let us define, for $t>2S$,
	\begin{equation*}
		\chi(t):=\inf_{0\leq s \leq t-S} g(s).
	\end{equation*}
	Now, ignore all the infections of $x\in \mathbf{P}_r$ by its parent except the first infection. Then, the contact process on the subtrees rooted at vertices $x\in\mathbf{P}_r$ which are infected at some
	time $s<S$ will evolve independently from time $s$ to time $t - S$, since by the branching property of the BGW tree and using that the background process is stationary. Then the vertex $x$ will be infected at time $t-S$ with probability at least $\chi(t)$.
	
	Furthermore, the event that $x\in \mathbf{P}_r$ is a stable star from $t-S$ onwards, i.e.\ that $\cS_{x,N}^{t-S}$ is satisfied, is independent of everything that happened before time $t-S$ and by symmetry it follows that $b_2$ is a lower bound for the probability of $x\in \mathbf{P}_r$ being a stable star for every realisation $\cT$, i.e.\ $\IP_{\cT,N}(\cS^{t-S}_{x,N})\geq b_2$.
	
	If we denote by $\widetilde{M}_r^{t-S}$ the random number of  vertices in $\mathbf{P}_r$ which are infected at time $t-S$ and are stable stars starting at $t-S$, we can conclude that $\widetilde{M}_r^{t-S}$  stochastically dominates a random variable $M_\chi$ that has distribution $\text{Bin}(M_r^S, b_2\chi(t))$. Appealing to Lemma \ref{lem:Pemantle} there exists $\epsilon_1>0$ such that
	\begin{equation*}
		\mathbb{P}(M_\chi\geq 1) \geq 2^{-1} b_1 b_2^{2} \IP( \zeta<L) N \mu^{r-1}_L  \mathbb{P}(\zeta= N)  \chi(t)\wedge \epsilon_1.
	\end{equation*}
	Moreover, since $\widetilde{M}_r^{t-S}$ dominates the random variable $M_\chi$, we obtain that, for $t>2S$
	\begin{equation}
		H_1(t)\geq \mathbb{P}(\widetilde{M}_r^{t-S} \geq 1)\geq \mathbb{P}(M_\chi\geq 1),
	\end{equation}
	which implies for $t>2S$
	\begin{equation}
		H_1(t)\geq 2^{-1} b_1 b_2^{2} \IP(\zeta<L) N \mu^{r-1}_L  \mathbb{P}(\zeta= N)  \chi(t)\wedge \epsilon_1.
	\end{equation}
	\textbf{Lower bound for $H_2$.} 
	Let $t>S$. Note that
	\begin{equation*}
		H_2(t)= \mathbb{P}_N^{\{\rho\}}\big(\rho \in \bfC_t \ | \ A_{t-S}  \big)
		\geq h(t)\widehat{h}(t),
	\end{equation*}
	where
	\begin{equation*}
		\begin{aligned}
			h(t) &= \mathbb{P}_N^{\{\rho\}}\big(\rho \in \bfC_s \ \text{for some}\ s\in [t-S, t-1-2T]\ | \ A_{t-S} \big),\\
			\widehat{h}(t) &= \mathbb{P}_N^{\{\rho\}}\big(\rho \in \bfC_t \ |\ \{\rho \in \bfC_s \ \text{for some} \ s\in [t-S,t-1-2T]  \} \cap  A_{t-S} 
			\big).
		\end{aligned}
	\end{equation*}
	Now, we need to find lower bounds for $h$ and $\widehat{h}$.
	
	For that purpose we need to introduce a bit more notation. Let $\cF_{s}$ be the $\sigma$-algebra which is generated by the graphical construction up to time $s$. Furthermore, we denote by $\cF^r$ the $\sigma$-algebra generated by $\Delta^{\text{rec}}_z$ and $\Delta^{\text{up}}_{\{y,z\}}$ for all $z\in \cV_{r+1}$, i.e.\ vertices in generation $r$, and $\{y,z\}\in \cE$ with $d(\rho,y)= r$, i.e.\ edges between generation $r$ and $r+1$. In words this means that $\cF^r$ contains all the information about recoveries in generation $r+1$ and update events between generation $r$ and $r+1$ for the whole time period $0$ to $\infty$. Finally we define the $\sigma$-algebra $\cF_{s}^r:=\sigma(\cF_s\cup \cF^r)$ for $s\geq 0$. 
	
	We begin with the lower bound for $h$, where we get by the tower property that
	\begin{align*}
		h(t)&\geq
		\frac{\mathbb{E}_N\big[\mathbb{P}_{\mathcal{T},N}\big(\rho \in \bfC_s\ \text{for some}\ s\in [t-S, S-1-2T] \mid \mathcal{F}_{t-S}^{r}\big)\1_{A_{t-S}}
			\big]}{ \mathbb{P}_N(A_{t-S})}.
	\end{align*}
	Recall that on the event $A_{t-S}$ we know that there exists an $x\in \mathbf{P}_r$ which is a stable star from $t-S$ onwards and $x\in \bfC_{t-S}$, i.e.\ $x$ is infected at time $t-S$. This means we are again in a situation that we can consider the restriction of the graphical representation to a path of length $r$ between $\rho$ and the neighbourhood of $x$. Then, it follows analogously as in the proof of Lemma~\ref{lem:pathstar} and \eqref{eq:pin} that
	\begin{equation*}
		\mathbb{P}_{\cT,N}\big(\rho \in \bfC_s \ \text{for some}\ s\in [t-S, t-1-2T] \mid \cF^{r}_{t-S}\big) \1_{A_{t-S}}\geq b_1 \1_{A_{t-S}},
	\end{equation*}
	where we used that the path is between generation $0$ and $r$, and thus independent of everything in generation $r+1$ and higher. Therefore, we can conclude that $h(t)\geq b_1 $.
	
	Next we derive a lower bound for $\widehat{h}$.	Let us now introduce the stopping time for the contact process $(\bfC_t)_{t\ge 0}$,
	\begin{equation*}
		\tau = \inf\{u\geq t-S:\  \rho \in \bfC_u\}.
	\end{equation*}
	The aim is to deduce a lower bound for the probability of the root $\rho$ being infected at time~$t$. Thus, we can use monotonicity to restart the infection process after time $\tau$ with only~$\rho$ being infected. Furthermore, after $\tau$ we restrict infections to $\rho$ and its neighbourhood, i.e.\ we restrict the graphical representation to $\cN_{\rho}\cup \{\rho\}$ from $\tau$ onwards, analogously as in the definition of $\bfC^{*}$. We denote 
	$$\mathcal{B}_\tau:=\left\{\inf_{ k \in r(\tau, T, t)} |\mathcal{G}'_{\rho,k}|\geq \delta \lambda T N p(N,L) \right\} \, \text{with} \,\, r(\tau, T, t):=  \left\{\bigg\lfloor\frac{\tau}{T}\bigg\rfloor+4, \dots, \bigg\lceil \frac{t-1}{T} \bigg\rceil-2\right\}.$$
	Note that on $\mathcal{B}_{\tau}$ we have at least $\delta \lambda T N p(N,L)$ good neighbours in $[\tau+4T, t -1]$. 
	Further, define $\mathcal{I}$ to be the event that at least one of the $ \delta \lambda T N p(N,L)$ neighbours which are infected at time $t-1$ infects $\rho$ at a time in $[t-1,t]$ before this neighbour recovers and 
	\begin{equation*}
		\mathcal{R}_\rho:= \{\rho\ \text{does not recover in}\ [t-1,t]\}.
	\end{equation*}
	With this notation we can estimate $\widehat{h}$ as follows, 
	\begin{equation}\label{eq:hhat}
		\widehat{h}(t)\geq \mathbb{P}\big(\mathcal{I}\cap\mathcal{R}_\rho \cap  \mathcal{B}_\tau \mid  \{\tau \leq t-1-2T\}\big),
	\end{equation}
	where we used that the event $\mathcal{I}\cap\mathcal{R}_\rho \cap  \mathcal{B}_\tau$ is independent of $\bigcup_{x\in \mathbf{P}_r} \cS_{x,N}^{t-S}$, since these events are defined on disjoint parts of the graphical representation.	On the other hand, recall that we denote by $\mathcal{F}_t$ the $\sigma$-algebra, which is generated by the graphical construction up to time $t$. Thus, using that $\mathcal{B}_{\tau}\cap \{\tau \leq t-1-2T\} \in \mathcal{F}_{t-1}$ and the independence of the infection and recovery events, we obtain 
	\begin{eqnarray*}
		\mathbb{P}\big(\mathcal{I}\cap\mathcal{R}_\rho \cap  \mathcal{B}_\tau \cap \{\tau \leq t-1-2T\}\big) & = &\mathbb{E}\Big[\mathbb{P}(\mathcal{I}\cap\mathcal{R}_\rho\ | \mathcal{F}_{t-1}) \1_{\mathcal{B}_{\tau} \cap  \{\tau \leq t-1-2T\}}\Big] \\ & =&  \mathbb{E}\Big[\mathbb{P}(\mathcal{I}\ | \mathcal{F}_{t-1}) \mathbb{P}(\mathcal{R}_{\rho}\ | \mathcal{F}_{t-1}) \1_{\mathcal{B}_{\tau}\cap\{  \tau \leq t-1-2T\}}\Big].
	\end{eqnarray*}
	Now we find a lower bound for the two conditional probabilities on the right-hand side of the previous equality. First note that $\mathbb{P}(\mathcal{R}_{\rho}\ | \mathcal{F}_{t-1}) \geq e^{-1}$.  Recall that $\cN_\rho$ are the neighbours of $\rho$ with degree less than $L$, i.e.\ $\de_{x_i}\leq L$ for all $i \in \{1,\dots,\Tilde{N}\}$. Denote $\Tilde{N}:=\lfloor\delta \lambda T N p(N, L)\rfloor$ and by $x_1, \dots, x_{\Tilde{N}}\in \cN_\rho $  the $\Tilde{N}$ neighbours that are infected at time $t-1$. 
	Note that we can rewrite the event $\mathcal{I}$ as follows
	\begin{equation*}
		\mathcal{I} = \bigcup_{i=1}^{\Tilde{N}} \{x_i\ \text{infects}\ \rho \ \text{at a time in}\ [t-1,t]\}.
	\end{equation*}
	Further, 
	\begin{equation*}
		\mathcal{I}^c  \subset \bigcap_{i=1}^{\Tilde{N}} \{\Delta^{\text{rec}}_{x_i}\cap [t-1,t]\neq \emptyset\}\cup \{\Delta^{\inf}_{\{x_i,\rho\}}\ \cap[t-1,t]=\emptyset\}.
	\end{equation*}
	Thus,
	\begin{equation*}
		\mathbb{P}(\mathcal{I} \mid \mathcal{F}_{t-1}) = 1-\IP(\mathcal{I}^c)\geq  1- (1-e^{-1}(1-e^{-\lambda}))^{\Tilde{N}}=: 1 - a_{\lambda}(N),
	\end{equation*}
	where we used that $\mathcal{I}$ only depends on the graphical representation in the time interval $[t-1,t]$.  	
	
	On the other hand, by monotonicity we get 
	\begin{equation*}
		\begin{split}
			\mathbb{P}\big(\mathcal{B}_{\tau} \mid \mathcal{F}_{\tau}\big) & \geq \mathbb{P}\left(\inf_{k \in  \hat{r}(\tau, T, t)} |\mathcal{G}_{\rho, k}'| \geq  \delta \lambda T N p(N,L) \ \Big|\Big. \  \mathcal{F}_{\tau}\right),
		\end{split}
	\end{equation*}
	where 
	$$\hat{r}(\tau, T, t):=  \left\{\bigg\lceil\frac{\tau+2T}{T}\bigg\rceil, \dots, \bigg\lceil \frac{S+\tau}{T} \bigg\rceil\right\}.$$
	Now, appealing to the strong Markov property and Lemma \ref{lem:goodneighbors}, we deduce that
	\begin{equation*}
		\begin{split}
			\mathbb{P}\big(\mathcal{B}_{\tau} \mid \mathcal{F}_{\tau}\big) & \ge \mathbb{P}\left(\inf_{2 \leq k\leq \bar{k}} |\mathcal{G}_{\rho, k}'| \geq  \delta \lambda T N p(N,L) \right) 
			\\ & \ge  \IP\left(\IP_{\cT, N}^{\{\rho\}}\Big(\inf_{2\leq k \leq \bar{k}} |\mathcal{G}_{\rho, k}'| \leq \delta \lambda T  Np(N,L)\,  \Big| \, \cF_{0,S}\Big)\1_{\mathcal{S}_{\rho,N}}\right)
			\\ & \geq (1 - R_{\lambda}(N))\mathbb{P}(\mathcal{S}_{\rho,N}) \ge (1 - R_{\lambda}(N)) b_2.
		\end{split}
	\end{equation*}
	%  \mar{\texttt{MS: Is there a reason why here still $x_1$ appears?}\\}
	Putting everything together yields
	\begin{align*}
		\mathbb{P}\big(\mathcal{I}\cap\mathcal{R}_\rho \cap  \mathcal{B}_\tau \cap \{\tau \leq t-1-2T\}\big) & \geq  (1 - a_{\lambda}(N))e^{-1} \mathbb{E}\Big[\mathbb{P}(\mathcal{B}_\tau \ | \mathcal{F}_{\tau} )\1_{\{\tau\leq t-1-2T\}}\Big] \\ & \geq    (1 - a_{\lambda}(N)) e^{-1} (1-R_{\lambda}(N))b_2\mathbb{P}(\tau \leq t-1-2T).
	\end{align*}
	Hence, substituting back into  \eqref{eq:hhat} we deduce
	\begin{equation*}
		\widehat{h}(t)  \geq  (1 - a_{\lambda}(N)) e^{-1} (1-R_{\lambda}(N)) b_2.
	\end{equation*}
	Note that  $a_{\lambda}(N)\to 0$ as $N\to \infty$ for all $\eta \in \mathbb{R}$. Further, we recall that  $R_{\lambda}(N) \to 0$ if $\eta> 0$, $R_{\lambda}(N)\to 1-e^{-2}$ if $\eta<0$ and $R_{\lambda}(N)\to 1-e^{-2(1+\nu_2)^{-1}}$ if $\eta=0$, as $N\to \infty$. Thus, we find a $N_2\geq N_1$ and a constant $b_3$ such that for $t>S$  we have $\widehat{h}(t)\geq b_3$ for all $N\geq N_2$. Finally, the lower bounds for $h$ and $\widehat{h}$ give that for $t>S$
	\begin{equation*}
		H_2(t)\geq h(t)\widehat{h}(t) \geq b_1 b_2 b_3.
	\end{equation*}
	This yields that, for $t>2S$,
	\begin{equation*}
		g(t)\geq H_1(t)H_2(t) \geq  c N \mu^{r-1}_L\mathbb{P}(\zeta = N) \chi(t)\wedge \epsilon_1 ,
	\end{equation*}
	where $c:=2^{-1}b_1^2 b_2^2 b_3 \mathbb{P}(\zeta< L)$. Under the assumption that $r=r_N$ satisfies Condition~\eqref{eq:restrikernel}, we have established the following lower bound  for $g$,
	\begin{equation*}
		g(t)\ge \left\{ \begin{array}{lcc}
			c N \mu^{r-1}_L\mathbb{P}(\zeta = N) \chi(t)\wedge \epsilon_1   &    & t>2S\\
			\\
			\inf_{0\leq s\leq 2S} g(s),   &  &  S\leq t\leq 2S.
		\end{array}
		\right.   
	\end{equation*}
	This together with the fact
	\begin{equation*}
		\inf_{0\leq s\leq 2S} g(s)\geq \mathbb{P}_N^{\{\rho\}}(\rho \ \text{never recovers in}\ [0,2S])= e^{-2S}>0,
	\end{equation*}
	implies that there exists $\epsilon>0$ such that
	\begin{equation*}
		g(t)\ge  c N \mu^{r-1}_L\mathbb{P}(\zeta = N) \chi(t)\wedge \epsilon\quad \text{for } t>S
	\end{equation*}
	and the desired result follows.
\end{proof}

The following lemma tells us that under the Condition \eqref{eq:assumtionkernel} we can find a natural number $r_N$ such that Condition \eqref{eq:restrikernel} holds for $N$ large enough.

\begin{lemma}\label{lemma:conAandr}
	Suppose $\lambda,N,\eta, \alpha$ satisfy \eqref{ass:parameteregime1} and that Condition~\eqref{eq:assumtionkernel} holds. Denote by 
	\begin{equation}\label{eq:r(N)}
		r=r_N=\left\lceil \frac{-\log\big(\mu^{-1}_L c N\mathbb{P}(\zeta = N)\big)}{\log \mu_L} \right\rceil,
	\end{equation}
	where $c$ is chosen as in Lemma~\ref{lem:gboundgH}. Then there exists a subsequence $(N_n)_{n\geq 0}$ with $N_n\to \infty$ as $n\to \infty$ such that $r_{N_n}\to \infty$ as $n\to \infty$ and  Condition \eqref{eq:restrikernel} is satisfied for $n$ large enough.
\end{lemma}

\begin{proof}
	By Assumption~\eqref{eq:assumtionkernel}, we have that there exists a subsequence $(N_n)_{n\geq 0}$ such that $N_n\to \infty$ as $n\to \infty$ and 
	\begin{equation}\label{eq:assumptionNn}
		\lim_{n\to \infty} \frac{\log \mathbb{P}(\zeta = N_n)}{ N_n^{1-\alpha-2(\eta\vee 0)}} = 0.
	\end{equation}
	For the rest of this proof we write $T_n=T_{N_n}$ and $r_n=r_{N_n}$. Note that for this sequence it holds  that $N_n\mathbb{P}(\zeta = N_n) \to 0$ as $n\to \infty$, which implies that $r_n \to \infty$ as $n\to \infty$. 
	
	The proof is completed as soon as we can show that Assumption \eqref{eq:assumtionkernel} implies for $N_n$ sufficiently large that the following inequality holds
	\begin{equation}\label{eq:lambdar}
		\bigg(\frac{ \lambda \nu_1 \kappa_1 c_p}{\lambda+\lambda \nu_1+\nu_1+1 }\bigg)^{r_n} C_p(N_n) >  \frac{4(8r_n+4T_n)(1+\nu_2 N_n^\eta)}{b_{\lambda}(N_n) e^{\delta \lambda^2 T^2_n N_np(N_n,L)/4} },
	\end{equation}
	where we can avoid the ceiling and the floor in the factors $S(8r_n+4T_n)^{-1}$ and $\bar{k}$ as $\lceil S(8r_n+4T_n)\rceil / S(8r_n+4T_n) \to 1$ as  $n\to \infty$ (and similarly for $\bar{k}$ and $r_n$). Therefore, \eqref{eq:lambdar} is equivalent to \eqref{eq:restrikernel} for $N_n$ large enough. Now, we see that \eqref{eq:lambdar} is equivalent to 
	\begin{equation}\label{eq:loglambda}
		\log \bigg(\frac{\lambda \nu_1 \kappa_1 c_p}{\lambda+\lambda \nu_1+\nu_1+1 }\bigg) > \frac{1}{r_n}\log \left( \frac{4(8r_n+4T_n)(1+\nu_2 N_n^\eta)}{ C_p(N_n) b_{\lambda}(N_n) e^{\delta \lambda^2 T^2_n N_np(N_n,L)/4}}\right).
	\end{equation}
	Further, we observe that
	\begin{equation}
		\lim_{n\to \infty} \frac{\log(8r_n+4T_n)}{r_n}=0 \qquad \text{and}\qquad \lim_{N\to\infty} \log(b_{\lambda}(N_n)) = \log(1-e^{-\gamma}),
	\end{equation}
	where in the first limit we have used that $T_n= (1+\nu_2 N_n^{\eta})^{-1}\in [0,1]$ for all $n\geq0$ and $r_n \to \infty$  as $n\to \infty$. For the second limit we have used that there exist $C\ge c>0$ such that 
	\begin{equation}\label{eq:ineqNp}
		c N_n^{1-\alpha-2(\eta\vee0)} \le T^2_n N_np(N_n,L) \le C  N_n^{1-\alpha-2(\eta\vee0)},
	\end{equation}
	which follows for $N_n$ large enough by \eqref{eq:kernelsgen} and  \eqref{ass:parameteregime1}. This yields that 
	\begin{eqnarray*}
		\lim_{N\to \infty} \frac{\log  (32r_n+16T_n) - \log (b_{\lambda}(N_n))}{r_n} =0.
	\end{eqnarray*}
	Therefore, when $N_n$ is large enough \eqref{eq:loglambda} is implied by 
	\begin{equation}\label{eq:loglambda2}
		\log \bigg(\frac{\lambda \nu_1 \kappa_1 c_p}{\lambda+\lambda \nu_1+\nu_1+1 }\bigg) > \frac{1}{r_n}\left(\log  \big(C_p(N_n)^{-1}(1+\nu_2 N_n^\eta)\big) - \log (e^{\delta \lambda^2 T^2_n N_np(N_n, L)/4})\right).
	\end{equation} 
	Furthermore, recall $C_p(N_n)$ from \eqref{eq:constantCp}, then we see that for $N_n$ large enough this is equivalent to
	\begin{equation}\label{eq:loglambdaineq}
		\log \bigg(\frac{\lambda \nu_1 \kappa_1 c_p}{\lambda+\lambda \nu_1+\nu_1+1 }\bigg) > \frac{1}{r_n}\left((2\alpha +2|\eta|-(\eta\vee 0)) \log (N_n) - \delta \frac{\lambda^2}{4} T^2_n N_np(N_n, L)\right).
	\end{equation} 
	Now, we will see that \eqref{eq:assumptionNn} implies \eqref{eq:loglambdaineq}. Suppose that $K>0$ is sufficiently large so that 
	\begin{equation}\label{eq:lambdabound}
		\frac{(1+\nu_1 ) e^{-K}}{ \nu_1 \kappa_1 c_p - (1+\nu_1) e^{-K}} < \lambda. 
	\end{equation} 
	Since $\alpha<1$, $\eta<\frac{1-\alpha}{2}$ and by~\eqref{eq:ineqNp}, we can assume that $N_n$ is sufficiently large 
	so that 
	\[ (2\alpha +2|\eta|-(\eta\vee 0))\log N_n  \leq \frac 12 \delta \frac{\la^2}{4} T_n^2 N_n p(N_n,L) . \]
	Moreover, by  \eqref{eq:assumptionNn}, we have
	\[\limsup_{n\to \infty}\frac{T_n^2 N_n p(N_n,L) }{r_{n}} =\infty. \]
	Then we can choose $N_n$ large enough such that
	\[ \frac{T_n^2 N_n p(N_n,L) }{r_{n}} \geq 8\frac{K}{\delta \la^2 } .\]
	In other words, we have for this choice of $N_n$ that 
	\[\begin{aligned} \frac{1}{r_{n}} \Big( (2\alpha +2|\eta|-(\eta\vee 0))\log (N_n)  - \delta \frac{\lambda^2}{4} T_n^2 N_n p(N_n,L)) \Big) 
		& \leq - \frac{1}{2} \frac{1}{r_{n}} \delta \frac{\lambda^2}{4} T_n^2 N_n p(N_n,L) \\
		& \leq -  K < \log \Big( \frac{ \la \nu_1 \kappa_1 c_p}{\la + \la \nu_1 + \nu_1 + 1}\Big), 
	\end{aligned}
	\]
	where the last inequality is implied by 
	the choice of $K$ in~\eqref{eq:lambdabound}.
	Therefore,~\eqref{eq:loglambdaineq} holds, which concludes the proof.  
\end{proof}

\begin{lemma}\label{lemma:conAandr2}
	Suppose $N, \lambda_N,\eta, \alpha$ satisfy \eqref{ass:parameteregime2} and assume that Condition~\eqref{eq:assumtionkernel2} holds.
	Then $r=r_N$ defined as in \eqref{eq:r(N)} satisfies Condition \eqref{eq:restrikernel} for some large $N$.
\end{lemma}

\begin{proof} 
	Similar as in the proof of the previous lemma by Assumption~\eqref{eq:assumtionkernel2}, we have that there exists a subsequence $(N_n)_{n\geq 0}$ such that $N_n\to \infty$ as $n\to \infty$ and 
	\begin{equation}\label{eq:boundonP2} \lim_{n\to \infty} \frac{\log \mathbb{P}(\zeta = N_n)}{ N_n^{1-\alpha}} = 0.\end{equation}
	Again we write the rest of this proof $\lambda_{n}=\lambda_{N_n}$, $T_n=T_{N_n}$ and $r_n=r_{N_n}$. It follows that for this sequence it holds that $N_n\mathbb{P}(\zeta = N_n) \to 0$ as $n\to \infty$, which implies that $r_n \to \infty$ as $n\to \infty$.

	We have to show that Assumption \eqref{eq:assumtionkernel2} implies for $N$ sufficiently large that the following inequality holds
	\begin{equation}\label{eq:boundA}
		\log\bigg(\frac{ \nu_1 \kappa_1 c_p \lambda_{n}}{\lambda_{n}+\nu_1 \lambda_{n} +\nu+1 }\bigg)> \frac{1}{r_n}\log\bigg( \frac{4(8r_n+4T_n)(1+\nu_2 N_n^\eta)}{C_p(N_n)b_{\lambda}(N_n) e^{\delta N_np(N_n,L)/4} }\bigg).
	\end{equation}
	Similarly as in Lemma \ref{lemma:conAandr}, we deduce that
	\begin{eqnarray*}
		\lim_{n\to \infty} \frac{\log  (32r_n+16T_n) - \log (b_{\lambda}(N))}{r_n} =0.
	\end{eqnarray*}
	Thus, when $N_n$ is large enough \eqref{eq:boundA} is implied by 
	\begin{equation*}\label{eq:loglambda3}
		\log \bigg(\frac{ \nu_1\lambda_{n} \kappa_1 c_p}{\lambda_{n}+\nu_1\lambda_{n} +\nu_1+1 }\bigg) > \frac{1}{r_n}\left(\log  (C_p(N_n)^{-1}(1+\nu_2 N_n^\eta)) - \log (e^{\delta N_np(N_n,L)/4})\right).
	\end{equation*}
	Note that the right-hand side does not depend on $\lambda_{n}$. Since  $\lim_{n\to \infty} \lambda_{n}=\infty$, we require for this equality to hold that 
	\begin{equation}\label{eq:boundv}
		\log \bigg(\frac{1+\nu_1}{c_p \nu_1}\bigg) < \lim_{n\to \infty}\frac{\log (\mu_L)\left(\log (N_n^{2\alpha}(1+\nu_2 N_n^\eta))-(\delta N_n^{1-\alpha}/4))\right)}{\log(\mu^{-1}_L c N_n \mathbb{P}(\zeta = N_n))}.
	\end{equation}
	We point out here that we are assuming that $\mu_L>1$, and thus $\log(\mu_L)>0$. Now the fact
	\[\lim_{n\to \infty}\frac{\log(\mu^{-1}_L c N_n)}{\log (\mu_L)\left(\log (N_n^{2\alpha}(1+\nu_2 N_n^\eta))-(\delta N_n^{1-\alpha}/4)\right)}=0,\]
	guarantees together with~\eqref{eq:boundonP2} that \eqref{eq:boundv} holds for $N_n$ large enough since the limit diverges. In other words, $r=r_{N_n}$ defined as in \eqref{eq:r(N)} satisfies Condition \eqref{eq:restrikernel} for some large $N_n$. 
\end{proof}

We are now ready to deduce Theorem \ref{thm:strongsurvival}. 

\begin{proof}[Proof of Theorem \ref{thm:strongsurvival}]
	Let us begin with the first claim. Assume that Condition \eqref{eq:assumtionkernel} holds, $\alpha<1$ and $\eta\in (-\infty,\tfrac{1-\alpha}{2})$. Fix $\lambda\in\big(0,\tfrac{1}{2}\big)$. Next define $r_N$ as in \eqref{eq:r(N)}. By Lemma~\ref{lemma:conAandr}, we know $r_N$ satisfies Condition~\eqref{eq:restrikernel} for  some $N$ large enough. Now, we are able to apply Lemma~\ref{lem:gboundgH} to deduce that the non-negative and non-decreasing function
	\begin{equation*}
		H(x)=c N\mu^{r-1}_L\mathbb{P}(\zeta = N) x\wedge\epsilon, \quad \text{for}\quad x\geq 0
	\end{equation*}
	satisfies
	\begin{equation*}
		g(t)\ge H\left(\inf_{0\leq s\leq t-S} g(s)\right), \quad \text{for}\quad t>S,
	\end{equation*}
	where $g(t)=\mathbb{P}^{\{\rho\}}(\rho \in \bfC_t \mid \zeta =N)$. Furthermore, we observe that
	\begin{equation*}
		\inf_{0\leq s\leq S} g(s)\geq \mathbb{P}(\rho \ \text{never recovers in}\ [0,S])= e^{-S}>0.
	\end{equation*}
	In addition, by definition of $r_N$, we have $c N\mu^{r-1}_L\mathbb{P}(\zeta = N) \geq 1$, which implies $H(x)\geq x$ in a neighborhood of 0. Thus appealing to Lemma~2.4 in \cite{pemantle1992}, we conclude that 
	$\liminf_{t\to \infty}g(t)>0,$
	and thus 
	\[\liminf_{t\to\infty} \IP^{\{\rho\}}\big( \rho \in \bfC_{t} \big) \ge  \mathbb{P}(\zeta = N)\liminf_{t\to\infty} g(t) > 0.  \]
	It follows from the reverse Fatou lemma (which applies as indicator functions are bounded by $1$)  that
	\[ \IP^{\{\rho\}}\big( \forall s \geq 0 \, \exists t \geq s \, : \, \rho \in \bfC_{t} \big) = \IE\Big[ \limsup_{t \rightarrow \infty} \1_{\{ \rho \in \bfC_t \}} \Big] \geq
	\limsup_{t\to\infty} \IP^{\{\rho\}}\big( \rho \in \bfC_{t} \big) > 0 . \]
	
	Thus, as $\lambda >0$ was arbitrary, it follows from Lemma~\ref{lemma:constantscritical} that $\lambda_2=0$.
	
	For the second part, we fix $\eta>0$ and $\alpha< 1$ and assume that Condition \eqref{eq:assumtionkernel2} holds. Now, 
	for $N \in \mathbb{N}$, we choose $\lambda_{N} = (3T_N)^{-1}$ and $r=r_N$ as in \eqref{eq:r(N)}. By Lemma~\ref{lemma:conAandr2}, we know that for some large $N$ the term
	$r_N$ satisfies Condition~\eqref{eq:restrikernel}. Therefore, using the same arguments as before we can conclude that for this choice of the infection rate, $\liminf_{t\to\infty} \IP^{\{\rho\}}\big( \rho \in \bfC_{t}\big)>0$,
	which allows us to conclude in the same way as above that $\lambda_2 \le \lambda_{N} <\infty$.
\end{proof}

\textbf{Acknowledgement.}
The authors would like to thank the anonymous referee for their detailed comments and suggestions which helped us to improve the article. MS was supported by the LOEWE programme of the state of Hessen (CMMS) in the course of this project.

%REFERENCES USING BIB
\bibliographystyle{abbrv}
\bibliography{references}

\end{document}